%% file: hard-core-paper.tex
\documentclass{article}

\input{files/macros}

\input{files/diagrams-ready}			

\begin{document}

\title{Metastability of hard-core dynamics on bipartite graphs}

\author{
Frank den Hollander%
\thanks{Mathematical Institute, Leiden University, Leiden, The Netherlands}
\and
Francesca~R. Nardi%
\thanks{Department of Mathematics, University of Florence, Italy}
\and
Siamak Taati%
\thanks{Department of Mathematics, University of British Columbia, Vancouver, Canada}
}

\date{\today}

\maketitle

\begin{abstract}
We study the metastable behaviour of a stochastic system of particles with hard-core 
interactions in a high-density regime. Particles sit on the vertices of a bipartite graph.
New particles appear subject to a neighbourhood exclusion constraint, while existing 
particles disappear, all according to independent Poisson clocks. We consider the 
regime in which the appearance rates are much larger than the disappearance rates,
and there is a slight imbalance between the appearance rates on the two parts of the 
graph. Starting from the configuration in which the weak part is covered with particles,
the system takes a long time before it reaches the configuration in which the strong 
part is covered with particles. We obtain a sharp asymptotic estimate for the expected 
transition time, show that the transition time is asymptotically exponentially distributed, 
and identify the size and shape of the critical droplet representing the bottleneck for the 
crossover. For various types of bipartite graphs the computations are made explicit. 
Proofs rely on potential theory for reversible Markov chains, and on isoperimetric 
results. In a follow-up paper we will use our results to study the performance of 
random-access wireless networks.

\vspace{1cm}
\medskip\noindent
\emph{Keywords:} Interacting particle systems, bipartite graphs, potential theory, 
metastability, isoperimetric problems.\\
\emph{MSC2010:} 60C05; 60K35; 60K37; 82C27.\\
\emph{Acknowledgment:} The research in this paper was supported through ERC 
Advanced Grant 267356-VARIS and NWO Gravitation Grant 024.002.003--NETWORKS.
The authors wish to thank A. van~Enter for a helpful comment.

\end{abstract}

\newpage

\renewcommand{\contentsname}{\vspace{-2em}}
{\footnotesize
\tableofcontents
}

\newpage


\section{Introduction and main results}
\label{sec:intro}


\subsection{Background}
\label{sec:back}

A \emph{metastable state} in a physical system is a quasi-equilibrium that persists
on a short time scale but relaxes to an equilibrium on a long time scale, called a
\emph{stable state}. Such behaviour often shows up when the system resides in
the vicinity of a configuration where its energy has a local minimum and is subjected
to a small noise: in the short run the noise is unlikely to have a significant impact on
the system, whereas in the long run the noise pulls the system away from the local
minimum and triggers a rapid transition towards a global minimum. When and how
this transition occurs depends on the depths of the energy valley around the metastable
state and the shape of the bottleneck separating the metastable state from the stable
state, called the set of \emph{critical droplets}.

Metastability for interacting particle systems on \emph{lattices} has been studied 
intensively in the past three decades. Representative papers --- dealing with Glauber, 
Kawasaki and parallel dynamics (= probabilistic cellular automata) at low temperature 
--- are \cite{CasGalOliVar84}, \cite{NevSch91}, \cite{KotOli94}, \cite{AroCer96}, 
\cite{HolOliSco00}, \cite{BovMan02}, \cite{CirNar03}, \cite{HolNarOliSco03}, 
\cite{GauOliSco05}, \cite{BovHolNar06}, \cite{CirNarSpi08},\cite{BelLan10}.
Various different approaches to metastability have been proposed, including:
\begin{enumerate}[label=(\Roman*)]
\item
The \emph{path-wise approach}, summarised in the monograph by Olivieri 
and Vares~\cite{OliVar04}, and further developed in \cite{ManNarOliSco04} ,
\cite{CirNar13}, \cite{CirNarSoh15}, \cite{FerManNarSco15}, \cite{NarZocBor15}, 
\cite{FerManNarScoSoh16}. 
\item
The \emph{potential-theoretic approach}, initiated in \cite{BovEckGayKle00}, 
\cite{BovEckGayKle01}, \cite{BovEckGayKle02} and summarised in the monograph 
by Bovier and den Hollander~\cite{BovHol15}. 
\end{enumerate}
Recently, there has been interest in metastability for interacting particle systems on 
\emph{graphs}, which is much more challenging because of lack of periodicity. See 
Dommers~\cite{Dpr}, Jovanovski~\cite{Jpr}, Dommers, den Hollander, Jovanovski 
and Nardi~\cite{DdHJNpr}, den Hollander and Jovanovski~\cite{dHJpr}, for examples. 
In these papers the focus is on Ising spins subject to a Glauber spin-flip dynamics. 
Particularly challenging are cases where the graph is \emph{random}, because the 
key quantities controlling the metastable crossover depend on the realisation of the 
graph.

In the present paper, we study the metastable behaviour of a stochastic system of
particles with \emph{hard-core interactions} in a \emph{high-density regime}. Particles
sit on the vertices of a \emph{bipartite graph}. New particles appear subject
to a neighbourhood exclusion constraint, while existing particles disappear, all
according to independent Poisson clocks. We consider the regime in which the
appearance rates are much larger than the disappearance rates, and there is a
\emph{slight imbalance} between the appearance rates on the two parts of the
graph. Starting from the configuration in which the weak part is covered with particles
(= \emph{metastable state}), the system takes a long time before it reaches the 
configuration in which the strong part is covered with particles (= \emph{stable state}).

We develop an approach for the hard-core model on general bipartite graphs that 
reduces the description of metastability to understanding the isoperimetric 
properties of the graph. The Widom-Rowlinson model on a given graph fits into 
our setting as the hard-core model on an associated bipartite graph we call the 
doubled graph. Exploiting the isoperimetric properties of the graph, we are able 
to obtain a sharp asymptotic estimate for the expected transition time, show 
that the transition time is asymptotically exponentially distributed, and identify 
the size and shape of the critical droplet. Interesting examples include the 
\emph{even torus}, the \emph{doubled torus}, the \emph{tree-like graphs} and the \emph{hypercube}. The 
isoperimetric problem we deal with is non-standard, but in some cases it can be 
reduced to certain standard edge/vertex isoperimetric problems. In the case of 
the even torus and the doubled torus, we derive complete information on the 
isoperimetric problem and hence obtain a complete description of metastability. 
In the case of the tree-like graphs and the hypercube our understanding of the 
isoperimetric problem is less complete, but we are still able to obtain some
relevant information on metastability. Proofs rely on potential theory for reversible 
Markov chains and on isoperimetric inequalities. In a follow-up paper we will use 
our results to study the performance of \emph{random-access wireless networks}
(see also~\cite{ZocBorLeeNar13}). This application is our main motivation.

Earlier work on the same model~\cite{NarZocBor15} focused on the case where 
the appearance rates are \emph{balanced}, and lead to results in the high-density 
regime for the transition time between the two \textit{stable} configurations in 
probability, in expected value and in distribution for finite lattices. The general framework 
in~\cite{NarZocBor15} was also exploited to derive results for the balanced hard-core 
model on non-bipartite graphs (e.g.\ the triangular lattice)~\cite{Zoc17a} and for 
the Widom-Rowlison model~\cite{Zoc17b}.

The remainder of the paper is organised as follows. In Section~\ref{sec:intro:model} we
define the model. In Section~\ref{sec:main} we state and discuss three metastability
theorems. In Section~\ref{sec:mc:reversible} we recall the main ingredients of potential
theory for reversible Markov chains, including the Nash-Williams inequalities for estimating
effective resistance. In Section~\ref{sec:mc:family} we develop a formulation of
metastability for a parametrized family of reversible Markov chains in an asymptotic
regime. In Sections~\ref{sec:hardcorebipartite}--\ref{sec:furtherprep} we apply the
framework of Sections~\ref{sec:mc:reversible}--\ref{sec:mc:family} to hard-core
dynamics. In Section~\ref{sec:main-theorems:proof} we give the proof of the three
metastability theorems of Section~\ref{sec:main}. Section~\ref{sec:main-examples:proof}
describes in more detail what is implied by these theorems in various concrete examples.
Section~\ref{sec:isoperimetric} is devoted to the study of certain isoperimetric problems
that arise in the identification of the critical droplet. Finally, Appendix~\ref{sec:proofspottheo}
provides the proofs of various claims made in Sections~\ref{sec:mc:reversible}--\ref{sec:isoperimetric}.
These are collected at the end in order to smoothen the presentation.


\subsection{Model}
\label{sec:intro:model}

We consider a system of particles living on a (finite, simple, undirected) connected 
graph $G=(V(G),E(G))$, where $V(G)$ is the set of \emph{vertices} and $E(G)$ is 
the set of \emph{edges} between them. We refer to vertices as \emph{sites}. Each 
site of the graph can carry 0 or 1 particle, but we impose the constraint that two 
adjacent sites cannot carry particles simultaneously. A (valid) \emph{configuration} 
of the model is thus an assignment $x\colon\,V(G)\to\{\symb{0},\symb{1}\}$ such that, 
for each pair of adjacent sites $i,j$, either $x_i=\symb{0}$ or $x_j=\symb{0}$. Alternatively, 
a valid configuration can be identified by an independent set of the graph, i.e., a subset 
$x\subseteq V(G)$ of sites having no edges between them. We will use these two 
representations interchangeably, and with some abuse of notation use the same 
symbol to denote the map $x\colon\,V(G)\to\{\symb{0},\symb{1}\}$ or the subset 
$x\subseteq V(G)$. The set of valid configurations is denoted by $\spX\subseteq
\{\symb{0},\symb{1}\}^{V(G)}$.

The configuration of the system evolves according to a continuous-time Markov chain.
Particles appear or disappear independently at each site, at fixed rates depending on 
the site and subject to the exclusion constraint. Namely, each site $k$ has two associated 
Poisson clocks $\xi^{\birth}_k$ and $\xi^{\death}_k$, signalling the (attempted) birth 
and death of particles:
\begin{description}[font=\it]
\item[Birth:] 
Clock $\xi^{\birth}_k$ has rate $\lambda_k>0$. Every time $\xi^{\birth}_k$ ticks, an 
attempt is made to place a particle at site~$k$. If one of the neighbours of site $k$ 
carries a particle, or if there is already a particle at~$k$, then the attempt fails.
\item[Death:] 
Clock $\xi^{\death}_k$ has rate $1$. Every time $\xi^{\death}_k$ ticks, an attempt 
is made to remove a particle from site $k$. If the site is already empty, then nothing 
is changed.
\end{description}
All the clocks are assumed to be independent.

The parameter $\lambda_k$ is called the \emph{activity} or \emph{fugacity} at site $k$.
We are interested in the asymptotic regime where $\lambda_k\gg 1$. It is easy to verify 
that the distribution
\begin{align}
\pi(x) &\isdef \frac{1}{Z}\prod_{k\in x}\lambda_k,
\end{align}
(where $Z$ is the appropriate normalising constant) is the unique (reversible) equilibrium 
distribution for this Markov chain. Note that when $\lambda_k\gg 1$, the distribution $\pi$ 
is mostly concentrated at configurations that are close to maximal packing.

We prefer to develop our theory in the discrete-time setting. Therefore, we simulate the 
above continuous-time Markov chain by means of a single Poisson clock $\xi$ with rate 
$\gamma\isdef\sum_{k\in V(G)}(\lambda_k+1)$ and a discrete-time Markov chain 
(independent of the clock) in the standard fashion. In this case, the discrete-time Markov 
chain becomes a Gibbs sampler for the distribution $\pi$: a transition of the discrete-time 
chain is made by first picking a random site $I$ with distribution $(i\mapsto \frac{1+\lambda_i}
{\gamma})$, and afterwards resampling the state of site $I$ according to $\pi$ conditioned 
on the rest of the current configuration, i.e., according to $(\symb{0}\mapsto\frac{1}{1+\lambda_I}, 
\symb{1}\mapsto \frac{\lambda_I}{1+\lambda_I})$ if the current configuration has no particle 
in the neighbourhood of $I$, and $(\symb{0}\mapsto 1, \symb{1}\mapsto 0)$ otherwise.
More explicitly, the transition probability from a configuration $x$ to a configuration $y\neq x$ 
(both in $\spX$) is given by
\begin{align}
K(x,y) &= \begin{cases}
\lambda_i/\gamma
& \text{if $x_i=\symb{0}$, $y_i=\symb{1}$, and $x_{V(G)\setminus\{i\}}=y_{V(G)\setminus\{i\}}$,}\\
1/\gamma
& \text{if $x_i=\symb{1}$, $y_i=\symb{0}$, and $x_{V(G)\setminus\{i\}}=y_{V(G)\setminus\{i\}}$,}\\
0
& \text{otherwise.}
\end{cases}
\end{align}
The probability $K(x,x)$ is simply chosen so as to make $K$ a stochastic matrix.

In summary, the discrete-time chain $(X(n))_{n\in\NN}$ (where $\NN\isdef\{0,1,2,\ldots\}$)
and the continuous-time chain $(\hat{X}(t))_{t\in[0,\infty)}$ are connected via the coupling 
$\hat{X}(t)\isdef X(\xi([0,t]))$, where $\xi$ is a Poisson process with rate $\gamma$ 
independent of $(X(n))_{n\in\NN}$. If $T$ is a stopping time for the discrete-time chain
and $\hat{T}$ is the corresponding stopping time for the continuous-time chain, then we have 
the relation $\xExp[T]=\gamma\xExp[\hat{T}]$.

The above process is the dynamic version of the \emph{hard-core} gas model. Throughout this 
paper, we assume that the underlying graph is \emph{bipartite}, i.e., the sites of the graph can 
be partitioned into two disjoint sets $U$ and $V$ in such a way that every edge of the graph 
has one endpoint in $U$ and the other endpoint in $V$. In the sequel, we will assume that 
$\lambda_k=\lambda$ for all $k\in U$ and $\lambda_k=\bar{\lambda}$ for all $k\in V$, where 
$\lambda,\bar{\lambda}\in\RR^+$. A simple example of a bipartite graph on which the hard-core 
dynamics exhibits very strong metastable behaviour is the \emph{complete} bipartite graph
(Fig.~\ref{fig:graph:complete-bipartite}) in which every site in $U$ is connected by an edge
to every site in $V$: starting from the configuration $u$ with particles at every site in $U$,
the system must first remove every single particle from $U$ in order to be able to place a 
particle on $V$ and eventually reach the configuration $v$ with particles at every site in $V$.
A more interesting example is an \emph{even torus} graph $\ZZ_{m}\times\ZZ_{n}$ ($m$ and 
$n$ even) with nearest-neighbour edges, in which case $U$ and $V$ can be chosen, 
respectively, to be the sets of sites with even and odd coordinates (Fig.~\ref{fig:graph:torus}).
A further class of interesting examples arises from the two-species Widom-Rowlinson model,
which has an equivalent representation in our setting.

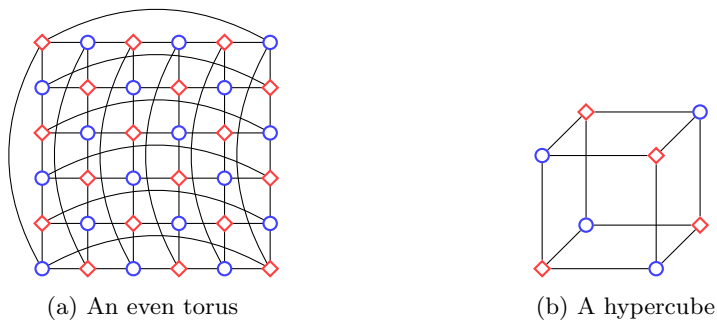
\begin{figure}[htbp]
\centering
\begin{subfigure}[b]{0.4\textwidth}
\centering
{
\tikzsetfigurename{graph_torus}
\begin{tikzpicture}[scale=0.6,baseline]
\torus
\end{tikzpicture}
}
\caption{An even torus}
\label{fig:graph:torus}
\end{subfigure}
\begin{subfigure}[b]{0.4\textwidth}
\centering
{	
\tikzsetfigurename{graph_hypercube}
\begin{tikzpicture}[scale=1.5,baseline]
\hypercube
\end{tikzpicture}
}
\caption{A hypercube}
\label{fig:graph:hypercube}
\end{subfigure}
\caption{More examples of bipartite graphs.}
\label{fig:graph:bipartite:main}
\end{figure}

The (dynamic) \emph{Widom-Rowlinson} model (see e.g.\ Lebowitz and Gallavotti~\cite{LebGal71})
is similar. In this model there are two types of particles, red and blue. Again, each site of the 
graph can be occupied by at most one particle, which can be of either type, but the exclusion 
constraint acts between opposite types only: two particles of opposite colour cannot simultaneously 
sit on two neighbouring sites. The dynamics is governed by three families of independent Poisson 
clocks:
\begin{description}[font=\it]
\item[Birth of red:] 
Clock $\xi^{\red\birth}_k$ has rate $\lambda_\red>0$. Every time $\xi^{\red\birth}_k$ ticks, an 
attempt is made to place a red particle at site $k$. If one of the neighbours of site $k$ carries a 
blue particle, or if there is already a particle on~$k$, then the attempt fails.
\item[Birth of blue:] 
Clock $\xi^{\blue\birth}_k$ has rate $\lambda_\blue>0$. Every time $\xi^{\blue\birth}_k$ ticks, 
an attempt is made to place a blue particle at site $k$. If one of the neighbours of site $k$ 
carries a red particle, or if there is already a particle on $k$, then the attempt fails.
\item[Death:] 
Clock $\xi^{\death}_k$ has rate $1$. Every time $\xi^{\death}_k$ ticks, an attempt is made to 
remove a particle from site $k$. If the site is already empty, then nothing is changed.
\end{description}

The Widom-Rowlinson model on a graph $G=(V(G),E(G))$ has a faithful representation in terms 
of the hard-core process on a bipartite graph $G^{[2]}$ obtained from $G$, which we call the 
\emph{doubled} version of $G$ (see Fig.~\ref{fig:graph:doubled-graph}). The graph $G^{[2]}$ has 
vertex set $V(G^{[2]}) \isdef V(G)\times\{\red,\blue\}$ with two parts $U^{[2]}\isdef\{(k,\red)\colon\,
k\in V(G)\}$ and $V^{[2]}\isdef\{(k,\blue)\colon\, k\in V(G)\}$, which are the coloured copies of 
$V(G)$. There is an edge between a red site $(i,\red)$ and a blue site $(j,\blue)$ if and only if 
either $i=j$ or $(i,j)$ is an edge in $E(G)$ (Fig.~\ref{fig:graph:doubled-graph}). There are 
no edges between red sites nor between blue sites. The configurations of the Widom-Rowlinson 
model on $G$ are in obvious one-to-one correspondence with the configurations of the 
hard-core model on $G^{[2]}$. Namely, a configuration $x$ of the Widom-Rowlinson model 
corresponds to a configuration $x^{[2]}$ of the hard-core model on the doubled graph where 
$x_i=\red$ if and only if $x_{(i,\red)}=\symb{1}$ and $x_i=\blue$ if and only if $x_{(i,\blue)}
=\symb{1}$. Furthermore, this correspondence is respected by the stochastic dynamics
of the two models. So in short, studying the Widom-Rowlinson model on $G$ amounts to
studying the hard-core model on the doubled graph $G^{[2]}$.

\begin{figure}[htbp]
\centering
\begin{subfigure}[b]{0.3\textwidth}
\centering
{
\tikzsetfigurename{graph_sample}
\begin{tikzpicture}[scale=1.1,baseline]
\samplegraph
\end{tikzpicture}
}
\caption{A graph $G$}
\label{fig:graph:example}
\end{subfigure}
\begin{subfigure}[b]{0.3\textwidth}
\centering
{
\tikzsetfigurename{graph_sample_doubled}
\begin{tikzpicture}[scale=0.9,baseline]
\samplegraphdoubled
\end{tikzpicture}
}
\caption{The doubled graph $G^{[2]}$}
\label{fig:graph:example:doubled}
\end{subfigure}
\begin{subfigure}[b]{0.3\textwidth}
\centering
{
\tikzsetfigurename{graph_sample_doubled_unwound}
\begin{tikzpicture}[scale=0.9,baseline]
\samplegraphdoubledunwound
\end{tikzpicture}
}
\caption{A different drawing of $G^{[2]}$}
\label{fig:graph:example:doubled:unwound}
\end{subfigure}
\caption{A graph and its doubled version.}
\label{fig:graph:doubled-graph}
\end{figure}
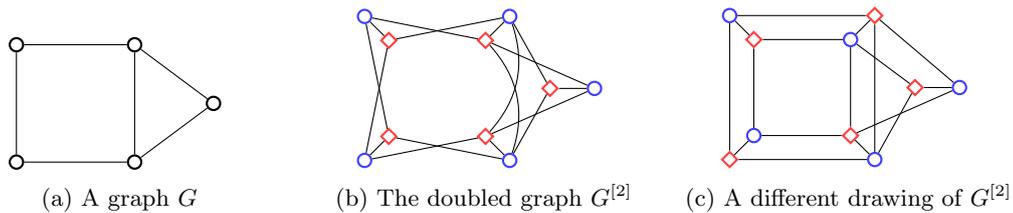

\subsection{Three metastability theorems}
\label{sec:main}

For the hard-core model on a bipartite graph $(U,V,E)$, we write $u$ for the configuration 
that has a particle at every site of $U$, and $v$ for the configuration that has a particle
at every site of $V$. For the activity parameters, we choose $\lambda_k=\lambda$ for 
$k\in U$ and $\lambda_k=\bar{\lambda}$ for $k\in V$, and we assume that
\begin{align}
\bar{\lambda}&=\varphi(\lambda)=\lambda^{1+\alpha+\smallo(1)}
\qquad\text{as $\lambda\to\infty$,}
\end{align}
for some constant $0<\alpha<1$. In other words, the activities of the sites in $V$ are slightly 
stronger than the sites in $U$. The symmetric scenario in which $\alpha=0$ is treated 
in Nardi, Zocca and Borst~\cite{NarZocBor15}.
In this paper, we focus on the case in which $\abs{U}<(1+\alpha)\abs{V}$ .  This ensures that $v$ has the largest
stationary probability among all configurations.
The opposite case can be treated similarly.

When $\lambda\to\infty$, we expect noticeable metastability when starting from $u$.
Namely, although the configuration $v$ takes up the overwhelmingly largest portion of the equilibrium 
probability mass, the process starting from $u$ remains in the vicinity of $u$ for a long time 
before the formation of a `critical droplet' and the eventual transition to $v$. The choice 
$\lambda^{1+\alpha+\smallo(1)}$ for $\varphi(\lambda)$ ensures that the size of the critical 
droplet is non-trivial (neither going to $0$ nor to $\infty$ as $\lambda\to\infty$). With this choice, 
we may think of
\begin{align}
H(x) &\isdef -\abs{x_U} - (1+\alpha)\abs{x_V}
\end{align}
(where $x_U\isdef x\cap U$ and $x_V\isdef x\cap V$)
as an appropriate notion of \emph{energy} or \emph{height} of configuration $x$, although we 
should keep in mind that the probability $\pi(x)$ and the height $H(x)$ are related only through
the asymptotic equality $\pi(x)=\frac{1}{Z}\lambda^{-H(x)+\smallo(1)}$. (In particular, note that 
the factor $\lambda^{\smallo(1)}$ is allowed to go to $\infty$ as $\lambda\to\infty$.) This 
interpretation provides the connection with the usual setting of metastability on which the 
current paper is based. As it turns out, the factor $\lambda^{\smallo(1)}$ does not alter the 
size or shape of the critical droplet, and only affects the transition time (see also Cirillo,
Nardi and Sohier~\cite{CirNarSoh15}).

On a typical transition path from $u$ to $v$, the configurations near the bottleneck (i.e., those 
representing the critical droplet) solve a (non-standard) \emph{isoperimetric problem} on the 
underlying bipartite graph. The \emph{isoperimetric cost} of a set $A\subseteq V$ is defined 
as $\Delta(A)\isdef\abs{N(A)}-\abs{A}$. The smallest possible isoperimetric cost for a set of 
cardinality $s$ is denoted by $\Delta(s)$. A set that achieves this minimum is said to be 
\emph{isoperimetrically optimal}. The isoperimetric problem associated with the graph $(U,V,E)$
asks for the optimal values $\Delta(s)$ and the optimal sets. An \emph{isoperimetric numbering} 
is a sequence $a_1,a_2,\ldots,a_n$ of distinct elements in $V$ such that for each $1\leq i\leq n$, 
the set $A_i\isdef\{a_1,a_2,\ldots,a_i\}$ is isoperimetrically optimal.

Our main results concern the hard-core model on a bipartite graph with the above choices of 
the relevant parameters, and rely on fairly general (though not necessarily easily verifiable)
\emph{hypotheses} regarding the isoperimetric properties of the underlying graph. These 
hypotheses are not the most general possible and can certainly be relaxed. Our goal is to show 
how they can be put to use in a few concrete examples: the torus $\ZZ_m\times\ZZ_n$ (where $m$ 
and $n$ are sufficiently large even numbers), the hypercube $\ZZ_2^m$, tree-like graphs and the 
doubled versions of these (see Fig.~\ref{fig:graph:bipartite:main}--\ref{fig:graph:doubled-graph}). 
In the case of the torus, where we have a rather complete understanding of the isoperimetric 
properties (via reduction to standard isoperimetric problems), we verify that all the required 
hypotheses are indeed satisfied. For the other examples, we are able to verify only some of 
the hypotheses, thereby obtaining only partial results. Complete descriptions remain contingent 
upon a better understanding of the corresponding isoperimetric problems.

Our first two theorems establish asymptotics for the mean and the distribution of the crossover 
time (i.e., the hitting time of $v$ starting from $u$). Let $s^*$ be the smallest positive integer 
maximising $g(s)\isdef\Delta(s)-\alpha(s-1)$. We call $s^*$ the \emph{critical size}. Let $\tilde{s}$ 
be the smallest integer larger than $s^*$ such that $\Delta(\tilde{s})\leq\alpha\tilde{s}$. We call 
$\tilde{s}$ the \emph{resettling size}. The required hypotheses for these two theorems are the 
following:
\begin{enumerate}[label=\fbox{H\arabic*},ref={\rm H\arabic*},series=hypo]
	\setcounter{enumi}{-1}
	\item
		\label{hypothesis:v-is-stable}
		$\abs{U}<(1+\alpha)\abs{V}$.
	\item 
		\label{hypothesis:numbering:single}
		There exists an isoperimetric numbering of length at least $\tilde{s}$.
	\item
		\label{hypothesis:numbering:all}
		For every $a\in V$, there exists an isoperimetric numbering of length at least $\tilde{s}$
		starting with $a$.
\end{enumerate}
Clearly \eqref{hypothesis:numbering:all} implies~\eqref{hypothesis:numbering:single}.
In fact, the following theorems require the stronger hypothesis~\eqref{hypothesis:numbering:all}
but we have stated~\eqref{hypothesis:numbering:single} for future reference.
The existence of the resettling size is ensured by hypothesis~\eqref{hypothesis:v-is-stable}.

Let $\hat{T}_v\isdef\{t\geq 0: X(t)=v\}$ be the first hitting time of configuration $v$.

\begin{theorem}[{\bf Mean crossover time: order of magnitude}]
\label{thm:mean-crossover:magnitude}
Suppose that conditions~\eqref{hypothesis:v-is-stable} and~\eqref{hypothesis:numbering:all} are satisfied.
Then
\begin{align}
	\xExp_u[\hat{T}_v] &\asymp
		\frac{\lambda^{\Delta(s^*)+s^*-1}}{\bar{\lambda}^{s^*-1}}
		= \lambda^{\Delta(s^*)-\alpha(s^*-1)+\smallo(1)}
	\qquad\text{as $\lambda\to\infty$,}
\end{align}
where $f(\lambda)\asymp g(\lambda)$ means that $f=\bigo(g)$ and $g=\bigo(f)$
as $\lambda\to\infty$.
\end{theorem}

\begin{theorem}[{\bf Exponential law for crossover time}]
\label{thm:crossover:exponential}
Suppose that conditions~\eqref{hypothesis:v-is-stable} and~\eqref{hypothesis:numbering:all} are satisfied.
Then
\begin{align}
\lim_{\lambda\to\infty}\xPr_u\left(
\frac{\hat{T}_v}{\xExp_u[\hat{T}_v]} > t \right) 
&= \ee^{-t}
\qquad\text{uniformly in $t\in\RR^+$.}
\end{align}
\end{theorem}

For the next theorem, we need a few extra definitions and hypotheses. Note that 
Theorem~\ref{thm:mean-crossover:magnitude} provides only the order of magnitude 
of the mean crossover time $\xExp_u[\hat{T}_v]$ as $\lambda\to\infty$. A more accurate asymptotics 
(the pre-factor) requires a more detailed description of the bottleneck (the critical droplets), which in 
turn requires a better understanding of the isoperimetric properties of the underlying graph. More 
specifically, we need an understanding of the evolution of the set of occupied sites in $V$ during 
the crossover from $u$ to $v$. We call a sequence of sets $A_0,A_1,\ldots,A_n\subseteq V$ a 
\emph{progression} from $A_0$ to $A_n$ if $\abs{A_i\triangle A_{i+1}}=1$ for each $0\leq i<n$.
A progression $A_0,A_1,\ldots,A_n$ is \emph{isoperimetric} if $A_i$ is isoperimetrically optimal 
for each $0\leq i\leq n$. An \emph{$\alpha$-bounded} progression is a progression $A_0,A_1,\ldots,A_n$
such that $\Delta(A_i)-\alpha\abs{A_i}\leq \Delta(s^*)-\alpha s^*$ for each $0\leq i\leq n$.

For our third theorem we need two more hypotheses:

\begin{enumerate}[resume*=hypo]
	\item 
		\label{hypothesis:uniqueness}
		The critical size $s^*$ is the unique maximiser of $g(s)\isdef\Delta(s)-\alpha(s-1)$
		in $\{0,1,\ldots,\tilde{s}\}$.
	\item
		\label{hypothesis:critical-set}
		There exist two families $\family{A},\family{B}$ of subsets of $V$ such that
		\begin{enumerate}[label=(\alph*),ref=\theenumi.\alph*]
			\item \label{hypothesis:critical-set:size}
				the elements of $\family{A}$ and $\family{B}$
				are isoperimetrically optimal with
				$\abs{A}=s^*-1$ for each $A\in\family{A}$ and
				$\abs{B}=s^*$ for each $B\in\family{B}$,
			\item \label{hypothesis:critical-set:before-A}
				for each $A\in\family{A}$, there is an isoperimetric progression
				from $\varnothing$ to $A$, consisting only of sets of size at most 
				$s^*-1$.
			\item \label{hypothesis:critical-set:after-B}
				for each $B\in\family{B}$, there is an isoperimetric progression
				from $B$ to a set of size $\tilde{s}$, consisting only of sets
				of size at least $s^*$,
			\item \label{hypothesis:critical-set:mandatory}
				for every $\alpha$-bounded progression $A_0,A_1,\ldots,A_n$
				with $A_0=\varnothing$ and $\Delta(A_n)\leq\alpha\abs{A_n}$, 
				there is an index 
				$0\leq k<n$ such that $A_k\in\family{A}$ and $A_{k+1}\in\family{B}$.
		\end{enumerate}
\end{enumerate}
We interpret an element of $\family{B}$ as a \emph{critical droplet} on $V$. Given two families 
$\family{A}$ and $\family{B}$ satisfying~\eqref{hypothesis:critical-set}, we define two sets of 
configurations $Q$ and $Q^*$ as follows. The set $Q^*$ consists of configurations $y$ such 
that $y_V=A$ and $y_U=U\setminus N(B)$ for some $A\in\family{A}$ and $B\in\family{B}$
with $\abs{B\setminus A}=1$. A configuration $x$ is in $Q$ if it can be obtained from a configuration 
$y\in Q^*$ by adding a particle on $U$. We denote by $[Q,Q^*]$ the set of possible transitions 
$x\to y$ where $x\in Q$ and $y\in Q^*$. In other words, $[Q,Q^*]$ consists of pairs $(x,y)\in Q
\times Q^*$ such that $x$ and $y$ differ by a single particle. The set $[Q,Q^*]$ is an example 
of what we call a \emph{critical gate}. Observe that
\begin{align}
\label{eq:critical-gate:size}
	\abs{[Q,Q^*]} &\isdef
	\mathop{\sum_{A\in\family{A}}\sum_{B\in\family{B}}}_{\abs{B\setminus A}=1}
	\abs{N(B)\setminus N(A)}. 
\end{align}

\begin{theorem}[{\bf Critical gate}]
\label{thm:crossover:sharp}
Suppose that conditions~\eqref{hypothesis:v-is-stable}, \eqref{hypothesis:numbering:all} and~\eqref{hypothesis:uniqueness} are 
satisfied. Suppose further that there are two families $\family{A}$ and $\family{B}$ of subsets 
of $V$ satisfying~\eqref{hypothesis:critical-set}. Let $[Q,Q^*]$ be the above-mentioned set of 
transitions associated to $\family{A}$ and $\family{B}$. Then
\begin{enumerate}[label={\rm (\roman*)}]
\item 
\emph{({\bf Mean crossover time: sharp asymptotics})}
\begin{align}
\xExp_u[\hat{T}_v] 
&= \frac{1}{\abs{[Q,Q^*]}} \frac{\lambda^{\Delta(s^*)+s^*-1}}{\bar{\lambda}^{s^*-1}}
[1+\smallo(1)] \qquad\text{as $\lambda\to\infty$.}
\end{align}
\item 
\emph{({\bf Passage through the gate})}\\
With probability approaching $1$ as $\lambda\to\infty$, the random trajectory from $u$ to $v$ 
makes precisely one transition $x\to y$ from $[Q,Q^*]$, every configuration that follows the 
transition $x\to y$ has at least~$s^*$ particles on~$V$, and every configuration preceding 
$x\to y$ has at most~$s^*-1$ particles on~$V$. Moreover, the choice of the transition $x\to y$ 
is uniform among all possibilities in $[Q,Q^*]$.
\end{enumerate}
\end{theorem}

Verifying condition~\eqref{hypothesis:critical-set} in concrete examples can be quite difficult.
However, sacrificing full generality, it is possible to give a rather explicit construction
of families $\family{A}$ and $\family{B}$ and replace~\eqref{hypothesis:critical-set}
with two other hypotheses that are more restrictive but much easier to verify.

Let $0\leq \kappa<\nicefrac{1}{\alpha}$ be an integer
(e.g., $\kappa\isdef\lceil\nicefrac{1}{\alpha}\rceil-1$)
and define
\begin{align}
	\family{A} &\isdef \{A\subseteq V: \text{$A$ is isoperimetrically optimal with $\abs{A}=s^*-1$} \} \;, \\
	\family{C} &\isdef \{C\subseteq V: \text{$C$ is isoperimetrically optimal with $\abs{A}=s^*+\kappa$} \} \;, \\
	\family{B} &\isdef
		\left\{B\subseteq V:%
			\parbox{0.55\textwidth}{%
				\centering
				there exists an isoperimetric progression $B_0,B_1,\ldots,B_n$\\
				with $B_0\in\family{A}$, $B_n\in\family{C}$ and $B_1=B$\\
				such that $s^*-1<\abs{B_i}< s^*+\kappa$ for $0<i<n$
			}
		\right\}
\end{align}
Observe that $\abs{B}=s^*$ for every $B\in\family{B}$.
Consider the following hypotheses:
\begin{enumerate}[resume*=hypo]
		\item
		\label{hypothesis:critical-set-replacement:values}
		\begin{enumerate}[label=(\alph*),ref=\theenumi.\alph*]
			\item
				\label{hypothesis:critical-set-replacement:values:A}
				$\Delta(s^*+\kappa)\geq\Delta(s^*+\kappa-1)$,
			\item
				\label{hypothesis:critical-set-replacement:values:B}
				$\Delta(s^*+i)\geq\Delta(s^*)$ for $0\leq i<\kappa$,
			\item
				\label{hypothesis:critical-set-replacement:values:C}
				$\Delta(s^*)=\Delta(s^*-1)+1$.
		\end{enumerate}
	\item
		\label{hypothesis:critical-set-replacement:existence}
		\begin{enumerate}[label=(\alph*),ref=\theenumi.\alph*]
			\item
				\label{hypothesis:critical-set-replacement:before-A}
				For each $A\in\family{A}$,
				there is an isoperimetric progression from $\varnothing$ to $A$,
				consisting only of sets of size at most $s^*-1$.
			\item
				\label{hypothesis:critical-set-replacement:after-C}
				For each $C\in\family{C}$,
				there is an isoperimetric progression from $C$ to a set of size $\tilde{s}$,
				consisting only of sets of size at least $s^*$.
		\end{enumerate}
\end{enumerate}

\begin{proposition}[{\bf Identification of critical gate}]
\label{prop:critical-gate:identification}
Suppose that conditions~\eqref{hypothesis:v-is-stable}, \eqref{hypothesis:numbering:single},
\eqref{hypothesis:uniqueness},
\eqref{hypothesis:critical-set-replacement:values}
and~\eqref{hypothesis:critical-set-replacement:existence} are satisfied.
Then the families $\family{A}$ and $\family{B}$ described above
satisfy condition~\eqref{hypothesis:critical-set}.
\end{proposition}

Theorems~\ref{thm:mean-crossover:magnitude}--\ref{thm:crossover:sharp} are 
proved in Section~\ref{sec:main-theorems:proof} after the necessary preparations. 
Section~\ref{sec:mc:reversible} recalls some basic facts from potential theory for 
reversible Markov chains. Section~\ref{sec:mc:family} provides a characterisation of 
metastability in terms of recurrence of metastable states and passage through 
bottlenecks. In Sections~\ref{sec:hardcorebipartite}--\ref{sec:furtherprep} and 
\ref{sec:main-examples:proof} we specialise to hard-core dynamics on bipartite 
graphs and look at both `simple examples' and `sophisticated examples', for which 
we identify $s^*$, $\Delta(s^*)$ and $[Q,Q^*]$. Section~\ref{sec:isoperimetric} is 
devoted to the isoperimetric problems associated with the `sophisticated examples' 
in Section~\ref{sec:hardcorebipartite}. Proposition~\ref{prop:critical-gate:identification} 
is proved in Appendix~\ref{apx:critical-gate:identification} using a detailed study
of typical paths near the critical droplet in Section~\ref{sec:hard-core:critical-gate}.
Appendix~\ref{sec:proofspottheo} collects the proofs of all the propositions and lemmas
appearing in Sections~\ref{sec:mc:reversible}--\ref{sec:isoperimetric}.


\section{Reversible Markov chains}
\label{sec:mc:reversible}

A useful tool for studying reversible Markov chains is their analogy with electric networks 
and potential theory. This analogy has been exploited in various contexts, most notably 
for the recurrence/transience problem. The use of potential theory in the study of metastability 
is pioneered by Bovier, Eckhoff, Gayrard and Klein~\cite{BovEckGayKle01} and is developed 
in detail in the monograph by Bovier and den Hollander~\cite{BovHol15}. We start by recalling 
the relevant aspects of the connection between electric networks and reversible Markov chains, 
while fixing our notation and terminology (see Section~\ref{sec:general:review}). Estimating 
the expected hitting time of a target set reduces via the above analogy to estimating the 
effective resistance between the starting point and the target as well as the voltage at different 
points of the network. Sharp estimates for effective resistance can be obtained using the machinery 
of the Nash-Williams inequalities (see Section~\ref{sec:general:eff-resistance:sharp-bounds})
or using the variational principles of Thomson and Dirichlet. A simpler estimate for effective 
resistance, capturing its order of magnitude, is given by ``critical resistance'', which is an 
abstract variant of the more standard notion of ``communication height'' often used in 
metastability theory (see Section~\ref{sec:general:eff-resistance:communication-height}).
Critical resistance can also be used to provide rough bounds for voltage (see 
Section~\ref{sec:general:voltage:rough-bounds}).


\subsection{Connection with electric networks}
\label{sec:general:review}

In this section we fix the general notation and terminology and recall a few relevant 
facts about reversible Markov chains and their analogy with electric networks. The 
proofs and the background could be found in various sources, e.g.\ Doyle and 
Snell~\cite{DoySne84}, Levin, Peres and Wilmer~\cite{LevPerWil08}, Grimmett~\cite{Gri10}, 
Lyons and Peres~\cite{LyoPer14}, Aldous and Fill~\cite{AldFil14}, Bovier and den 
Hollander~\cite{BovHol15}.

We let $(X(n))_{n\in\NN}$ be a discrete-time Markov chain with finite state 
space $\spX$ and transition matrix $K:\spX\times\spX\to[0,1]$. We assume that $K$ 
is \emph{irreducible} and has a \emph{reversible} stationary distribution $\pi$.
We write $\xPr_x$ and $\xExp_x$ to denote probability and expectation conditioned on 
the event $X(0)=x$. The first \emph{hitting time} of a set $A\subseteq\spX$ is denoted 
by
\begin{align}
	T_A &\isdef \inf\{n\geq 0: X(n)\in A\} \;.
\end{align}
When we disregard the case $X(0)\in A$, we write
\begin{align}
	T_A^+ &\isdef \inf\{n> 0: X(n)\in A\} \;.
\end{align}
The first \emph{passage time} through a transition $x\to y$ is likewise denoted by
\begin{align}
	T_{xy} &\isdef \inf\{n>0: \text{$X(n-1)=x$ and $X(n)=y$}\} \;.
\end{align}

An analogy is made between the above reversible Markov chain and an electric network 
with nodes labelled by the elements of $\spX$ in which node $x$ is connected to node 
$y$ by a resistor with \emph{conductance} $c(x,y)\isdef \pi(x)K(x,y)=\pi(y)K(y,x)$ (and 
\emph{resistance} $r(x,y)=1/c(x,y)\in(0,\infty]$).
We write $x\link y$ when $c(x,y)>0$.
The first basic connection between the two objects is that the function
\begin{align}
h(x) &\isdef \xPr_x(T_A<T_B)
\end{align}
is the unique harmonic function with boundary conditions $h|_A\equiv 1$ and $h|_B\equiv 0$.
Therefore $\xPr_x(T_A<T_B)$ coincides with the \emph{voltage} $W_{A,B}(x)$ at node 
$x$ if all the nodes in $B$ are connected to the ground and all the nodes in $A$ are 
connected to a unit voltage source.

The \emph{effective resistance} and \emph{effective conductance} between two sets 
$A,B\subseteq\spX$ will be denoted by $\effR{A}{B}$ and $\effC{A}{B}$, respectively.
An easy consequence of the above connection is the equality
\begin{align}
\xPr_a(T_B<T^+_a) = \frac{1}{\pi(a)\effR{a}{B}}
\end{align}
for every state $a\in\spX$ and set $B\subseteq\spX$ not containing $a$.

When $T$ is a stopping time, we denote by $G_T(a,x)$ the expected number of visits 
to state $x$ if the chain is started at state $a$ and stopped at $T$, i.e.,
\begin{align}
G_T(a,x) &\isdef \xExp_a\left[\text{$\#$ of visits to $x$ before $T$}\right].
\end{align}
In case $x=a$, time $0$ is also counted. The function $G_T$ is the \emph{Green function} 
associated with $T$. The second basic connection between a reversible Markov chain
and its corresponding electric network is an electric interpretation of the Green functions 
associated to hitting times. Namely, it can be shown, for a state $a\in\spX$ and a set 
$B\subseteq\spX$ not containing $a$, that the function $h(x)\isdef G_{T_B}(a,x)/\pi(x)$ 
is harmonic with boundary conditions $h|_a\equiv\effR{a}{B}$ and $h|_B\equiv 0$. Therefore 
$G_{T_B}(a,x)/\pi(x)$ agrees with the voltage at $x$ provided all the nodes in $B$ are 
connected to the ground and $a$ is connected to a unit current source. It follows that
\begin{align}
\label{eq:green:expression}
G_{T_B}(a,x) &= \effR{a}{B}\pi(x)W_{a,B}(x),
\end{align}
where $W_{a,B}(x)=\xPr_x(T_a < T_B)$ is the voltage at $x$ when $B$ is connected to 
the ground and $a$ is connected to a unit voltage source. As an immediate corollary, we 
get the useful equality
\begin{align}
\xExp_a[T_B] &= \effR{a}{B} \sum_x \pi(x) W_{a,B}(x),
\end{align}
for every state $a\in\spX$ and set $B\subseteq\spX$ not containing $a$.

If $t$ is a non-negative constant, then by reversibility we have the general identity
\begin{align}
\pi(x) G_t(x,y) &= \pi(y) G_t(y,x).
\end{align}
This identity remains valid for Green functions associated with hitting times:
  
\begin{align}
\label{eq:green:reciprocity}
\pi(x) G_{T_Z}(x,y) &= \pi(y) G_{T_Z}(y,x)
\end{align}
for every two states $x,y\in\spX$ and every set $Z\subseteq\spX$. A similar reciprocity 
law holds for hitting order probabilities:
\begin{align}
\label{eq:hitting-order:reciprocity}
\effR{x}{Z}\xPr_y(T_x<T_Z) &= \effR{y}{Z}\xPr_x(T_y<T_Z)
\end{align}
for every two states $x,y\in\spX$ and every set $Z\subseteq\spX$.

The notion of projection for electric networks is much more relaxed than the notion of 
projection for Markov chains. Namely, identifying two nodes with the same voltage
(i.e., making a short circuit between them) we do not affect the voltage at other nodes.
As a corollary, we have that the effective resistance $\effR{A}{B}$ between two disjoint 
sets $A,B\subseteq\spX$ remains unchanged when we contract $A$ into a single node 
$a$ and $B$ into a single node $b$. This simplify some arguments.


\subsection{Sharp bounds for effective resistance}
\label{sec:general:eff-resistance:sharp-bounds}

The variational principles of Thomson and Dirichlet are the most common tools to obtain 
upper and lower bounds for effective resistance. An alternative combinatorial approach 
due to Nash-Williams often gives simple and useful estimates.

We consider a graph on the state set $\spX$ whose edges are the pairs $(x,y)$ with 
$c(x,y)>0$. Let $A,B\subseteq\spX$ be disjoint. A \emph{cut} separating $A$ from $B$ 
is a set $C\subseteq\spX$ such that $A\subseteq C\subseteq B^\complement$. Given 
a cut $C$, we write $\partial C\isdef\{(x,y): \text{$x\in C$, $y\notin C$ and $c(x,y)>0$}\}$
for the set of edges between $C$ and $C^\complement$. The simplest form of the 
Nash-Williams inequality is the intuitive inequality
\begin{align}
\label{eq:nash-williams:simplified}
\effC{A}{B} &\leq \effC{C}{C^\complement} 
\leq \abs{\partial C} \;\sup_{\mathclap{x\in C, y\notin C}}\; c(x,y)
\end{align}
for every cut $C$ separating $A$ from $B$. A dual (and equally intuitive) inequality
\begin{align}
\label{eq:nash-williams:dual:simplified}
\effR{A}{B} &\leq r(\omega)
\leq \abs{\omega} \sup_{e\in\omega} r(e)
\end{align}
holds for every path $\omega$ from $A$ to $B$. These two inequalities are special 
cases of the more general Nash-Williams inequalities, but can also be derived from 
the Dirichlet and the Thomson variational principles.

While the above upper bound for effective conductance is sufficient for our purpose, 
we need a more accurate lower bound. The following extended version of the 
(dual) Nash-Williams inequality due to Berman and Konsowa~\cite{BerKon90} 
provides a method to obtain sharp lower bounds.

\begin{proposition}[{\bf Extended dual Nash-Williams inequality}]
\label{prop:nash-williams:dual:extended}
Let $A,B\subseteq\spX$. Let $(\omega_i)_{i \in\NN}$ be an arbitrary sequence 
of simple paths from $A$ to $B$, with the property that no two paths $\omega_i$ and 
$\omega_j$ pass through a common edge in opposite directions. For each edge $e$, 
let $n(e)$ denote the number of paths $\omega_k$ that pass through $e$. Then
\begin{align}
\effC{A}{B} &\geq \sum_k \frac{1}{\sum_{e\in\omega_k} n(e)r(e)}.
\end{align}
\end{proposition}

The proof is similar to the proof of the standard Nash-Williams inequality, but 
for completeness, we include it in Appendix~\ref{apx:nash-williams}. We note
that the latter inequality is sharp: by allowing repetitions in the sequence 
$(\omega_i)_{i\in\NN}$ we get arbitrarily close lower 
bounds for the conductance $\effC{A}{B}$.


\subsection{Rough estimates for effective resistance}
\label{sec:general:eff-resistance:communication-height}

The order of magnitude of effective resistance is captured by the notion of ``critical resistance'',
which is much easier to evaluate. We define the \emph{critical resistance} between two sets 
$A,B\subseteq\spX$ as
\begin{align}
\Psi(A,B) &\isdef \inf_{\omega:A \pathto B} \sup_{e\in\omega}\; r(e),
\end{align}
where the infimum is taken over all paths (sequences of distinct states) connecting $A$ to $B$, 
and the supremum is over all edges (pairs of consecutive states) on the path. For a path $\omega$, 
we refer to $\Psi(\omega)\isdef \sup_{e\in\omega}\; r(e)$ as the critical resistance of $\omega$.

Critical resistance is closely related to the notion of \emph{communication height}, which is often 
used in the study of metastability in Metropolis dynamics (see Olivieri and Vares~\cite{OliVar04},
Bovier and den Hollander~\cite{BovHol15}). The two notions are connected via the (imprecise) 
correspondence $\Psi(A,B)\approx\ee^{\beta\Phi(A,B)}$, where $\Phi(A,B)$ is the communication 
height between $A$ and $B$ and $\beta$ is the inverse temperature. While somewhat less intuitive, 
the notion of critical resistance has two advantages. First, it is defined for individual Markov chains
(rather than parametric families of Markov chains), and therefore can also be used in asymptotic 
regimes other than $\beta\to\infty$, in particular, when there is no clear-cut notion of energy.
Second, while the height of a path $\omega$ is often defined as the maximum energy of a 
\emph{state} on $\omega$, the maximisation in the critical resistance is taken over \emph{pairs 
of consecutive states} on~$\omega$. As noted in Cirillo, Nardi and Sohier~\cite{CirNarSoh15}, 
this turns out to be the appropriate definition for general (non-Metropolis) Markov chains.

The effective resistance $a,b\mapsto\effR{a}{b}$ defines a metric on $\spX$. The critical resistance, 
on the other hand, defines an ultra-metric on $\spX$:
\begin{itemize}
\item 
$\Psi(x,y)\geq 0$ with equality if and only if $x=y$,
\item 
(\emph{symmetry}) $\Psi(x,y)=\Psi(y,x)$,
\item (\emph{strong triangle inequality}) 
$\Psi(x,z)\leq \max\left\{\Psi(x,y),\Psi(y,z)\right\}$.
\end{itemize}
The following proposition shows that the two metrics $a,b\mapsto\effR{a}{b}$ and $a,b\mapsto\Psi(a,b)$ 
are equivalent up to constants depending only on the graph (and not on the resistances $r$). Its proof 
can be found in Appendix~\ref{apx:critical-resistance}.

\begin{proposition}[{\bf Equivalence of metrics}]
\label{prop:bounds:communication-height:abstract}
There exist a constant $k\geq 1$ such that, for every two sets $A,B\subseteq\spX$,
\begin{align}
\label{eq:eff-resist:equiv-metric}
\frac{1}{k}\Psi(A,B) &\leq \effR{A}{B} \leq k\, \Psi(A,B).
\end{align}
The constant $k$ can be chosen to be $\abs{\spX}^2$.
\end{proposition}

To understand the geometry of $\Psi$, let us recall two basic facts. First, every triangle 
in a general ultra-metric space is isosceles, with two equal sides and a third side that is 
no larger than the other two (i.e., the three sides can be ordered as $a\leq b=c$).
Second, suppose that $\spT$ is a minimal spanning tree on $\spX$ (where edge $e$ is 
weighted by its resistance $r(e)$). Then, the $\Psi$-distance between two points $a,b\in\spX$
is simply the maximal resistance of the unique path between $a$ and $b$ on $\spT$.
In other words, every path on $\spT$ is geodesic with respect to $\Psi$.


\subsection{Rough estimates for voltage}
\label{sec:general:voltage:rough-bounds}

In order to estimate the Green function via~\eqref{eq:green:expression}, we will also need 
rough estimates for the voltage. The following proposition corresponds to Bovier and
den Hollander \cite[Lemma 7.13(iii)]{BovHol15}. Its proof can be found in 
Appendix~\ref{apx:voltage:estimate}.

\begin{proposition}[{\bf A priori estimate}]
\label{prop:voltage:bound:effective-resistance}
Let $A,B\subseteq\spX$ be two disjoint sets. For every node $x\in\spX\setminus (A\cup B)$,
\begin{align}
1 - \frac{\effR{x}{A}}{\effR{A}{B}} 
&\leq W_{A,B}(x) \leq \frac{\effR{x}{B}}{\effR{A}{B}},
\end{align}
where $W_{A,B}(x)=\xPr_x(T_a < T_B)$ is the voltage at $x$ when $B$ is connected to the 
ground and $A$ is connected to a unit voltage source.
\end{proposition}

Using the inequalities between effective resistance and critical resistance 
(Proposition~\ref{prop:bounds:communication-height:abstract}), we obtain 
the following proposition as a corollary of the above two estimates.

\begin{proposition}[{\bf A priori estimate}]
\label{prop:voltage:bounds:communication-height:abstract}
There is a constant $\bar{k}\geq 1$ such that, for every two disjoint sets $A,B\subseteq\spX$
and every node $x\in\spX\setminus (A\cup B)$,
\begin{align}
1 - \bar{k}\frac{\Psi(x,A)}{\Psi(A,B)} 
&\leq W_{A,B}(x) \leq \bar{k}\frac{\Psi(x,B)}{\Psi(A,B)}.
\end{align}
The constant $\bar{k}$ can be chosen to be $\abs{\spX}^4$.
\end{proposition}

The following is a generalisation of the latter proposition. It expresses the intuition that
small distance between two nodes implies small difference between their voltages. Its 
proof can be found in Appendix~\ref{apx:voltage:estimate}.

\begin{proposition}[{\bf A priori estimate}]
\label{prop:voltage:bounds:valleys:abstract}
There is a constant $\bar{k}\geq 1$ such that, for every two disjoint sets $A,B\subseteq\spX$
and every two nodes $x,y\in\spX$,
\begin{align}
\abs{W_{A,B}(x)-W_{A,B}(y)} 
&\leq \bar{k}\frac{\Psi(x,y)}{\Psi(A,B)}.
\end{align}
The constant $\bar{k}$ can be chosen to be $\abs{\spX}^4$.
\end{proposition}


\section{Metastability in reversible Markov chains}
\label{sec:mc:family}

In this section we discuss the metastable behaviour of reversible Markov chains 
in a certain asymptotic regime. Our treatment is based on Bovier and den 
Hollander~\cite[Chapters 7, 8 and 16]{BovHol15}, although our exposition is 
somewhat different. In Section~\ref{sec:hardcorebipartite} we will specialize to 
hard-core dynamics.

We consider a one-parameter family of discrete-time irreducible Markov chains
$\{X_{\lambda}(t)\}_{t\in\NN}$ on a finite state space $\spX$ with transition matrix 
$K_\lambda$ and \emph{reversible} stationary distribution $\pi_\lambda$.
The parameter $\lambda$ is assumed to be a real number. For hard-core dynamics, 
$\lambda$ determines the activity parameter at each site. (For Glauber dynamics of 
the Ising model, $\lambda$ would be the inverse temperature.) For brevity, we drop 
the subscript $\lambda$ from $X_{\lambda}(t)$, $K_\lambda$ and $\pi_\lambda$. We 
focus on the asymptotic regime $\lambda\to\infty$, where metastable phenomena 
are more prominent.

We will use the following notation for asymptotics:
\begin{itemize}
\item 
$f(\lambda)\prec g(\lambda)$ if $f(\lambda)=\smallo(g(\lambda))$ as $\lambda\to\infty$,
\item 
$f(\lambda)\preceq g(\lambda)$ if $f(\lambda)=\bigo(g(\lambda))$ as $\lambda\to\infty$, and
\item 
$f(\lambda)\asymp g(\lambda)$ if $f(\lambda)\preceq g(\lambda)$ and 
$g(\lambda)\preceq f(\lambda)$ as $\lambda\to\infty$.
\end{itemize}
For simplicity, we make a \emph{smoothness} assumption. Namely, we assume that all 
the transition probabilities $K(x,y)$ for different pairs $(x,y)$ are asymptotically comparable, 
i.e., for every two pairs of states $(x,y)$ and $(x',y')$, either $K(x,y)\prec K(x',y')$ or $K(x,y)
\succeq K(x',y')$ as $\lambda\to\infty$, and for every two states $x$ and $y$, either $\pi(x)
\prec\pi(y)$ or $\pi(x)\succeq\pi(y)$ as $\lambda\to\infty$. These conditions are trivially 
satisfied for hard-core dynamics on a bipartite graph. For convenience, we also assume 
that the graph of probable transitions of $K$ remains unchanged for all sufficiently large 
$\lambda$.

In Section~\ref{sec:metastability:formulation} we characterise metastabilty in terms of 
recurrence of metastable states. In Section~\ref{sec:duration} we link the mean 
metastable transition time to the effective resistance of an associated electric network. 
In Section~\ref{sec:expescape} we explain the ubiquity of the exponential limit law for the 
metastable transition time divided by its mean. In Section~\ref{sec:asymptail} we look 
at tail probabilities of the metastable transition time. In Section~\ref{sec:sharper} we 
derive a sharp asymptotics for the effective resistance. In Section~\ref{sec:bottle} we 
look at the passage through bottlenecks.


\subsection{A characterisation of metastability}
\label{sec:metastability:formulation}

One way to formulate metastability (in the asymptotic regime $\lambda\to\infty$) is in 
terms of the recurrence behaviour of individual states. A metastable state behaves 
as a recurrent state on short time scales and as a transient state on long time scales.
Other manifestations of metastability include a short transition period on the critical 
time scale and approximate exponentiality of the distribution of the transition time.

More specifically, when $\tau=\tau(\lambda)$ is a non-negative real-valued function,
we say that a state $a\in\spX$ is \emph{transient at time scale $\tau$} (or 
\emph{$\tau$-transient}, for short) when $G_\tau(a,a)\prec\tau$ as $\lambda\to\infty$
and \emph{recurrent at time scale $\tau$} (or \emph{$\tau$-recurrent}) when $G_\tau(a,a)
\succeq\tau$ as $\lambda\to\infty$. In intuitive terms, state $a$ is $\tau$-recurrent if
the Markov chain starting from $a$ spends, on average, a non-negligible fraction of
the time interval $[0,\tau)$ at $a$, and is $\tau$-transient otherwise.

In the reversible setting, there is a more convenient way to characterise recurrence 
and transience on a time scale, namely, in terms of escape times. For $a\in\spX$, define
\begin{align}
J(a) 
&\isdef \left\{x\neq a\colon\,\text{$\pi(x)\succeq\pi(a)$ as $\lambda\to\infty$}\right\},\\
J^-(a) 
&\isdef \left\{x\colon\, \text{$\pi(x)\succ\pi(a)$ as $\lambda\to\infty$}\right\}.
\nonumber
\end{align}
Thus, $J(a)$ is the set of states whose stationary probabilities are asymptotically not 
negligible compared to $a$, and $J^-(a)$ consists of those states whose stationary 
probabilities are asymptotically larger than the stationary probability of $a$. Whether 
$a$ is $\tau$-transient or not depends on whether the chain has sufficient time to reach 
$J^-(a)$ or not: once the chain is in $J^-(a)$, it will spend only a negligible portion of its time 
in $a$. We refer to the time taken to go from $a$ to $J^-(a)$ as the \emph{escape time} from $a$. 
The proof of the following proposition is given in Appendix~\ref{apx:transience:characterization}.

\begin{proposition}[{\bf Charactersation of metastability}]
\label{prop:transience:characterization}
Suppose that $\tau=\tau(\lambda)$ is a non-negative real-valued function. For every state 
$a\in\spX$, $G_\tau(a,a)\prec\tau$ if and only if $\xExp_a[T_{J^-(a)}]\prec\tau$ as $\lambda\to\infty$.
\end{proposition}

It follows from Proposition~\ref{prop:transience:characterization} that $\tau$-transience is 
monotone in $\tau$: if a state is transient at a time scale $\tau$, then it is also transient at 
any time scale $\tau'\succeq\tau$. In particular, the recurrence behaviour of every state $a$ 
undergoes a transition at the time scale $\tau_a\isdef\xExp_a[T_{J^-(a)}]$: the state $a$ is 
recurrent at any time scale $\tau\preceq\tau_a$ (a short time scale) and transient at any 
time scale $\tau\succ\tau_a$ (a long time scale). We call this the \emph{metastability transition} 
of state $a$. We refer to a state $a$ as a \emph{metastable} state when its metastability transition 
is non-trivial, i.e., when $J^-(a)\neq\varnothing$ and $\tau_a\to\infty$ as $\lambda\to\infty$.
Note that if $J^-(a)$ is empty, then the critical scale $\tau_a$ is $\infty$ ($a$ is recurrent at 
any scale). Hence, in this case we call $a$ a \emph{stable} state.

Our main objective is to derive a sharp asymptotics for the mean and the distribution of the 
escape time $\tau_a=\xExp_a[T_{J^-(a)}]$, and to provide some information (albeit partial) about 
the typical escape trajectories. In case of the hard-core dynamics on a bipartite graph (satisfying 
certain conditions) we will provide such a description for the state in which the weak part of 
the graph $U$ is covered with particles. This state turns out to be the ``most stable'' metastable 
state, i.e., the metastable state with the largest metastability scale. The transition from this 
metastable state to the stable state requires the formation of \emph{critical droplets} whose 
size and shape are characterised by the solutions of an isoperimetric problem.


\subsection{Mean escape time and transition duration}
\label{sec:duration}

The proofs of the following two propositions are given in Appendix~\ref{apx:escape:mean}.
The mean escape time from a metastable state has the following rough asymptotics in terms 
of critical resistance.

\begin{proposition}[{\bf Mean escape time: rough estimate}]
\label{prop:escape:mean}
For every $a\in\spX$, $\xExp_a[T_{J^-(a)}] \asymp \pi(a)\allowbreak\Psi(a,J^-(a))$ as $\lambda\to\infty$.
\end{proposition}

\noindent
(We use the convention $\Psi(x,\varnothing)\isdef\infty$.)
This estimate gives the order of the magnitude of the mean escape time, but fails to provide 
the pre-factor. On the other hand, replacing $J^-(a)$ with $J(a)$, we have the following sharp 
estimate for the mean passage time from $a$ to $J(a)$ in terms of effective resistance.

\begin{proposition}[{\bf Link between mean escape time and effective resistance}] 
\label{prop:almost-escape:mean}
For every $a\in\spX$, $\xExp_a[T_{J(a)}]=\pi(a)\effR{a}{J(a)}[1+\smallo(1)]$ as 
$\lambda\to\infty$.
\end{proposition}

In conjunction with a good estimate on effective resistance, the above two propositions can 
often be used to give a sharp asymptotic estimate (with a precise pre-factor) for the escape 
time from a metastable state. Indeed, suppose we know that, for every $x\in J(a)
\setminus J^-(a)$, the critical resistance $\Psi(x,J^-(x))$ is asymptotically smaller than 
the critical resistance $\Psi(a,J(a))$. Then Propositions~\ref{prop:almost-escape:mean} and
\ref{prop:escape:mean} immediately give $\xExp_a[T_{J^-(a)}]=\xExp_a[T_{J(a)}][1+\smallo(1)]$. 

We state this observation as the following corollary, which is proved in Appendix~\ref{apx:escape:mean}.
We say that a set $Z\subseteq\spX$ is \emph{upward closed} if $y\in Z$ whenever $\pi(y)
\succeq\pi(x)$ for some $x\in Z$. In the following we may for instance set $Z=J^-(a)$ or 
$Z=\{v\}$, where $v$ is the unique stable state.

\begin{corollary}[{\bf Mean escape time: sharp estimate}]
\label{cor:escape:mean:cascade}
Let $a\in\spX$, and let $Z\subseteq J(a)$ be a non-empty upward closed set. Suppose that 
$\pi(x)\Psi(x,J^-(x))\prec\pi(a)\Psi(a,J(a))$ for every $x\in J(a)\setminus Z$. Then $\xExp_a[T_Z]
=\pi(a)\effR{a}{J(a)}[1+\smallo(1)]$ as $\lambda\to\infty$.
\end{corollary}

A typical aspect of metastability is the relatively short duration of the transition on the critical 
time scale: while the system spends a long time before leaving a metastable state and moving 
to a more stable state, the actual transition occurs on a relatively shorter time scale. To formulate 
this, let $T^{(k)}_a\isdef\inf\{t>T^{(k-1)}\colon\, X(t)=a\}$ with $T_a^{(0)}=0$ be the $k$-th return 
time of state $a$. Given $Z\not\ni a$, define $N_Z\isdef\sup\{n>0\colon\,T^{(n)}_a<T_Z\}$.
The difference $T_Z-T^{(N_Z)}_a$ is the duration of the transition from $a$ to~$Z$. Note that, 
by the Markov property and time-homogeneity, $\xPr_a(T_Z-T^{(N_Z)}_a\in\cdot) =\xPr_a(T_Z
\in\cdot\,|\,T_Z<T^+_a)$. The following corollary is proved in Appendix~\ref{apx:avalanche}.

\begin{corollary}[{\bf Rapid transition}]
\label{cor:transition:avalanche}
Let $a\in\spX$, and let $Z\subseteq J(a)$ be a non-empty upward closed set. Suppose that 
$\pi(x)\Psi(x,J^-(x))\prec\pi(a)\Psi(a,J(a))$ for every $x\in J(a)\setminus Z$. Then $\xExp_a
[T_Z\,|\,T_Z<T^+_a]\prec\xExp_a[T_Z]$ as $\lambda\to\infty$.
\end{corollary}


\subsection{Exponential law for escape times}
\label{sec:expescape}

If $a$ is a metastable state (i.e., $J^-(a)\neq\varnothing$ and $\pi(a)\Psi(a,J^-(a))$ is large),
then it can take a long time for the chain to pass from $a$ to $J^-(a)$. Starting from $a$, the 
chain is much more likely to return back to $a$ quickly than to pass through the bottleneck 
between $a$ and $J^-(a)$. Each time the chain returns to $a$, the process starts afresh.
The transition thus requires many repeated trials, each with a small success probability.

The hitting time of a rare event in a regenerative process approximately follows an exponential 
law (Keilson~\cite[Section 8]{Kei79}). The following proposition formulates a version of this 
phenomenon. See Appendix~\ref{apx:exponential-law} for its proof.

\begin{proposition}[{\bf Exponential law for regenerative processes}]
\label{prop:exponential:renewal}
Let $\delta T$ be a positive random variable with finite mean and $B$ a Bernoulli random 
variable with success probability $\varepsilon>0$. Let $(\delta T_k, B_k)_{k\in\ZZ^+}$ be a 
sequence of independent copies of the pair $(\delta T,B)$. Define the associated renewal 
process by $T_0\isdef 0$ and $T_k\isdef T_{k-1}+\delta T_k$ for $k\geq 1$. Set $N\isdef
\inf\{k\colon\, B_k=\symb{1}\}$, $\mu\isdef\xExp[\delta T\,|\,B=\symb{0}]$,
$\eta\isdef\xExp[\delta T\,|\,B=\symb{1}]$
and $M\isdef\xExp[T_N]$. Take $\varepsilon$, $M$, $\mu$ and $\eta$ to be functions of a 
parameter $\lambda\in\RR$. Then
\begin{align}
\label{eq:thm:exponential}
\lim_{\lambda\to\infty} \xPr\left(\frac{T_N}{\xExp[T_N]} > t\right)
&= \ee^{-t} \qquad\text{uniformly in $t\in\RR^+$,}
\end{align}
provided $\varepsilon=\smallo(1)$ and $\varepsilon\frac{\eta}{\mu}=\smallo(1)$ (or equivalently, 
$\varepsilon=\smallo(1)$ and $\frac{\eta}{M}=\smallo(1)$) as $\lambda\to\infty$.
\end{proposition}

An immediate consequence is the approximate exponential distribution for the escape time from a 
metastable state, stated in the following corollary. See Appendix~\ref{apx:exponential-law} 
for its proof.

\begin{corollary}[{\bf Exponential escape time}]
\label{cor:escape:exponential}
Let $a\in\spX$, and let $Z\subseteq J(a)$ be a non-empty upward closed set (see Sec.~\ref{sec:duration}).
Suppose that 
$\pi(a)\Psi(a,J(a))\succ 1$, and $\pi(x)\Psi(x,J^-(x))\prec\pi(a)\Psi(a,J(a))$ for every $x\in J(a)
\setminus Z$. Then 
\begin{align}
\lim_{\lambda\to\infty} \xPr_a\left(\frac{T_Z}{\xExp_a[T_Z]} > t\right) 
&= \ee^{-t} \qquad\text{uniformly in $t\in\RR^+$.}
\end{align}
A similar statement holds for the continuous-time version of the process $\hat{X}(t)\isdef 
X\big(\xi([0,t])\big)$ constructed via an independent Poisson process $\xi$ with rate $\gamma$.
\end{corollary}


\subsection{Asymptotics for tail probabilities}
\label{sec:asymptail}

In the previous section, we saw that the tail probability of the escape time from a metastable 
state is asymptotically exponentially small, namely, $\xPr_a\big(T_Z>t\xExp_a[T_Z]\big)
=\ee^{-t}[1+\smallo(1)]$ as $\lambda\to\infty$. In this section, we derive similar exponential 
upper bounds for the tail probabilities and conditional tail probabilities of more general hitting 
times using rougher but more flexible regeneration arguments. Such exponential upper bounds 
are one of the ingredients of the path-wise approach to metastability (see e.g.\ the paper by 
Manzo, Nardi, Olivieri and Scoppola~\cite{ManNarOliSco04}). The material of this section is 
not used in the rest of the current paper but will be needed in our follow-up paper.

Recall from Proposition~\ref{prop:escape:mean} that $\xExp_a[T_{J^-(a)}]\asymp\pi(a)\Psi(a,J^-(a))$
for each $a\in\spX$. For $A\subseteq\spX$, define
\begin{align}
	\Gamma(A) &\isdef
		\sup_{x\in A} \pi(x)\Psi\big(x,J^-(x)\big)\;.
\end{align}
and note that $\sup_{x\in A}\xExp_x[T_{A^\complement}]\preceq\Gamma(A)$ as $\lambda\to\infty$. 
(Recall the convention $\Psi(x,\varnothing)\isdef\infty$.) The following proposition is a variant of 
Theorem~3.1 in~\cite{ManNarOliSco04}. Its proof can be found in Appendix~\ref{apx:exponential-law}.

\begin{proposition}[{\bf Tail probabilities of exit time}]
\label{prop:SES:exit}
	Let $A\subseteq\spX$ be an arbitrary non-empty set of states. There is a constant $\alpha<1$ 
	such that, for every function $\rho=\rho(\lambda)\succ 1$,
	\begin{align}
		\sup_{x\in A} \xPr_x\big(T_{\partial A}>\rho\Gamma(A)\big)
			&\preceq \alpha^\rho	\qquad\text{as $\lambda\to\infty$.}
	\end{align}
\end{proposition}
\noindent 
Examples of useful choices for $\rho$ are $\rho\isdef\lambda^\delta$ (for a small constant $\delta>0$)
and $\rho\isdef\log\lambda$.

The above proposition can be used to bound the tail and expected value of the exit time of a set $A$
conditioned on hitting a certain subset of $\partial A$ upon exit. Set
\begin{align}
	\kappa=\kappa(\lambda) &\isdef
		\min\{K(a,b): a,b\in\spX, K(a,b)>0\} \;.
\end{align}

\begin{proposition}[{\bf Tail probabilities of conditional exit time}]
\label{prop:SES:exit:tail:conditional}
	Let $A\subseteq\spX$ be an arbitrary set of states. Consider an arbitrary partitioning 
	of $\partial A$ into two non-empty sets $B_1$ and $B_2$. There is a constant 
	$\alpha<1$ \textup{(}the one in Proposition~\ref{prop:SES:exit}\textup{)} such that, 
	for every function $\rho=\rho(\lambda)\succ 1$,
	\begin{align}
		\sup_{x\in A}
			\xPr_x\big(T_{B_1}>\rho\Gamma(A) \,\big|\, T_{B_1}<T_{B_2}\big)
			&\preceq \alpha^\rho\kappa^{-\abs{A}}
			\qquad\text{as $\lambda\to\infty$.}
	\end{align}
\end{proposition}

\begin{proposition}[{\bf Conditional mean exit time}]
\label{prop:SES:exit:mean:conditional}
	Let $A\subseteq\spX$ be an arbitrary set of states. Consider an arbitrary partitioning 
	of $\partial A$ into two non-empty sets $B_1$ and $B_2$. There is a constant $\alpha<1$
	\textup{(}the one in Proposition~\ref{prop:SES:exit}\textup{)} such that, for every function 
	$\rho=\rho(\lambda) \succ 1$,
	\begin{align}
		\sup_{x\in A}
			\xExp_x\big(T_{B_1} \,\big|\, T_{B_1}<T_{B_2}\big)
			&\preceq \rho\Gamma(A)
			\qquad\text{as $\lambda\to\infty$,}
	\end{align}
	provided $\alpha^\rho\,\kappa^{-\abs{A}}\to 0$
	as $\lambda\to\infty$.
\end{proposition}

\noindent The proofs can be found in Appendix~\ref{apx:exponential-law}.


\subsection{Sharp asymptotics for effective resistance}
\label{sec:sharper}

As we saw earlier, a sharp estimate on the mean escape time requires a sharp estimate on 
effective resistance. Sharp asymptotics for effective resistance between two sets can be 
obtained through a detailed understanding of the bottleneck between them. The bottleneck 
between two sets is often described by a notion of \emph{critical gate}, which pinpoints
the critical transitions in a typical passage from one set to another. The notion of critical 
gate used below is not as general as it seems. For instance, it is not directly applicable 
to Glauber dynamics for the Ising model, but it suffices for our hard-core model. 

Let $A,B\subseteq\spX$ be two disjoint non-empty sets. We call a pair of disjoint sets 
$Q,Q^*\subseteq\spX$ a \emph{critical pair} between $A$ and $B$ when (see 
Fig.~\ref{fig:critical-gate})
\begin{enumerate}[label=\alph*)]
\item 
$r(x,y)\asymp \Psi(A,B)$ for every $x\in Q$ and $y\in Q^*$ with $x\link y$,
\item 
$\Psi(A,x)\prec\Psi(A,B)$ for every $x\in Q$,
\item 
$\Psi(y,B)\prec\Psi(A,B)$ for every $y\in Q^*$,
\item 
every optimal path from $A$ to $B$ passes through a transition $x\to y$ with $x\in Q$ 
and $y\in Q^*$.
\end{enumerate}
By an \emph{optimal path} from $A$ to $B$, we mean a path whose critical resistance is 
of the same order as $\Psi(A,B)$, i.e., a path $\omega\colon\,A\pathto B$ with $r(\omega)
\asymp\sup_{e\in\omega}r(e)\asymp\Psi(A,B)$ as $\lambda\to\infty$. Observe that an 
optimal path $A\pathto Q$ does not pass through $Q^*$, and an optimal path $Q^*\pathto B$ 
does not pass through $Q$. If $(Q,Q^*)$ is a critical pair between $A$ and $B$, then we 
call the set
\begin{align}
[Q,Q^*] &\isdef \{(x,y)\colon \text{$x\in Q$, $y\in Q^*$ and $x\link y$}\}
\end{align}
of probable transitions between $Q$ and $Q^*$ the \emph{critical gate} between $A$ and $B$.

\begin{figure}[htbp]
\centering
{\tikzsetfigurename{critical_gate}
\begin{tikzpicture}[scale=1,baseline,>=stealth,shorten >=1]
\criticalgate
\end{tikzpicture}
}
\caption{A critical gate $[Q,Q^*]$ between $A$ and $B$.}
\label{fig:critical-gate}
\end{figure}
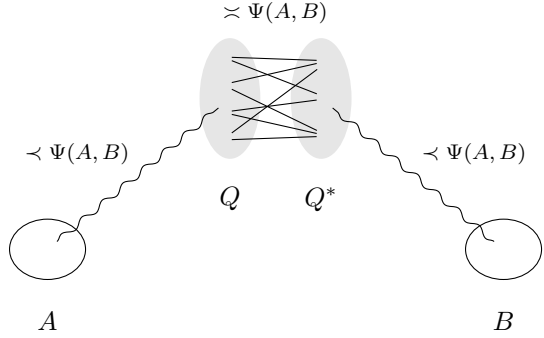

Given a critical gate $[Q,Q^*]$ between $A$ and $B$, we define a set
\begin{align}
\label{eq:critical-gate:behind}
S(A,Q,Q^*,B) &\isdef \left\{x\in\spX\colon\,
\parbox[c]{20em}{there exists a path $\omega:A\pathto x$ not passing $Q^*$ \\
such that $\Psi(\omega)\preceq\Psi(A,B)$}\right\},
\end{align}
which we think of as the set of states ``behind the critical gate''. We have used the notation 
$\Psi(\omega)\isdef\sup_{e\in\omega}r(e)$ for the critical resistance of the path $\omega$.
Note that $\Psi(\omega)\asymp r(\omega)$. The following proposition is proved in 
Appendix~\ref{apx:critical-gate}.

\begin{proposition}[{\bf Characterisation of critical gate}]
\label{prop:critical-gate:behind}
Let $[Q,Q^*]$ be a critical gate between two disjoint non-empty sets $A,B\subseteq\spX$
and $S\isdef S(A,Q,Q^*,B)$. If $(x,y)\in S\times S^\complement$ and $x\link y$, then either 
$r(x,y)\succ\Psi(A,B)$ or $(x,y)\in Q\times Q^*$.
\end{proposition}

In general, a critical gate between two sets $A$ and $B$ (as defined above) may or may not 
exist.  Even when it exists, identifying a critical gate may require painstaking combinatorial 
analysis. However, once available, a critical gate provides a sharp estimate on the effective 
resistance between $A$ and $B$. The following proposition is proved in Appendix~\ref{apx:critical-gate}.

\begin{proposition}[{\bf Effective resistance: sharp estimate using critical gate}]
\label{prop:effective-conductance:critical-gate}
Suppose that $(Q,Q^*)$ is a critical pair between $A$ and $B$.  Then
\begin{align}
	\effC{A}{B} &= c(Q,Q^*)\, [1+\smallo(1)] \asymp \frac{1}{\Psi(A,B)}
		\qquad\text{as $\lambda\to\infty$,}
\end{align}
where as usual, $c(Q,Q^*) \isdef \displaystyle{\mathop{\sum_{x\in Q}
\sum_{y\in Q^*}}_{x\link y} c(x,y)}$.
\end{proposition}


\subsection{Passage through the bottleneck}
\label{sec:bottle}

Let $a$ be an arbitrary state and $B$ a set not containing $a$. If a critical gate between 
$a$ and $B$ exists, then the passage from $a$ to $B$ is almost surely through the critical 
gate. The following proposition is proved in Appendix~\ref{apx:critical-gate}.

\begin{proposition}[{\bf Critical gate is bottleneck}]
\label{prop:critical-gate:abstract:choice}
Suppose that $(Q,Q^*)$ is a critical pair between $a$ and $B$, and $S\isdef S(a,Q,Q^*,B)$ 
the set of states behind the critical gate. As $\lambda\to\infty$,
\begin{enumerate}[label={\rm (\roman*)}]
	\item 
		$\xPr_a(T_{xy}\leq T_B)=\smallo(1)$ for
		$(x,y)\in (S\times S^\complement) \setminus (Q\times Q^*)$ with $x\link y$,
	\item 
		$\xPr_a(T_{yx}\leq T_B)=\smallo(1)$ for $(x,y)\in S\times S^\complement$ with $x\link y$,
	\item 
		$\xPr_a(T_{xy}\leq T_B)=\dfrac{c(x,y)}{c(Q,Q^*)}[1+\smallo(1)]$ for $(x,y)\in Q\times Q^*$ 
		with $x\link y$.
\end{enumerate}
\end{proposition}


\section{Hard-core dynamics on bipartite graphs}
\label{sec:hardcorebipartite}

In this section, we apply the results in Sections~\ref{sec:mc:reversible}--\ref{sec:mc:family} 
to describe the metastable behaviour of the hard-core process on bipartite graphs. We use 
the setting of Section~\ref{sec:intro}.

After some preparatory observations (Section~\ref{sec:preparatory}), we start by listing a few 
`simple examples' for which the above tasks can be carried out via simple inspection 
(Section~\ref{sec:hard-core:examples}). For more `sophisticated examples' the problem of 
identifying the critical resistance and the critical gate lead to a (non-standard) combinatorial 
isoperimetric problem (Section~\ref{sec:hard-core:isoperimetric}). One advantage of working with 
bipartite graphs is a natural ordering on the configuration space (Section~\ref{sec:hard-core:ordering}). 
We exploit this ordering to identify the critical resistance (Section~\ref{sec:hard-core:progressions}) 
and to prove the absence of trap states (Section~\ref{sec:hard-core:no-trap}) under certain 
assumptions on the solutions of the isoperimetric problem. After that we are ready to give the 
proof of Theorems~\ref{thm:mean-crossover:magnitude}--\ref{thm:crossover:sharp} 
(Section~\ref{sec:main-theorems:proof}). The identification of the critical gate requires a 
detailed combinatorial analysis of the configurations close to the critical droplet 
(Section~\ref{sec:hard-core:critical-gate}). We illustrate the results with four more 
`sophisticated examples', the hard-core model and the Widom-Rowlinson model on a torus, 
on a hypercube and on tree-like graphs (Section~\ref{sec:main-examples:proof}).


\subsection{Preparatory observations}
\label{sec:preparatory}

Recall that the underlying bipartite graph has two parts $U$ and $V$. Particles are added 
to or removed from each site independently with constant rates and subject to the exclusion 
constraints prescribed by the graph. The rates of adding particles to empty sites in $U$ and 
$V$ are $\lambda$ and $\bar{\lambda}$, respectively, and the rate of removing a particle 
from a site is $1$. We assume that $\bar{\lambda}=\varphi(\lambda)=\lambda^{1+\alpha
+\smallo(1)}$ as $\lambda\to\infty$, where $0<\alpha<1$. We write $u$ and $v$ to denote
the fully-packed configurations with particles at every site of $U$ and $V$, respectively.

We let $K$ be the transition kernel of the discrete-time version of the Markov chain, and 
$\gamma=(1+\lambda)\abs{U} + (1+\bar{\lambda})\abs{V}$ the Poisson rate for the 
continuous-time Markov chain. The stationary distribution of the Markov chain is 
\begin{align}
\pi(x) &= \frac{1}{Z}\lambda^{\abs{x_U}}\bar{\lambda}^{\abs{x_V}},
\end{align}
where $x_U=x\cap U$ and $x_V=x\cap V$ are the restrictions of the configuration $x$
to $U$ and $V$, respectively, and $Z$ is the normalising constant. This has the asymptotic 
form
\begin{align}
\pi(x) &= \frac{1}{Z}\lambda^{-H(x)+\smallo(1)}\qquad\text{as $\lambda\to\infty$,}
\end{align}
where $H(x)\isdef-\abs{x_U}-(1+\alpha)\abs{x_V}$ is the height or energy of configuration $x$.
The conductance between two configurations $x,y\in\spX$ is given by
\begin{align}
\label{eq:hard-core:conductance}
c(x,y) 
&= \frac{1}{\gamma}\max\{\pi(x),\pi(y)\} 
= \frac{1}{\gamma Z}\lambda^{-\min\{H(x),H(y)\}+\smallo(1)}
\end{align}
when $x$ and $y$ differ at a single site, and $0$ otherwise.

A transition between two distinct configurations $x$ to $y$ occurs by adding or removing 
a particle. We denote a transition corresponding to adding a particle by $x\xadd[V]y$ or 
$x\xadd[U]y$, depending on whether the particle is added to $V$ or to $U$. If we do not 
want to emphasise where the new particle is placed, then we simply write $x\xadd y$.
Transitions corresponding to removing a particle are denoted accordingly by $x\xremove[V]y$, 
$x\xremove[U]y$ or $x\xremove y$.

In the asymptotic regime $\lambda\to\infty$, the configuration $v$ is a stable state, in the 
sense that it is recurrent on any time scale (see Section~\ref{sec:metastability:formulation}),
as long as $\abs{U}<(1+\alpha)\abs{V}$. Once the chain reaches the state $v$, it spends 
an overwhelming portion of its time at $v$. In particular, all the other states are transient 
on every time scale larger than $\sup_{x\neq v}\xExp_x[T_v]$. Among the other states, 
we expect $u$ to be the most stable. Our aim is to describe the transition from $u$ to $v$,
at least for some characteristic choices of the underlying graph. To this end, we
\begin{enumerate}[label=(\roman*)]
\item
identify $\Psi\big(u,J(u)\big)$, the critical resistance between $u$ and $J(u)$,
\item 
verify that the Markov chain has no \emph{trap} state, i.e., every configuration $x\notin\{u,v\}$ 
satisfies $\pi(x)\Psi\big(x,J^-(x)\big)\prec\pi(u)\Psi\big(u,J(u)\big)$ as $\lambda\to\infty$,
\item 
identify a critical gate between $u$ and $J(u)$.
\end{enumerate}
Item (ii), together with Corollary~\ref{cor:escape:exponential}, shows the exponentiality of the 
distribution of the transition time from $u$ to $v$ on the time scale $\pi(u)\Psi(u,v)$. 
Items (i--iii), together with Corollary~\ref{cor:escape:mean:cascade} and 
Propositions~\ref{prop:effective-conductance:critical-gate}--\ref{prop:critical-gate:abstract:choice},
lead to a sharp asymptotic estimate for the expected transition time and the identification 
of the shape of the critical droplets.


\subsection{Simple examples}
\label{sec:hard-core:examples}

\begin{example}[{\bf Complete bipartite graph}]
\label{exp:hard-core:complete-bipartite}
The most pronounced example of metastability of the hard-core process occurs when 
the underlying graph is a \emph{complete} bipartite graph $K_{m,n}$, i.e., $\abs{U}=m$ 
and $\abs{V}=n$, and every site in $U$ is connected by an edge to every site in $V$ 
(Fig.~\ref{fig:graph:complete-bipartite}). The configuration space is $\spX=2^U\cup 2^V$.
We assume that $m\leq(1+\alpha)n$ to make sure that the configuration $v$ is a stable 
state, in particular, $v\in J(u)$. Note that every path from $u$ to $v$ has a transition 
from a configuration with a single particle on $U$ and no particle on $V$ to the empty 
configuration $\varnothing$. Such a transition has the largest resistance $\frac{\gamma}
{\lambda\pi(\varnothing)}=\gamma Z\lambda^{-1}$. Therefore the critical resistance between 
$u$ and $v$ is $\Psi(u,v)=\frac{\gamma}{\lambda\pi(\varnothing)}$. On the other hand, 
from any other configuration $x\notin\{u,v\}$ it is possible to add a new particle, which 
means that $\Psi(x,J^-(x))\preceq\frac{\gamma}{\lambda\pi(x)}$. Therefore
\begin{align}
\pi(x)\Psi(x,J^-(x))\preceq\gamma\lambda^{-1} 
&\prec \gamma\lambda^{-1}\frac{\pi(u)}{\pi(\varnothing)} = \pi(u)\Psi(u,v),
\end{align}
i.e., the chain has no trap. In particular,
\begin{align}
\xExp_u[T_v] &= \pi(u)\effR{u}{v} [1 + \smallo(1)] \qquad\text{as $\lambda\to\infty$}
\end{align}
(Corollary~\ref{cor:escape:mean:cascade}) with an asymptotic exponential law for $T_v$ and 
its continuous-time version $\hat{T}_V$ (Corollary~\ref{cor:escape:exponential}), and rapid 
transition from $u$ to $v$ (Corollary~\ref{cor:transition:avalanche}).
	
The effective resistance can now be accurately estimated by identifying the critical gate 
between $u$ and $v$, but for the sake of exposition, let us estimate it by direct calculation. 
This is possible because of the high degree of symmetry in the graph. Let $W$ be the 
voltage when $u$ is connected to a unit voltage source and $v$ is connected to the ground.
By symmetry, all the configurations with $i\neq 0$ particles on $U$ have the same voltage.
Therefore, by the short-circuit principle, we can identify them with a single node, which we 
call $\tbinom{U}{i}$. Similarly, we can contract all the configurations with $j\neq 0$ particles 
on $V$ with a single node $\tbinom{V}{j}$. We then obtain a new network with nodes
\begin{align}
\left\{\tbinom{U}{m}, \tbinom{U}{l-1}, \ldots, \tbinom{U}{1}, \varnothing,
\tbinom{V}{1}, \tbinom{V}{2}, \ldots, \tbinom{V}{n}\right\},
\end{align}
where $\tbinom{U}{i}$ is connected to $\tbinom{U}{i-1}$ by a resistor with conductance
\begin{align}
c^*(\tbinom{U}{i},\tbinom{U}{i-1}) 
&= \mathop{\sum_{x\in \tbinom{U}{i}}\sum_{y\in\tbinom{U}{i-1}}}_{y\link x} c(x,y)
= i\,\binom{m}{i} \frac{\lambda^i}{Z\,\gamma},
\end{align}
and, similarly, $\tbinom{V}{j}$ is connected to $\tbinom{V}{j-1}$ by a resistor with conductance
\begin{align}
c^*(\tbinom{V}{j},\tbinom{V}{j-1}) 
&= j\,\binom{n}{j} \frac{\bar{\lambda}^j}{Z\,\gamma}.
\end{align}
We now have, by the series law, 	
\begin{align}
\effR{u}{v} = \effR[*]{\tbinom{U}{l}}{\tbinom{V}{m}} 
&= \sum_{i=1}^m \frac{Z\,\gamma}{i\,\binom{m}{i}\lambda^i} 
+ \sum_{j=1}^n \frac{Z\,\gamma}{j\,\binom{n}{j}\bar{\lambda}^j}.
\end{align}
As $\lambda\to\infty$, the dominant term is $i=1$ (corresponding to removal of the last particle 
from $U$).  Hence,
\begin{align}
\effR{u}{v} &= \frac{Z\,\gamma}{m\,\lambda}[1+\smallo(1)].
\end{align}
	
Alternatively, it is easy to see that if we let $Q$ be the set of all configurations that have a 
single particle on $U$ and $Q^*\isdef\{\varnothing\}$, then $[Q,Q^*]$ is a critical gate between 
$u$ and $v$, and we obtain (Proposition~\ref{prop:effective-conductance:critical-gate}) that
\begin{align}
\effC{u}{v} &= c(Q,Q^*)[1+\smallo(1)] = m \frac{\lambda}{\gamma\, Z}[1+\smallo(1)].
\end{align}

In conclusion,	
\begin{align}
\xExp_u[T_v] &= \frac{1}{m}\gamma\,\lambda^{m-1} [1+\smallo(1)]
\qquad\text{as $\lambda\to\infty$,}
\end{align}
for the hitting time in the discrete-time setting and $\xExp_u[\hat{T}_v] = \frac{1}{m}\lambda^{m-1}
[1+\smallo(1)]$ for the hitting time in the continuous-time setting. Furthermore, we know that the 
trajectory from $u$ to $v$ almost surely involves a transition through exactly one of the $m$ 
transitions $Q\to Q^*$, each occurring with probability $1/m$ 
(Proposition~\ref{prop:critical-gate:abstract:choice}).
\hfill\exampleqed
\end{example}

\begin{figure}[htbp]
\centering
\begin{subfigure}[b]{0.3\textwidth}
\centering
{
\tikzsetfigurename{graph_completebipartite}
\begin{tikzpicture}[scale=0.8,baseline]
\bipartitecomplete
\end{tikzpicture}
}
\caption{A complete bipartite graph}
\label{fig:graph:complete-bipartite}
\end{subfigure}
\begin{subfigure}[b]{0.3\textwidth}
\centering
{
\tikzsetfigurename{graph_evencycle}
\begin{tikzpicture}[scale=1,baseline] 
\evencycle
\end{tikzpicture}
}
\caption{An even cycle}
\label{fig:graph:even-cycle}
\end{subfigure}
\begin{minipage}[b]{0.3\textwidth}
\centering
\begin{subfigure}[b]{\textwidth}
\centering
{
\tikzsetfigurename{graph_oddpath}
\begin{tikzpicture}[scale=0.7,baseline]
\oddpath
\end{tikzpicture}
}
\caption{A path with odd length}
\label{fig:graph:odd-path}
\end{subfigure}
\smallskip

\begin{subfigure}[b]{\textwidth}
\centering
{
\tikzsetfigurename{graph_evenpath}
\begin{tikzpicture}[scale=0.6,baseline] 
\evenpath
\end{tikzpicture}
}
\caption{A path with even length}
\label{fig:graph:even-path}
\end{subfigure}
\end{minipage}
\caption{Some examples of bipartite graphs.}
\label{fig:graph:bipartite:examples}
\end{figure}
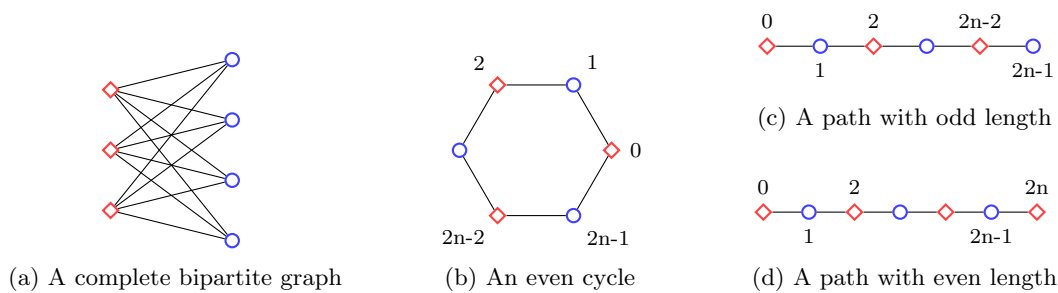

\begin{example}[{\bf Even cycle}]
\label{exp:hard-core:cycle}
Suppose that the underlying graph is an even cycle $\ZZ_{2n}$ (Fig.~\ref{fig:graph:even-cycle})
with $U=\{0,2,\ldots,2n-2\}$ and $V=\{1,3,\ldots,2n-1\}$. The critical transition when going from 
$u$ to $v$ in an optimal path is between a configuration with a single particle missing from a 
site in $U$ and a configuration with two particles missing from two consecutive sites in $U$.
After that, the Markov chain can go ``downhill'' by adding a particle to the freed site in $V$ and 
continue alternating between moves $\symb{-U}$ and $\symb{+V}$ until the stable configuration 
$v$ is reached. Thus, if $Q$ is the set of configurations with a particle missing from a single site 
in $U$ and $Q^*$ is the set of configurations with particles missing from two consecutive sites 
in $U$, the critical gate is $[Q,Q^*]$.  Assuming that there is no trap state (i.e., $\pi(x)\Psi(x,J^-(x))
\prec\pi(u)\Psi(u,v)$ for all $x\in J(u)\setminus\{v\}$), we find
\begin{align}
\effR{u}{v} &= \frac{1}{2n} \gamma Z \lambda^{-(n-1)}[1+\smallo(1)]
\qquad\text{as $\lambda\to\infty$,}
\end{align}
which gives
\begin{align}
\xExp_u[T_v] 
&= \frac{1}{2n}\gamma\lambda[1+\smallo(1)],\\
\xExp_u[\hat{T}_v] 
&= \frac{1}{2n}\lambda[1+\smallo(1)],
\nonumber
\end{align}
in the discrete-time and continuous-time setting, respectively. The hitting times $T_v$ and 
$\hat{T}_v$ are again asymptotically exponentially distributed, and the Markov chain undergoes 
a rapid transition when going from $u$ to $v$. Furthermore, the chain goes almost surely 
through exactly one of the critical transitions $Q\to Q^*$ when going from $u$ to $v$, each 
chosen with probability $\frac{1}{2n}$.
	
To see that the chain has no trap, we note that any configuration in $J(u)$ must have at least 
one particle on $V$. Thus from a configuration $x\in J(u)\setminus\{v\}$,
it is either possible to add a new particle on 
$V$ or first remove a particle from $U$ and then add a new particle on $V$, so that $\pi(x)
\Psi(x,J^-(x))\preceq\gamma$ as $\lambda\to\infty$.
\hfill\exampleqed
\end{example}


\begin{example}[{\bf Path with odd length}]
\label{exp:hard-core:path:odd-length}
Consider a path with odd length (Fig.~\ref{fig:graph:odd-path}), and let $U=\{0,2,\ldots,2n-2\}$ 
and $V=\{1,3,\ldots,2n-1\}$. Despite its simplicity, this example illustrates a phenomenon
that is not present in the other examples considered in this paper.
Namely, in this example the condition of absence of traps is not satisfied.
As a result, the scaled crossover time from~$u$ to~$v$ does not converge to an exponential random variable
but to the sum of $n$ independent exponential random variables.

Indeed, consider the continuous-time process and assume that $\lambda$ is very large.
Starting from $u$, it takes a rate~$1$ exponential time for each particle on $U$ to be removed.
Once a particle is removed, it is quickly replaced by another particle in a time that is $\smallo(1)$
so that at an overwhelming majority of the times the system is at a maximally packed configuration.
If the particle is removed from any site other than $2n-2$, the new particle arrives necessarily
at the same position, while if the particle is removed from site $2n-2$, the replacing particle
arrives with probability $1-\smallo(1)$ at site $2n-1$.
In the next stage, after a time with approximate exponential distribution,
a particle is removed from site $2n-4$ and is replaced with a particle at site $2n-3$.
In the same fashion, after $n$ such replacements, the Markov chain arrives at configuration $v$.
Thus, in the limit $\lambda\to\infty$, the crossover time $\hat{T}_v$ starting from $u$
becomes a sum of $n$ independent exponential random variables each with rate $1$.

Let us sketch how this can be made precise using the machinery of
the previous sections.
For $k\in\{0,\ldots,n-1\}$, let $q_k$ denote the configuration with particles
on $\{2i: i<2(n-k)\}\cup\{2i+1: i\geq 2(n-k)\}$, and let $q^*_k$ be the configuration
obtained from $q_k$ by removing a particle from $2(n-k-1)$.
Observe that $q_0=u$ and set $q_n\isdef v$.
One can verify that $\Psi\big(q_k,J(q_k)\big)=r(q_k,q^*_k)=\gamma/\pi(q_k)$
and that $(\{q_k\},\{q^*_k\})$ is a critical pair between $q_k$ and $J(q_k)$.
Therefore, Corollary~\ref{cor:escape:mean:cascade} and Proposition~\ref{prop:effective-conductance:critical-gate}
imply that $\xExp_{q_k}[T_{J(q_k)}]=\gamma[1+\smallo(1)]$ and Corollary~\ref{cor:escape:exponential}
shows that starting from $q_k$, the hitting time $T_{J(q_k)}/\gamma$ is asymptotically exponentially distributed
with rate~$1$.
Proposition~\ref{prop:critical-gate:abstract:choice} and the fact that
$K\big(q^*_k,q_{k+1}\big)=\bar{\lambda}/\gamma=1-\smallo(1)$ imply that
$\xPr_{q_k}(T_{J(q_k)}=T_{q_{k+1}})=1-\smallo(1)$.
It follows that as $\lambda\to\infty$, the scaled crossover time $T_{q_n}/\gamma$
converges in distribution to a sum of $n$ independent exponential random variables with rate~$1$
corresponding to the segments $T_{q_{k+1}}-T_{q_k}$.
\hfill\exampleqed
\end{example}

\begin{example}[{\bf Path with even length and even endpoints}]
\label{exp:hard-core:path:even-length}
The hard-core process on a path with \emph{even} length (Fig.~\ref{fig:graph:even-path})
has quite a different behaviour. Let $U=\{0,2,\ldots,2n\}$ and $V=\{1,3,\ldots,2n-1\}$, so 
both endpoints of the path belong to $U$. In this case, the trajectory from $u$ to $v$ is 
closer to the hard-core model on an even cycle (Example~\ref{exp:hard-core:cycle}).
We similarly find that
\begin{align}
\xExp_u[\hat{T}_v] 
&= \frac{1}{2n}\lambda[1+\smallo(1)]
\qquad\text{as $\lambda\to\infty$,}
\end{align}
with an asymptotic exponential law for $\hat{T}_v$.
\hfill\exampleqed
\end{example}

\begin{example}[{\bf Even cyclic ladder}]
\label{exp:hard-core:ladder}
Let the underlying graph be the cyclic ladder $\ZZ_{2n}\times\ZZ_2$ 
(Fig.~\ref{fig:graph:cyclic-ladder}) with $U\isdef\{(i,j)\colon\,i+j=0\pmod{2}\}$ and 
$V\isdef\{(i,j)\colon\, i+j=1\pmod{2}\}$. Every site in the graph has three neighbours.
Let $Q$ be the set of configurations that are obtained from $u$ by removing two 
particles from the neighbourhood of a site $k\in V$, and $Q^*$ the set of configurations 
that are obtained from $u$ by removing three particles from the neighbourhood of a 
site $k\in V$. We may verify that $[Q,Q^*]$ is a critical gate, and that the Markov chain 
has no trap. There are $6n$ possible transitions $Q\to Q^*$, each having resistance 
$\gamma Z \lambda^{-(n-2)}$. It follows that state $u$ undergoes a metastability 
transition with
\begin{align}
\xExp_u[\hat{T}_v] 
&= \frac{1}{6n}\lambda^2[1+\smallo(1)]
\qquad\text{as $\lambda\to\infty$,}
\end{align}
and from $u$ the distribution of $\hat{T}_v/\xExp_u[\hat{T}_v]$ converges to an exponential 
random variable with unit rate. Furthermore, the transition occurs within a relatively 
shorter period and goes (almost surely) through exactly one of the moves $Q\to Q^*$,
each with probability $\frac{1}{6n}$.
\hfill\exampleqed
\end{example}

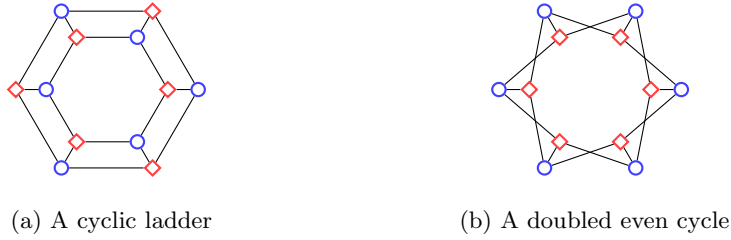
\begin{figure}[htbp]
\centering
\begin{subfigure}[b]{0.4\textwidth}
\centering
{
\tikzsetfigurename{graph_cyclicladder}
\begin{tikzpicture}[scale=0.8,baseline]
\cyclicladder
\end{tikzpicture}
}
\caption{A cyclic ladder}
\label{fig:graph:cyclic-ladder}
\end{subfigure}
\begin{subfigure}[b]{0.4\textwidth}
\centering
{
\tikzsetfigurename{graph_evencycle}
\begin{tikzpicture}[scale=0.8,baseline]
\doubledevencycle
\end{tikzpicture}
}
\caption{A doubled even cycle}
\label{fig:graph:even-cycle:doubled}
\end{subfigure}
\caption{A doubled even cycle is isomorphic to a cyclic ladder.}
\label{fig:graph:bipartite:ladder}
\end{figure}

\begin{example}[{\bf Widom-Rowlinson on an even cycle}]
\label{exp:widom-rowlinson:cycle}
As discussed earlier, the Widom-Rowlinson model on a graph is equivalent to the 
hard-core model on the doubled version of that graph. This example reduces to 
Example~\ref{exp:hard-core:ladder} after we note that the doubled graph of a cycle 
$\ZZ_{2n}$ is isomorphic to a cyclic ladder (Fig.~\ref{fig:graph:bipartite:ladder}).
\hfill\exampleqed
\end{example}

Note that in each of the above examples, the expected transition time $\xExp_u[\hat{T}_v]$
and the critical gate are independent of the parameter $\alpha$. This is not consistent with 
the physical intuition of a critical droplet as a point of balance between the cost of removing 
particles from $U$ and the gain of placing particles on $V$. Such physical intuition becomes
the key to identifying the critical gate when the underlying graph has a more geometric 
structure. We will keep as our guiding example an even torus $\ZZ_m\times\ZZ_n$.


\subsection{Sophisticated examples}
\label{sec:hard-core:isoperimetric}

The problem of identifying the critical gate between $u$ and $v$ (or $u$ and $J(u)$)
gives rise to a combinatorial isoperimetric problem. The reason for the appearance of 
an isoperimetric problem can be intuitively understood as follows. When $\lambda$ is 
large, the Markov chain tends to remain at configurations of particles that are close to 
maximal packing arrangements. Whenever one or more particles disappear from the 
graph, other particles quickly replace them, though potentially on different sites. Since 
the disappearance of particles is a much slower process, the typical trajectories tend 
to go through configurations that require the removal of the least possible number of 
particles. The system thus tends to make the transition from $u$ to $v$ by growing a 
droplet of closely-packed particles on $V$ in such a way as to require the removal 
of less particles from $U$. In particular, near the bottleneck between $u$ and $v$ (i.e., 
close to the largest necessary deviation), the system typically goes through maximal 
packing configurations that are as efficient as possible, playing the role of \emph{critical 
droplet}. Near the bottleneck, the system solves the optimisation problem of maximal 
packing with a constraint on the number of particles on $V$, i.e., the size of the critical 
droplet.

Let us therefore define
\begin{align}
\Delta(x) &\isdef \abs{U\setminus x_U} - \abs{x_V} = \abs{U} - \abs{x}
&& \text{for $x\in\spX$,}\nonumber\\
\Delta(A) &\isdef \abs{N(A)}-\abs{A}
&& \text{for $A\subseteq V$,}\nonumber\\
\Delta(s) &\isdef \inf\{\Delta(A)\colon\, \text{$A\subseteq V$ and $\abs{A}=s$}\}\nonumber\\
&= \inf\{\Delta(x)\colon\,\text{$x\in\spX$ and $\abs{x_V}=s$}\}
&& \text{for $s\in\NN$.}
\end{align}
Note that the stationary probability of a configuration $x\in\spX$ with $s\isdef \abs{x_V}$ 
can be written as
\begin{align}
\pi(x) &= \pi(u)\frac{\bar{\lambda}^{\abs{x_V}}}{\lambda^{\abs{U\setminus x_U}}}
= \pi(u)\bar{\lambda}^s\lambda^{-s-\Delta(x)},
\end{align}
which is bounded from above by
\begin{align}
\pi(u)\bar{\lambda}^s\lambda^{-s-\Delta(s)} 
&= \pi(u)\lambda^{- \Delta(s) + \alpha s + \smallo(1)}
\end{align}
as $\lambda\to\infty$. We call $\Delta(A)$ and $\Delta(x)$ the \emph{isoperimetric cost} of 
$A$ and $x$. The \emph{(bipartite) isoperimetric problem} asks for the sets $A$ of fixed 
cardinality that minimise the cost $\Delta(A)$. We say that $A$ is (isoperimetrically) 
\emph{optimal} if $\Delta(A)=\Delta(\abs{A})$. More generally, we say that $A$ is 
\emph{$\varepsilon$-optimal} when $\Delta(A)\leq\Delta(\abs{A})+\varepsilon$. Similarly, 
we call a configuration $x$ \emph{$\varepsilon$-optimal} when $\Delta(x)\leq\Delta(\abs{x_V})
+\varepsilon$.

Let us also introduce some terminology to describe evolutions of subsets of $V$. A sequence 
of subsets $A_0,A_1,\ldots,A_n\subseteq V$ is called a \emph{progression} from $A_0$ to 
$A_n$ if $\abs{A_i\triangle A_{i+1}}=1$ for each $0\leq i<n$. A progression $A_0,A_1,\ldots,A_n$ 
is \emph{nested} if $A_0\subseteq A_1\subseteq\cdots\subseteq A_n$ and \emph{isoperimetric} 
if $A_i$ is isoperimetrically optimal for each $0\leq i\leq n$. A nested isoperimetric progression 
from $A_0=\varnothing$ to $A_n$ is associated with a sequence $a_1,a_2,\ldots,a_n$ of distinct 
elements in $V$ with $A_k\isdef\{a_1,a_2,\ldots,a_k\}$. We call such a sequence an \emph{isoperimetric 
numbering} of (some) elements of $V$.

The relevance of the isoperimetric problem will be further clarified in the following sections.
For now, we mention four non-trivial examples of graphs for which we know (partial) solutions 
for the isoperimetric problem.

\begin{example}[{\bf Even torus}]
\label{exp:isoperimetric:torus}
Rather than the isoperimetric problem on the torus $\ZZ_m\times\ZZ_n$, we describe the 
solutions of the isoperimetric problem on the infinite lattice $\ZZ\times\ZZ$. These solutions 
would be valid for the torus as long as the sets that we are considering are small enough that 
they cannot wrap around the torus. The solutions are obtained via reduction to the standard 
edge isoperimetric problem whose solutions are well known~\cite{HarHar76,AloCer96}.
The argument for the reduction is given in Section~\ref{sec:isoperimetric:torus}.
	
The lattice $\ZZ\times\ZZ$ with the nearest neighbour edges is bipartite with $U=\{(a,b)\colon\, 
a+b=0\pmod{2}\}$ and $V=\{(a,b)\colon\, a+b=1\pmod{2}\}$. The isoperimetric function 
$s\mapsto\Delta(s)$ on $\ZZ\times\ZZ$ is given by
\begin{align}
\label{eq:hard-core:isoperimetric:case-1}
\Delta(\ell^2 + i) &= 2(\ell+1)
&& \text{for $\ell>0$ and $0< i\leq \ell$,}\\
\label{eq:hard-core:isoperimetric:case-2}
\Delta(\ell(\ell + 1) + j) &= 2(\ell+1) + 1
&& \text{for $\ell\geq 0$ and $0< j\leq \ell+1$,}
\end{align}
and $\Delta(0)=0$,
which can also be written in a concise algebraic form
\begin{align}
\Delta(s) &= \left\lceil 2\sqrt{s}\right\rceil+1
\end{align}
for $s>0$.
The optimal sets $A$ realising $\Delta(\abs{A})$ are the following:
\begin{itemize}
\item 
A set $A\subseteq V$ with $\abs{A}=\ell^2$ is optimal if and only if it consists of a tilted 
square of size $\ell$ (see Fig.~\ref{fig:isoperimetric:lattice:square},
Eq.~\eqref{eq:hard-core:isoperimetric:case-2} and Sec.~\ref{sec:isoperimetric:torus}).
\item 
A set $A\subseteq V$ with $\abs{A}=\ell^2 + i$ with $0< i\leq \ell$ is optimal if and only if 
it consists of a tilted square of size $\ell$ plus a row of $i$ elements along one of the four 
sides of the square (see Fig.~\ref{fig:isoperimetric:lattice:square-plus}, 
Eq.~\eqref{eq:hard-core:isoperimetric:case-1} and Sec.~\ref{sec:isoperimetric:torus}).
\item 
A set $A\subseteq V$ with $\abs{A}=\ell(\ell+1)+j$ with $0< j\leq \ell$ is optimal if and only if 
it consists of a tilted $\ell\times(\ell+1)$ rectangle plus a row of $j$ elements along one of 
the four sides of the rectangle (see Fig.~\ref{fig:isoperimetric:lattice:quasi-square}, 
Eq.~\eqref{eq:hard-core:isoperimetric:case-2} and Sec.~\ref{sec:isoperimetric:torus}).
\end{itemize}

We point out that some of the optimal sets described above can be generated by suitable 
isoperimetric numberings. Indeed, if we number the elements of $V$ in an spiral fashion 
as in Fig.~\ref{fig:isoperimetric:lattice:numbering}, then every initial segment of this numbering 
is an optimal set. Note, however, that some optimal sets will not be captured by such a 
numbering. For instance, the example in Fig.~\ref{fig:isoperimetric:lattice:stuck} cannot be 
extended to an optimal set one element larger.
\hfill\exampleqed
\end{example}

\begin{figure}[htbp]
\centering
\begin{subfigure}[b]{0.3\textwidth}
\centering
{
\tikzsetfigurename{lattice_isoperimetric_A}
\begin{tikzpicture}[baseline,scale=0.3,>=stealth,shorten >=1]
\DrawLatticeIsoperimetricA
\end{tikzpicture}
}
\caption{$\abs{A}=\ell^2$.}
\label{fig:isoperimetric:lattice:square}
\end{subfigure}
\begin{subfigure}[b]{0.3\textwidth}
\centering
{
\tikzsetfigurename{lattice_isoperimetric_B}
\begin{tikzpicture}[baseline,scale=0.3,>=stealth,shorten >=1]
\DrawLatticeIsoperimetricB
\end{tikzpicture}
}
\caption{$\abs{A}=\ell^2+i$.}
\label{fig:isoperimetric:lattice:square-plus}
\end{subfigure}
\begin{subfigure}[b]{0.3\textwidth}
\centering
{
\tikzsetfigurename{lattice_isoperimetric_C}
\begin{tikzpicture}[baseline,scale=0.3,>=stealth,shorten >=1]
\DrawLatticeIsoperimetricC
\end{tikzpicture}
}
\caption{$\abs{A}=\ell(\ell+1)+j$.}
\label{fig:isoperimetric:lattice:quasi-square}
\end{subfigure}
\caption{Solutions of the bipartite isoperimetric problem on the lattice/torus.}
\label{fig:isoperimetric:lattice}
\end{figure}
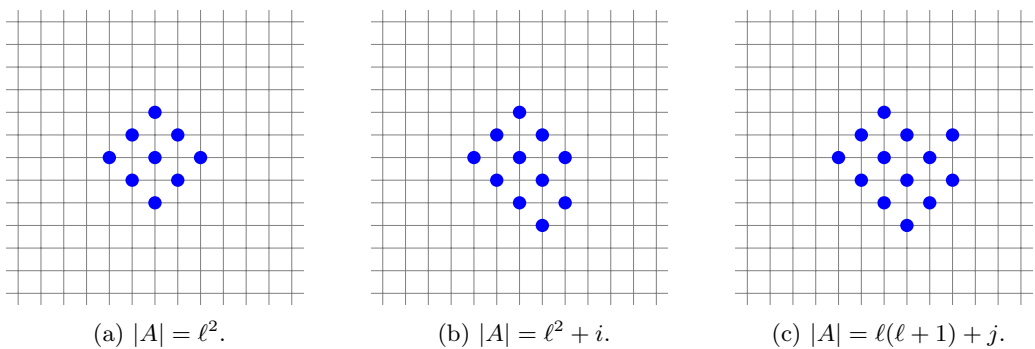

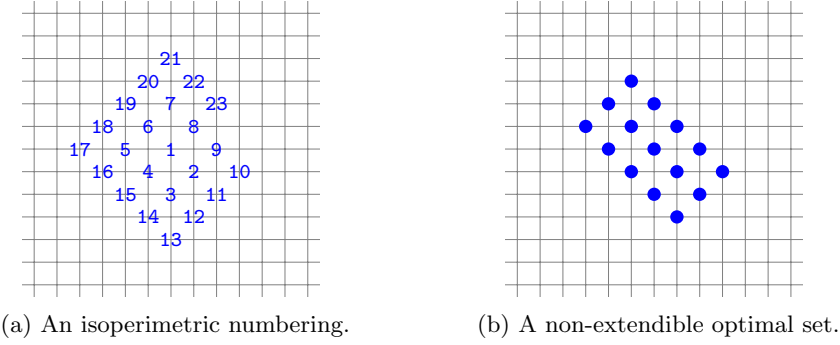
\begin{figure}[htbp]
\centering
\begin{subfigure}[b]{0.40\textwidth}
\centering
{
\tikzsetfigurename{lattice_isoperimetric_numbering}
\begin{tikzpicture}[baseline,scale=0.3,>=stealth,shorten >=1]
\DrawLatticeIsoperimetricNumbering
\end{tikzpicture}
}
\caption{An isoperimetric numbering.}
\label{fig:isoperimetric:lattice:numbering}
\end{subfigure}
\begin{subfigure}[b]{0.40\textwidth}
\centering
{
\tikzsetfigurename{lattice_isoperimetric_stuck}
\begin{tikzpicture}[baseline,scale=0.3,>=stealth,shorten >=1]
\DrawLatticeIsoperimetricStuck
\end{tikzpicture}
}
\caption{A non-extendible optimal set.}
\label{fig:isoperimetric:lattice:stuck}
\end{subfigure}
\caption{The bipartite isoperimetric problem on the lattice/torus via isoperimetric numberings.}
\label{fig:isoperimetric:lattice*}
\end{figure}

\begin{example}[{\bf Doubled torus}]
\label{exp:isoperimetric:doubled-torus}	
As in the previous example, we concentrate on the infinite lattice $\ZZ\times\ZZ$ rather 
than the torus $\ZZ_m\times\ZZ_n$. The solutions for small cardinalities will coincide up 
to translations.
	
Consider the doubled lattice, which is a bipartite graph with parts $U\isdef \ZZ\times\ZZ
\times\{\red\}$ and $V\isdef \ZZ\times\ZZ\times\{\blue\}$. Note that the set of neighbours 
of a set $A\times\{\blue\}\subseteq V$ is $\big(A\cup N(A)\big)\times\{\red\}$, where 
$N(A)$ denotes the neighbourhood of $A$ in the original lattice. In particular, the bipartite 
isoperimetric cost of a set $A\times\{\blue\}$ is simply $\abs{N(A)\setminus A}$, which is 
the size of the \emph{vertex boundary} of $A$ in $\ZZ\times\ZZ$.  This is indeed the case 
for every doubled graph (Observation~\ref{obs:isoperimetric:doubled}). It follows that the 
bipartite isoperimetric problem on the doubled lattice is equivalent to the \emph{vertex 
isoperimetric problem} on the lattice.
	
The vertex isoperimetric problem on the lattice has been addressed by Wang and 
Wang~\cite{WanWan77}, who found optimal sets of every cardinality. Their solutions 
are given by an isoperimetric numbering that identifies an infinite nested family of 
optimal sets. Fig.~\ref{fig:isoperimetric:doubled-lattice:numbering} illustrates an 
isoperimetric numbering similar to but somewhat different from that of Wang and Wang.
	
The isoperimetric function $s\mapsto\Delta(s)$ on the doubled lattice can now be given by
\begin{align}
\label{eq:widom-rowlinson:isoperimetric}
\Delta(\ell^2+(\ell-1)^2+i) &= \begin{cases}
4\ell		        & \text{if $i=0$,} \\
4\ell+1		& \text{if $1\leq i<\ell$,} \\
4\ell+2		& \text{if $\ell\leq i<2\ell$,} \\
4\ell+3		& \text{if $2\ell\leq i<3\ell$,} \\
4\ell+4		& \text{if $3\ell\leq i<4\ell$.}
\end{cases}
\end{align}
and $\Delta(0)=0$.
Note that every positive integer can be written in a unique way as $\ell^2+(\ell-1)^2 + i$ with 
$\ell>0$ and $0\leq i<4\ell$.

Characterising all the optimal sets is more complicated. Vainsencher and Bruckstein~\cite{VaiBru08} 
have obtained a characterisation of the optimal sets with certain cardinalities, namely, those with 
$i\in\{0,\ell-1,2\ell-1,3\ell-1\}$ in~\eqref{eq:widom-rowlinson:isoperimetric}. A characterisation of the 
optimal sets of other cardinalities is still missing. See Section~\ref{sec:isoperimetric:doubled-torus}
for further details and some conjectures.
\hfill\exampleqed
\end{example}

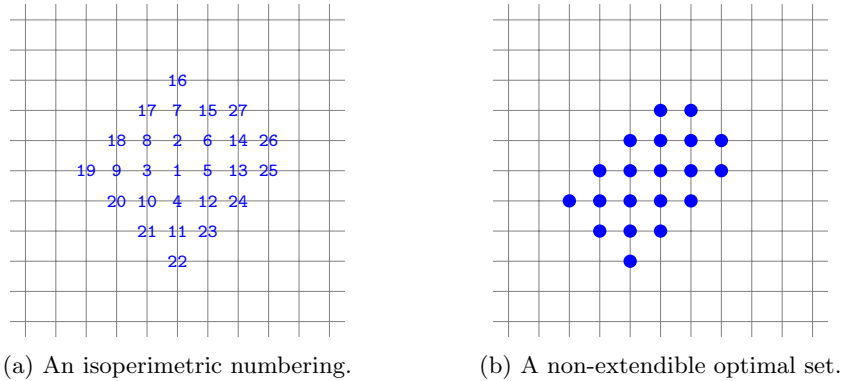
\begin{figure}[htbp]
	\centering
	\begin{subfigure}[b]{0.40\textwidth}
		\centering
		{%
		\tikzsetfigurename{doubled_lattice_isoperimetric_numbering}
		\begin{tikzpicture}[baseline,scale=0.4,>=stealth,shorten >=1]
			\DrawDoubledLatticeIsoperimetricNumbering
		\end{tikzpicture}
		}%
		\caption{An isoperimetric numbering.}
		\label{fig:isoperimetric:doubled-lattice:numbering}
	\end{subfigure}
	\begin{subfigure}[b]{0.40\textwidth}
		\centering
		{%
		\tikzsetfigurename{doubled_lattice_isoperimetric_stuck}
		\begin{tikzpicture}[baseline,scale=0.4,>=stealth,shorten >=1]
			\DrawDoubledLatticeIsoperimetricStuck
		\end{tikzpicture}
		}%
		\caption{A non-extendible optimal set.}
		\label{fig:isoperimetric:doubled-lattice:stuck}
	\end{subfigure}
	\caption{The isoperimetric problem on the doubled lattice/torus via isoperimetric numberings.}
	\label{fig:isoperimetric:doubled-lattice}
\end{figure}

\begin{example}[{\bf Tree-like regular graphs and their doubled graphs}]
\label{exp:isoperimetric:tree-like}	
Consider a $d$-regular graph $G$ in which every cycle has length at least $\ell$, where 
$d\geq 2$ and $\ell$ is large. Such a graph locally looks like a tree, in particular, every 
ball of radius $r<\ell/2$ in $G$ induces a tree.
	
First, suppose that $G$ is bipartite with two parts $U$ and $V$. If $G$ were an infinite 
$d$-regular tree, then every non-empty finite set $A\subseteq V$ would satisfy
$\abs{N(A)}\geq (d-1)\abs{A}+1$ with equality if and only if $A\cup N(A)$ is connected.
This follows by 
induction or by a double counting argument. The same holds for a finite tree-like regular 
graph as long as $\abs{A}<\ell/2$. In particular, $\Delta(s)=(d-2)s+1$ for $0<s<\ell/2$. Any 
sequence $a_1,a_2,\ldots,a_m$ with $m<\ell/2$, satisfying $N(a_i)\cap N(\{a_1,\ldots,a_{i-1}\})
\neq\varnothing$ for $1<i\leq m$, would make an isoperimetric numbering.
	
Next, let us consider the isoperimetric problem on the doubled graph $G^{[2]}$ with $U\isdef 
V(G)\times\{\red\}$ and $V\isdef V(G)\times\{\blue\}$. In this case, we can easily verify that 
every $\varnothing\neq\bar{A}\isdef A\times\{\blue\}\subseteq V$ with $\abs{A}<\ell-1$ satisfies
$\abs{N^{[2]}(\bar{A})}-\abs{\bar{A}}=\abs{N(A)\setminus A}\geq(d-2)\abs{A}+2$ with equality 
if and only if $A$ is connected in~$G$. In particular, $\Delta(s)=(d-2)s+2$ for $0<s<\ell-1$. An 
isoperimetric numbering of length $\ell-2$ is obtained by any sequence $(a_1,\blue),(a_2,\blue),
\ldots,(a_{\ell-2},\blue)\in V$ satisfying the condition that $a_i$ is connected to $\{a_1,\ldots,
a_{i-1}\}$ for each $1<i\leq\ell-2$.
\hfill\exampleqed
\end{example}

\begin{example}[{\bf Hypercube and doubled hypercube}]
\label{exp:isoperimetric:hypercube}
The \emph{$d$-dimensional hypercube} is a graph $H_d$ whose vertices are the binary 
words $w\in\{\symb{0},\symb{1}\}^d$ and in which two vertices $a$ and $b$ are connected 
by an edge if they disagree at exactly one coordinate, i.e., if their Hamming distance is $1$.
The bipartite isoperimetric problem on the doubled graph $H_d^{[2]}$ is equivalent to the 
vertex isoperimetric problem on $H_d$ (Observation~\ref{obs:isoperimetric:doubled}).
	
The hypercube $H_d$ itself is bipartite with $U\isdef\{w\colon\,\norm{w}=0\pmod{2}\}$ and 
$V\isdef\{w: \norm{w}=1\pmod{2}\}$, where $\norm{w}$ denotes the number of $\symb{1}$s 
in $w$. It is interesting to note that the doubled hypercube $H_d^{[2]}$ is isomorphic to the 
$(d+1)$-dimensional hypercube $H_{d+1}$ (Observation~\ref{obs:bipartite:doubled}).
Therefore, the solution of the vertex isoperimetric problem on hypercubes of arbitrary dimension 
also solves the bipartite isoperimetric problem on hypercubes. If $A\subseteq V(H_d)$ is an 
optimal set for the vertex isoperimetric problem on $H_d$, then the set $\hat{A}\isdef\{wa\colon\, 
\text{$w\in A$ and $\norm{wa}=1\pmod{2}$}\}$ is optimal for the bipartite isoperimetric problem 
on $H_{d+1}$ and vice versa.
	
For the vertex isoperimetric problem on $H_d$, Harper~\cite{Har66} provided an isoperimetric 
numbering of the entire graph (see also~Bezrukov~\cite{Bez94}, Harper~\cite{Har04}). This 
numbering is obtained by ordering the elements of $\{\symb{0},\symb{1}\}^d$ first according 
to the number of $\symb{1}$s, and then according to the reverse lexicographic order among the words 
with the same number of $\symb{1}$s. More specifically, the vertices of $H_d$ are numbered 
according to the total order $\unlhd$, where $w\unlhd w'$ when $\norm{w}<\norm{w'}$, or when
$\norm{w}=\norm{w'}$ and there is a $k\in\{1,2,\ldots,d\}$ such that $w_i=w'_i$ for $i<k$ 
and $w_k=\symb{1}$ and $w'_k=\symb{0}$. Bezurukov~\cite{Bez89} has obtained a characterisation 
of the optimal sets of some but not all cardinalities.
	
For every $0\leq r\leq d$, the \emph{Hamming balls}
\begin{align}
B^{(d)}_r(w) 
&\isdef \{w'\colon\, \text{$w$ and $w'$ disagree on at most $r$ coordinates}\}
\end{align}
around vertices $w\in\{\symb{0},\symb{1}\}^d$ are the optimal sets of cardinality $\sum_{i=0}^r
\binom{d}{i}$. In particular, we have $\Delta_{d+1}\big(\sum_{i=0}^r\binom {d}{i}\big)=\binom{d}{r+1}$,
where $\Delta_{d+1}$ denotes the bipartite isoperimetric cost in $H_{d+1}$, or equivalently, 
the vertex isoperimetric cost in $H_d$. In Section~\ref{sec:isoperimetric:hypercube}, we will 
derive a recursive expression for the value of $\Delta_{d+1}(s)$ for general $s$.
\hfill\exampleqed
\end{example}


\section{Further preparation for sophisticated examples}
\label{sec:furtherprep}

Before we proceed with the `sophisticated examples' of 
Section~\ref{sec:hard-core:isoperimetric}, we need some further preparation.


\subsection{Ordering and correlations}
\label{sec:hard-core:ordering}

An advantage of working with bipartite graphs is that the space of valid hard-core configurations 
on a bipartite graph admits a natural partial ordering. The transition kernel of the hard-core 
process is monotone with respect to this ordering and its unique stationary distribution is positively 
associated. Furthermore, two hard-core processes whose parameters satisfy appropriate 
inequalities can be coupled in such a way as to ensure that one always dominates the other.
This ordering has earlier been exploited in the equilibrium setting by van den Berg and 
Steif~\cite{BerSte94}.

For two configurations $x,y\in\spX$, we write $x\sqsubseteq y$ if $x_U\supseteq y_U$ and 
$x_V\subseteq y_V$. The relation $\sqsubseteq$ is a partial order and turns $\spX$ into a 
lattice. The supremum $x\lor y$ and infimum $x\land y$ of two configurations $x,y\in\spX$
are given by
\begin{alignat}{3}
(x\lor y)_V &\isdef x_V \cup y_V
&& \qquad \text{and} \qquad & (x\lor y)_U &\isdef x_U \cap y_U, \nonumber\\
(x\land y)_V &\isdef x_V \cap y_V
&& \qquad \text{and} \qquad & (x\land y)_U &\isdef x_U \cup y_U.
\end{alignat}

For every two finite sets $A,B$ we clearly have
\begin{align}
\abs{A\cup B} + \abs{A\cap B} &= \abs{A} + \abs{B}.
\end{align}
It follows that the stationary distribution of the hard-core process satisfies
\begin{align}
\label{eq:stationary:log-modular}
\pi(x\lor y)\pi(x\land y) &= \pi(x)\pi(y) \qquad\text{for all $x,y\in\spX$.}
\end{align}
By the theorem of Fortuin, Kasteleyn and Ginibre (see e.g.\ Grimmett~\cite[Section 4.2]{Gri10}),
the above condition guarantees that $\pi$ is positively associated, i.e., $\pi(A\cap B)\geq\pi(A)
\pi(B)$ for every two increasing events $A,B\subseteq\spX$. We will, however, use the condition in 
\eqref{eq:stationary:log-modular} directly.

The monotonicity of the transition kernel $K$ can be seen via a direct coupling: given two 
configurations $x,x'\in\spX$ where $x\sqsubseteq x'$, it is easy (e.g.\ via the construction 
described in Section~\ref{sec:intro:model}) to construct two copies of the Markov chain 
$\{X(n)\}_{n\in\NN}$ and $\{X'(n)\}_{n\in\NN}$ with $X(0)=x$ and $X'(0)=x'$ such that almost 
surely $X(n)\sqsubseteq X'(n)$ for all $n\in\NN$. 

Let us mention an extension of the latter observation that we will need in a follow-up paper.
Let $(\lambda_1,\bar{\lambda}_1)$ and $(\lambda_2,\bar{\lambda}_2)$ be two choices for 
the activity parameters of the sites in $U$ and $V$, and assume that $\lambda_1\geq\lambda_2$ 
and $\bar{\lambda}_1\leq\bar{\lambda}_2$. Given $x^{(1)},x^{(2)}\in\spX$ satisfying $x^{(1)}
\sqsubseteq x^{(2)}$, we can construct a coupling $\{(\hat{X}^{(1)}(t),\hat{X}^{(2)}(t))\}_{t\in
[0,\infty)}$ of the continuous-time hard-core processes with parameters $(\lambda_1,
\bar{\lambda}_1)$ and $(\lambda_2,\bar{\lambda}_2)$, respectively, in such a way that 
almost surely $\hat{X}^{(1)}(t)\sqsubseteq\hat{X}^{(2)}(t)$ for all $t\in[0,\infty)$. Namely, we 
use the same clocks $\xi^{\death}_k$ for the death of particles in both systems and we 
couple the birth clocks $\xi^{\birth,1}_k$ and $\xi^{\birth,2}_k$ used for $\hat{X}^{(1)}$ 
and $\hat{X}^{(2)}$ such that $\xi^{\birth,1}_k\supseteq\xi^{\birth,2}_k$ for $k\in U$ and 
$\xi^{\birth,1}_k\subseteq\xi^{\birth,2}_k$ for $k\in V$.


\subsection{Paths and progressions}
\label{sec:hard-core:progressions}

Heuristically, we expect the transition from $u$ to $v$ to happen through the formation 
and growth of a droplet of particles on $V$. Such a growth process can be described 
by a progression from $\varnothing$ to $V$.

Progressions correspond to paths in the configuration space $\spX$ in a natural way.
First, if $\omega\isdef\omega(0)\to\omega(1)\to\cdots\to\omega(n)$ is a path in $\spX$,
then the sequence $A_0,A_1,\ldots,A_m$ obtained after removing repetitions from 
$\omega_V(0),\omega_V(1),\ldots,\omega_V(n)$ is a progression. We call this progression 
the \emph{trace} of $\omega$ on~$V$. Conversely, given a progression $A_0,A_1,\ldots,A_m$, 
we can construct a path $\omega$ in the following fashion (see Fig.~\ref{fig:progression:path}).
The path $\omega$ consists of segments corresponding to transitions $A_{i-1}\to A_i$ for 
$i=1,2,\ldots,m$. At the beginning of the segment corresponding to $A_{i-1}\to A_i$,
the path is at the configuration with particles on $A_{i-1}$ and $U\setminus N(A_{i-1})$.
If $A_{i-1}\subsetneq A_i$, the path then proceeds by removing particles one by one
from the neighbours of the unique site $a_i\in A_i\setminus A_{i-1}$ and then placing a 
particle at $a_i$.  If $A_{i-1}\supsetneq A_i$, the path $\omega$ does the reverse: it first 
removes the particle that is on the unique site $a_i\in A_{i-1}\setminus A_i$ and then 
places particles on the neighbours of $a_i$, one after another. Observe that the trace of 
the path $\omega$ thus obtained is precisely the progression $A_0,A_1,\ldots,A_m$.
In particular, there are indices $0= k_0<k_1<\cdots<k_m= n$ such that $\omega_V(k_i)
=A_i$ and $\omega_U(k_i)=U\setminus N(A_i)$. We call the sequence $\omega(k_0),
\omega(k_1),\ldots,\omega(k_m)$ the \emph{backbone} of~$\omega$.

\begin{figure}[htbp]
\centering
	{
	\tikzsetfigurename{path_assoc_to_progression}
	\begin{tikzpicture}[baseline,scale=0.5,>=stealth,shorten >=1]
		\PathAssocToProgression
		\draw[->] (-2,2) -- node[left] {\footnotesize $-\log\pi$} (-2,3.4);
		\begin{scope}[opacity=0]
		\draw[->] (22,2) -- node[right] {\footnotesize $-\log\pi$} (22,3.4);
		\end{scope}
	\end{tikzpicture}
	}
\caption{%
	The path associated to a typical progression.
	The configurations in the backbone are marked with squares.
}
\label{fig:progression:path}
\end{figure}
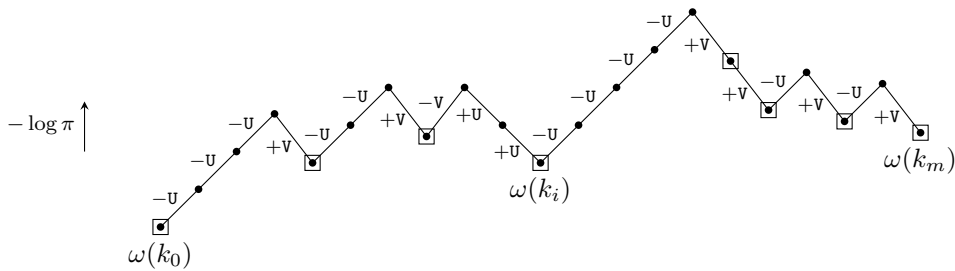

The path associated to a progression is \emph{locally optimal}, in the sense that the critical 
resistance of the segment $\omega(k_{i-1})\to\omega(k_{i-1}+1)\to\cdots\to\omega(k_i)$
corresponding to $A_{i-1}\to A_i$ achieves $\Psi\big(\omega(k_{i-1}),\omega(k_i)\big)$.
When the progression is isoperimetric, the critical resistance of the associated path has 
a sharp upper bound in terms of the isoperimetric function $\Delta(s)$.

\begin{lemma}[{\bf Critical resistance of an isoperimetric progression}]
\label{lem:hard-core:progression:critical-resistance}
	Let $A_0,A_1,\ldots,A_m$ be an isoperimetric progression,
	and set $s_{\min}\isdef\min\{\abs{A_i}: 0\leq i\leq m\}$ and
	$s_{\max}\isdef\max\{\abs{A_i}: 0\leq i\leq m\}$. The critical 
	resistance of the associated path $\omega$ satisfies
	\begin{align}
	\label{eq:progression:path:upper-bound}
		\Psi(\omega) &\leq
			\frac{\gamma}{\pi(u)}\frac{\lambda^{\Delta(s^\dagger)+s^\dagger-1}}{\bar{\lambda}^{s^\dagger-1}}
			= \frac{\gamma}{\pi(u)} \lambda^{\Delta(s^\dagger) - \alpha(s^\dagger-1) + \smallo(1)}
			\qquad\text{as $\lambda\to\infty$,}
	\end{align}
	where $s^\dagger$ is a maximiser of the function $g(s)\isdef\Delta(s)-\alpha(s-1)$
	over the set $\{s_{\min}+1,s_{\min}+2,\ldots,s_{\max}\}$. Furthermore, the equality 
	holds provided the progression is nested and $N(A_1)\not\subseteq N(A_0)$.
\end{lemma}

\noindent
See Appendix~\ref{apx:standard-path} for the proof.

We say that a path $\omega=\omega(0)\to\omega(1)\to\cdots\to\omega(n)$ is \emph{monotone} 
when $\omega(i)\sqsubseteq\omega(i+1)$ for each~$i$, in other words, when $\omega$ consists 
only of transitions of the type $\symb{-U}$ (i.e., removing of a particle from $U$) and $\symb{+V}$ 
(i.e., adding a particle to $V$). Observe that the trace of a monotone path is a nested progression.
Conversely, the path associated to a nested progression is monotone. We call the path associated 
to a nested isoperimetric progression a \emph{standard} path. Clearly, the configurations in the 
backbone of a standard path are isoperimetrically optimal. Moreover, every configuration $x$ on 
a standard path that is not part of the backbone satisfies $\Delta(s)\leq\Delta(x)\leq\Delta(s+1)+1$ 
where $s\isdef\abs{x_V}$. An argument for the following lemma can be found in 
Appendix~\ref{apx:standard-path}.

\begin{lemma}[{\bf Optimality of standard paths}]
\label{lem:hard-core:standard-path:optimality}
	Every standard path is optimal.
\end{lemma}

Assuming the existence of sufficiently long isoperimetric numberings,
Lemmas~\ref{lem:hard-core:progression:critical-resistance}
and~\ref{lem:hard-core:standard-path:optimality} can be combined
to identify the critical resistance between $u$ and $J(u)$.

\begin{proposition}[{\bf Identification of the critical resistance}]
\label{prop:hard-core:critical-resistance:identification}
	Let $\tilde{s}>0$ be an integer such that $\Delta(\tilde{s})\leq\alpha\tilde{s}$,
	and let $s^*$ be a maximiser of the function $g(s)\isdef \Delta(s)-\alpha(s-1)$
	over the set $\{1,\ldots,\tilde{s}\}$. Suppose that an isoperimetric numbering 
	of at least $\tilde{s}$ vertices in $V$ exists. Then the critical resistance between 
	$u$ and $J(u)$ is given by
	\begin{align}
		\Psi\big(u,J(u)\big) 
		&= \frac{\gamma}{\pi(u)}
		\frac{\lambda^{\Delta(s^*)+s^*-1}}{\bar{\lambda}^{s^*-1}}
		= \frac{\gamma}{\pi(u)} \lambda^{\Delta(s^*) - \alpha(s^*-1) + \smallo(1)}
		\qquad\text{as $\lambda\to\infty$.}
	\end{align}
\end{proposition}


\subsection{Absence of traps}
\label{sec:hard-core:no-trap}

In this section we provide a general condition for the absence of traps (i.e., $\pi(x)\Psi(x,J^-(x))
\prec\pi(u)\Psi(u,J(u))$ for every $x\in\spX\setminus\{u,v\}$). The argument provided in 
Appendix~\ref{apx:no-trap} is an adaptation of the one for Glauber dynamics of the Ising 
model (see Bovier and den Hollander~\cite[Section 17.3.1]{BovHol15}), and crucially relies 
on the presence of a partial ordering on the configuration space with respect to which the 
stationary distribution satisfies the FKG condition~\eqref{eq:stationary:log-modular}.
Although the following proposition does not cover all the possible cases, it is simple and 
requires only a simple assumption.

\begin{proposition}[{\bf Absence of traps}]
\label{prop:hard-core:no-trap}
Assume that $\abs{U}<(1+\alpha)\abs{V}$.
Suppose further that, for every $j\in V$, there is a standard path $\omega:u\pathto J(u)$
such that the first particle that $\omega$ places on $V$ is at~$j$. Then every configuration 
$x\notin\{u,v\}$ satisfies $\pi(x)\Psi\big(x,J^-(x)\big)\prec\pi(u)\Psi\big(u,J(u)\big)$ as 
$\lambda\to\infty$.
\end{proposition}

The hypothesis of Proposition~\ref{prop:hard-core:no-trap} can be rewritten
in terms of isoperimetric numberings, hence providing an isoperimetric criterion
for the absence of traps.

\begin{corollary}[{\bf Absence of traps via isoperimetric numberings}]
\label{prop:hard-core:no-trap:isoperimetric}
	Suppose that hypotheses~\eqref{hypothesis:v-is-stable} and~\eqref{hypothesis:numbering:all} are satisfied
	\textup{(}see Sec.~\ref{sec:main}\textup{)}.
	Then every configuration $x\notin\{u,v\}$ 
	satisfies $\pi(x)\Psi\big(x,J^-(x)\big)\prec\pi(u)\Psi\big(u,J(u)\big)$ as $\lambda\to\infty$.
\end{corollary}


\subsection{Critical gate and progressions}
\label{sec:hard-core:critical-gate:progression}

Once we establish the absence of traps, we can use Corollary~\ref{cor:escape:mean:cascade}
to write the mean crossover time $\xExp_u[T_v]$ in terms of the effective resistance
$\effR{u}{J(u)}$.  As we saw in Proposition~\ref{prop:effective-conductance:critical-gate},
a sharp estimate for the effective resistance~$\effR{u}{J(u)}$ can be obtained
if we are able to identify the critical gate between $u$ and $J(u)$.

The purpose of hypothesis~\eqref{hypothesis:critical-set} in Section~\ref{sec:main}
was to describe the critical gate between $u$ and $J(u)$ in terms of
the isoperimetric properties of the underlying graph.
The following proposition clarifies this connection and is verified
in Appendix~\ref{apx:critical-gate:identification}.

\begin{proposition}[{\bf Critical gate in terms of progressions}]
\label{prop:critical-gate:progressions}
	Suppose that hypotheses~\eqref{hypothesis:numbering:single},
	\eqref{hypothesis:uniqueness}
	and~\eqref{hypothesis:critical-set} \textup{(}see Sec.~\ref{sec:main}\textup{)}
	are satisfied, and let $Q$ and $Q^*$ be the described sets of configurations.
	Then, the pair $(Q,Q^*)$ is a critical pair \textup{(}in the sense of 
	Sec.~\ref{sec:sharper}\textup{)} between $u$ and $J(u)$.
\end{proposition}


\subsection{Optimal paths close to the bottleneck}
\label{sec:hard-core:critical-gate}

In order to identify the critical gate between $u$ and $J(u)$, we need an understanding
of the optimal paths from $u$ to $J(u)$ at and around the bottleneck. In this section,
we demonstrate that the configurations close to the bottleneck in every such optimal path
have to be almost isoperimetrically optimal. We state the lemmas in general setting, but 
the reader should keep the even torus (Example~\ref{exp:isoperimetric:torus}) as a guiding 
example.

We assume that there is a standard path between $u$ and $J(u)$, and we let $s^*$ be as 
in Proposition~\ref{prop:hard-core:critical-resistance:identification}. We use the shorthand
\begin{equation}
\dd\Delta(s) \isdef \Delta(s) - \Delta(s^*),
\qquad \dd s \isdef s - s^*,
\end{equation}
for $s\in\NN$. We verify that, near the bottleneck, every basic step of an optimal path is 
through an isoperimetrically optimal configuration.

Let $\omega=\omega(0)\to\omega(1)\to\cdots\to\omega(n)$ be a path on the configuration 
space. We call $\omega(k)$ a \emph{basic step} of $\omega$ when $\omega(k-1)$ or when
$\omega(k+1)$ has less particles than $\omega(k)$. Note that if $\omega(k-1)$ has 
less particles than $\omega(k)$, then we get $r\big(\omega(k-1),\omega(k)\big) 
= \frac{\gamma}{\pi(\omega(k))}$, and similarly, if $\omega(k+1)$ has less particles than 
$\omega(k)$, then $r\big(\omega(k),\omega(k+1)\big) = \frac{\gamma}{\pi(\omega(k))}$.
Therefore, in either case, the critical resistance of $\omega$ satisfies $\Psi(\omega)
\geq\frac{\gamma}{\pi(\omega(k))}$.

The following three lemmas indicate the isoperimetric optimality of basic configurations
in an optimal path $u\pathto J(u)$ when it passes the bottleneck. The proofs can be 
found in Appendix~\ref{apx:bottleneck}.

\begin{lemma}[{\bf Optimality close the bottleneck}]
\label{lem:hard-core:bipartite:optimal-path:isoperimetric:1}
Let $\omega\colon\,u\pathto J(u)$ be an arbitrary optimal path, and let $x$ be a basic 
configuration in $\omega$ with $s$ particles on $V$. Suppose that $\dd \Delta(s)+\varepsilon\geq
\alpha(\dd s+1)$ for some $\varepsilon\geq 0$. Then $x$ is isoperimetrically $\varepsilon$-optimal.
In particular, $x$ is optimal when $s< s^* + \nicefrac{1}{\alpha}-1$ and $\Delta(s)
\geq\Delta(s^*)$.
\end{lemma}

\begin{lemma}[{\bf Optimality close the bottleneck}]
\label{lem:hard-core:bipartite:optimal-path:isoperimetric:2}
Let $\omega\colon\,u\pathto J(u)$ be an arbitrary optimal path, and let $x$ be the first configuration 
in $\omega$ that has $s+1$ particles on $V$. Suppose that $\Delta(s+1)\geq\Delta(s)$ and
$\dd \Delta(s)+\varepsilon\geq\alpha(\dd s+1)$ for some $\varepsilon\geq 0$. Then $x$ is 
isoperimetrically $\varepsilon$-optimal. In particular, $x$ is optimal when $s< s^* 
+ \nicefrac{1}{\alpha}-1$ and $\Delta(s+1)\geq\Delta(s)\geq\Delta(s^*)$.
\end{lemma}

Let $t^*\isdef \abs{U}-s^*-\Delta(s^*)$ denote the number of particles on $U$ in an isoperimetrically 
optimal configuration that has $s^*$ particles on $V$.

\begin{lemma}[{\bf Optimality close the bottleneck}]
\label{lem:hard-core:bipartite:optimal-path:isoperimetric:3}
Let $\omega\colon\,u\pathto J(u)$ be an arbitrary optimal path and assume that $s^*\geq 2$
and $\Delta(s^*)=\Delta(s^*-1)+\delta$ for some $\delta\geq 0$. Let $\omega(q)$ be a basic configuration in 
$\omega$ with at least $s^*$ particles on $V$. Let $\omega(p)$ \textup{(}with $p<q$\textup{)} be the last basic 
configuration before $\omega(q)$ with less than $s^*-1$ particles on $V$. Then the next basic 
configuration after $\omega(p)$ has $s^*-1$ particles on $V$ and at least $t^*+2$ particles on 
$U$. In particular, it is isoperimetrically $(\delta-1)$-optimal.
\end{lemma}

The next proposition combines the above three lemmas to describe an isoperimetric constraint
on the optimal paths $u\pathto J(u)$, which in some cases will help us identify the critical 
gate. See Fig.~\ref{fig:optimal-path:top-of-the-hill} for an illustration.

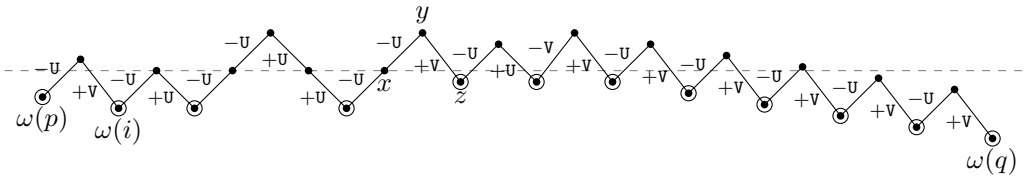
\begin{figure}[htbp]
\centering
{
\tikzsetfigurename{optimal_path_top}
\begin{tikzpicture}[baseline=(x.base),scale=0.5,>=stealth,shorten >=1]
\OptimalPathTopOfTheHill
\end{tikzpicture}
}
\caption{An example of an optimal path near the bottleneck.
In this example, $\abs{z_V}=s^*$ and $\Delta(s^*)=\Delta(s^*-1)+1$.
The circled configurations are isoperimetrically optimal.
There are no basic configurations beyond the dashed line.		
}
\label{fig:optimal-path:top-of-the-hill}
\end{figure}

\begin{proposition}[{\bf Constraint on optimal paths}]
\label{prop:hard-core:bipartite:optimal-path:top-of-the-hill}
Assume that hypotheses~\eqref{hypothesis:numbering:single} and~\eqref{hypothesis:uniqueness}
are satisfied.
Let $\kappa$ be an integer satisfying $0\leq \kappa<\nicefrac{1}{\alpha}$.
\textup{(}For instance, we can take $\kappa\isdef\lceil\nicefrac{1}{\alpha}\rceil-1$.\textup{)}
Suppose that $\Delta(s^*+\kappa)\geq\Delta(s^*+\kappa-1)$, $\Delta(s^*+i)\geq\Delta(s^*)$ 
for $0\leq i<\kappa$, and $\Delta(s^*)=\Delta(s^*-1)+\delta$ for some $\delta\geq 1$. Then 
every optimal path $u\pathto J(u)$ contains a segment $x\xremove[U]y\xadd[V]z$ with the 
following properties:
\begin{enumerate}[label={\rm (\alph*)}]
\item 
$z$ is an isoperimetrically optimal configuration with $\abs{z_V}=s^*$,
\item 
$x$ is an isoperimetrically $\delta$-optimal configuration and
$A\isdef x_V$ is an isoperimetrically $(\delta-1)$-optimal set,
\item
there is an isoperimetric progression $B_0,B_1,\ldots,B_\ell$ with $B_0=z_V$ and
$\abs{B_\ell}=s^*+\kappa$ such that $\abs{B_i}\geq s^*$ for all $i$.
\end{enumerate}
\end{proposition}

\noindent
See Appendix~\ref{apx:bottleneck} for the proof.

As we saw in Proposition~\ref{prop:critical-gate:progressions}, finding two families 
$\family{A},\family{B}$ satisfying hypothesis~\eqref{hypothesis:critical-set} of 
Section~\ref{sec:main} allows us to identify the critical gate between $u$ and $J(u)$.
With the help of Proposition~\ref{prop:hard-core:bipartite:optimal-path:top-of-the-hill},
we can replace hypothesis~\eqref{hypothesis:critical-set} with
hypotheses~\eqref{hypothesis:critical-set-replacement:values}
and~\eqref{hypothesis:critical-set-replacement:existence}
and prove Proposition~\ref{prop:critical-gate:identification}.
See Appendix~\ref{apx:critical-gate:identification} for
the proof of Proposition~\ref{prop:critical-gate:identification}.


\section{Proof of the three metastability theorems}
\label{sec:main-theorems:proof}


\subsection{Mean crossover time: order of magnitude}

\begin{proof}[Proof of Theorem~\ref{thm:mean-crossover:magnitude}]
	As discussed in Section~\ref{sec:intro:model}, we chose to work with the discrete-time 
	version of the Markov chain, so we estimate $\xExp_u[T_v]$ and use the relation 
	$\xExp_u[T_v]=\gamma\xExp_u[\hat{T}_v]$. We apply Corollary~\ref{cor:escape:mean:cascade} 
	with $a\isdef u$ and $Z\isdef\{v\}$ to get
	\begin{align}
		\xExp_u[T_v] &= \pi(u)\effR{u}{J(u)}[1+\smallo(1)] \qquad\text{as $\lambda\to\infty$.}
	\end{align}
	The assumption of absence of traps used in Corollary~\ref{cor:escape:mean:cascade}
	follows from Corollary~\ref{prop:hard-core:no-trap:isoperimetric}
	and hypotheses~\eqref{hypothesis:v-is-stable} and~\eqref{hypothesis:numbering:all}.
	From Proposition~\ref{prop:bounds:communication-height:abstract},
	we know that $\effR{u}{J(u)}\asymp\Psi\big(u,J(u)\big)$ as $\lambda\to\infty$.
	Proposition~\ref{prop:hard-core:critical-resistance:identification}
	together with~\eqref{hypothesis:numbering:all} gives
	\begin{align}
		\Psi\big(u,J(u)\big) 
		&= \frac{\gamma}{\pi(u)}
		\frac{\lambda^{\Delta(s^*)+s^*-1}}{\bar{\lambda}^{s^*-1}}
		= \frac{\gamma}{\pi(u)} \lambda^{\Delta(s^*) - \alpha(s^*-1) + \smallo(1)}
		\qquad\text{as $\lambda\to\infty$.}
	\end{align}
	The claim follows.
\end{proof}


\subsection{Exponential law for crossover time}

\begin{proof}[Proof of Theorem~\ref{thm:crossover:exponential}]
	Apply Corollary~\ref{cor:escape:exponential} with $a\isdef u$ and $Z\isdef\{v\}$.
	The assumption of absence of traps used in Corollary~\ref{cor:escape:exponential}
	follows from Corollary~\ref{prop:hard-core:no-trap:isoperimetric}
	and hypotheses~\eqref{hypothesis:v-is-stable} and~\eqref{hypothesis:numbering:all}.
	To see that the other assumption $\pi(u)\Psi\big(u,J(u)\big)\succ 1$ holds,
	recall that the underlying graph is assumed to be connected.
	Therefore, the first move of every path $\omega:u\pathto J(u)$
	is of the type $\symb{-U}$ (i.e., removing a particle from $U$) and
	\begin{align}
		\Psi\big(u,J(u)\big) &\succeq r\big(u,\omega(1)\big) = \frac{\gamma}{\pi(u)}
			\qquad\text{as $\lambda\to\infty$,}
	\end{align}
	using~\eqref{eq:hard-core:conductance}.
	Hence, $\pi(u)\Psi\big(u,J(u)\big)\succeq\gamma\succ 1$.
\end{proof}


\subsection{Critical gate}

\begin{proof}[Proof of Theorem~\ref{thm:crossover:sharp}]\leavevmode
	\begin{enumerate}[label={\rm (\roman*)}]
		\item As in the proof of Theorem~\ref{thm:mean-crossover:magnitude},
			we estimate $\xExp_u[T_v]$ and use the relation $\xExp_u[T_v]=\gamma
			\xExp_u[\hat{T}_v]$ to get a corresponding estimate for $\xExp_u[\hat{T}_v]$.
			Hypotheses~\eqref{hypothesis:v-is-stable} and~\eqref{hypothesis:numbering:all}
			imply the absence of traps
			via Corollary~\ref{prop:hard-core:no-trap:isoperimetric}, so we can apply 
			Corollary~\ref{cor:escape:mean:cascade} with $a\isdef u$ and $Z\isdef\{v\}$,
			to get
			\begin{align}
				\label{eq:thm:crossover:sharp:proof:I}
				\xExp_u[T_v] &= \pi(u)\effR{u}{J(u)}[1+\smallo(1)] \qquad\text{as $\lambda\to\infty$.}
			\end{align}
			To estimate $\effR{u}{J(u)}$, we identify a critical gate between $\{u\}$ and $J(u)$ 
			and apply Proposition~\ref{prop:effective-conductance:critical-gate}. Since 
			conditions~\eqref{hypothesis:numbering:single}, \eqref{hypothesis:uniqueness}
			and~\eqref{hypothesis:critical-set} are satisfied,
			Proposition~\ref{prop:critical-gate:progressions} implies that
			the sets $Q$ and $Q^*$ form a critical pair between $\{u\}$ and $J(u)$.
			Therefore
			\begin{align}
				\label{eq:thm:crossover:sharp:proof:II}
				\effR{u}{J(u)} &= \frac{1+\smallo(1)}{c(Q,Q^*)},
			\end{align}
			where $c(Q,Q^*)\isdef
			\displaystyle{\mathop{\sum_{x\in Q}\sum_{y\in Q^*}}_{x\link y} c(x,y)}$.
			On the other hand, whenever $x\in Q$ and $y\in Q^*$ and $x\link y$,
			the configuration $y$ is obtained from $x$ by removing a particle from $U$,
			and furthermore, $y_V=A$ and $y_U=U\setminus N(B)$ for some $A\in\family{A}$ 
			and $B\in\family{B}$ with $\abs{B\setminus A}=1$. Therefore, $\abs{x_V}=\abs{A}=s^*-1$
			and $\abs{x_U}=\abs{U\setminus N(B)}+1=\abs{U}-\abs{N(B)}+1=\abs{U}-s^*-\Delta(s^*)+1$.
			Therefore
			\begin{align}
				\label{eq:thm:crossover:sharp:proof:III}
				c(x,y) &= \frac{1}{\gamma}\pi(x) =
					\frac{1}{\gamma}\pi(u)
						\frac{\bar{\lambda}^{\abs{x_V}}}{\lambda^{\abs{U\setminus x_U}}}
				=
					\frac{1}{\gamma}\pi(u)
						\frac{\bar{\lambda}^{s^*-1}}{\lambda^{s^*+\Delta(s^*)-1}} \;.
			\end{align}
			Combining~\eqref{eq:thm:crossover:sharp:proof:I},
			\eqref{eq:thm:crossover:sharp:proof:II} and~\eqref{eq:thm:crossover:sharp:proof:III},
			the result follows.
		\item We apply Proposition~\ref{prop:critical-gate:abstract:choice} with $a\isdef u$ and
			$B\isdef\{v\}$. From Proposition~\ref{prop:critical-gate:progressions}
			and using~\eqref{hypothesis:numbering:single}, \eqref{hypothesis:uniqueness}
			and~\eqref{hypothesis:critical-set},
			we know that $(Q,Q^*)$ is a critical pair between $\{u\}$ and $J(u)$.
			Corollary~\ref{prop:hard-core:no-trap:isoperimetric} and 
			hypotheses~\eqref{hypothesis:v-is-stable} and~\eqref{hypothesis:numbering:all}
			imply the absence of traps.
			Observe that in absence of traps, a critical pair between $\{u\}$ and $J(u)$ is 
			also a critical pair between $\{u\}$ and $\{v\}$. The result now follows after we 
			observe from~\eqref{eq:thm:crossover:sharp:proof:III} that for all pairs $x\in Q$ 
			and $y\in Q^*$ with $x\link y$, the conductance $c(x,y)$ has the same value.
			\qedhere
	\end{enumerate}
\end{proof}


\section{Sophisticated examples: the isoperimetric problem}
\label{sec:isoperimetric}

The bipartite isoperimetric problem introduced in Section~\ref{sec:hard-core:isoperimetric}
belongs to a general class of combinatorial isoperimetric problems. An isoperimetric 
problem on a graph asks for a set of vertices with a given cardinality that has the smallest 
boundary. Depending on how we measure the size of the boundary of a set (called the 
\emph{isoperimetric cost}), we get various versions of the isoperimetric problem. In this 
section, we study the bipartite isoperimetric problem for the examples of graphs considered in 
Section~\ref{sec:hard-core:isoperimetric} by reducing the problem to classical isoperimetric 
problems for which more information is available. In Section~\ref{sec:isoperimetric:edge}, 
we derive the solutions of the bipartite isoperimetric problem on the torus by reducing it to 
the edge isoperimetric problem. In Section~\ref{sec:isoperimetric:vertex}, we study cases 
in which the bipartite isoperimetric problem can be reduced to the vertex isoperimetric problem.


\subsection{Reduction to edge isoperimetry}
\label{sec:isoperimetric:edge}

\subsubsection{Even torus}
\label{sec:isoperimetric:torus}

The aim of this section is to derive the solutions of the bipartite isoperimetric problem
on an even torus, which are described in Example~\ref{exp:isoperimetric:torus}.
For simplicity, we first consider the bipartite isoperimetric problem on the infinite lattice 
$\ZZ\times\ZZ$. We follow the approach of den Hollander, Nardi and Troiani~\cite{HolNarTro11} 
to reduce the problem to the standard edge isoperimetric problem on the lattice. The 
edge isoperimetric problem on the two-dimensional square lattice was solved by Harary 
and Harborth~\cite{HarHar76}, and later independently (and more completely) by Alonso 
and Cerf~\cite{AloCer96}.

Let us start by recalling the edge isoperimetric problem on graphs. Consider a locally finite 
graph~$G$. The \emph{edge boundary} of a set $A\subseteq V(G)$, denoted by $\partial A$,
is the set of edges between $A$ and its complement. The \emph{edge isoperimetric problem} 
on $G$ is the isoperimetric problem in which $\abs{\partial A}$ is counted as the the isoperimetric 
cost of $A$.

Now, let $G$ be bipartite with parts $U$ and $V$, and assume that $G$ is $r$-regular.
For a finite set $A\subseteq V$, we get the identity
\begin{align}
r\abs{N(A)} &= r\abs{A} + \abs{\partial \big(A\cup N(A)\big)}
\end{align}
by counting the edges incident to $N(A)$ in two ways. As a result, we get the 
following convenient representation of the bipartite isoperimetric cost (see 
Fig.~\ref{fig:isoperimetric:lattice:example}).

\begin{observation}[{\bf Isoperimetric cost in regular graphs}]
\label{obs:isoperimetric:regular}
Let $G=(U,V,E)$ be an $r$-regular bipartite graph. Then $\Delta(A)=\frac{1}{r}
\abs{\partial \big(A\cup N(A)\big)}$ for every $A\subseteq V$. In words, the bipartite 
isoperimetric cost of $A$ is the same as the edge isoperimetric cost of $A\cup N(A)$
up to a constant factor.
\end{observation}

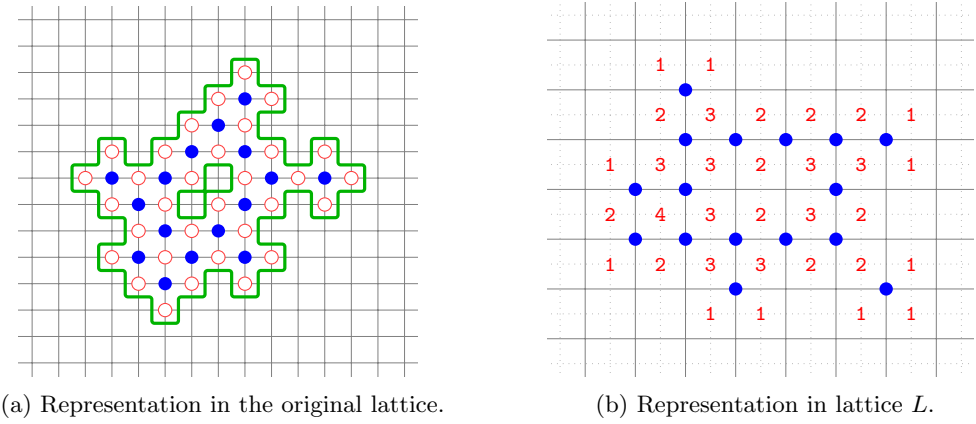
\begin{figure}[htbp]
\centering
\begin{subfigure}[b]{0.45\textwidth}
\centering
{
\tikzsetfigurename{lattice_isoperimetric_example}
\begin{tikzpicture}[baseline,scale=0.35,>=stealth,shorten >=1]
\DrawLatticeIsoperimetricExample
\end{tikzpicture}
}
\caption{Representation in the original lattice.}
\label{fig:isoperimetric:lattice:example}
\end{subfigure}
\begin{subfigure}[b]{0.45\textwidth}
\centering
{
\tikzsetfigurename{lattice_isoperimetric_example_rotated}
\begin{tikzpicture}[baseline,scale=0.33,>=stealth,shorten >=1]
\DrawLatticeIsoperimetricExampleRotated
\end{tikzpicture}
}
\caption{Representation in lattice $L$.}
\label{fig:isoperimetric:lattice:example:rotated}
\end{subfigure}
\caption{An example of a set $A\subseteq V$. The elements of $A$ are represented 
as solid blue circles, and the elements of $N(A)$ as red circles. The isoperimetric cost 
$\Delta(A)$ is the total length of the green contours, which are the dual representation 
of $\partial\big(A\cup N(A)\big)$. Number $k$ represents an element of $N_k(A)$.
}
\label{fig:isoperimetric:lattice:example-and-representation}
\end{figure}

Let us next return to the infinite lattice $\ZZ\times\ZZ$, which is $4$-regular and bipartite, 
with parts $U=\{(a,b)\colon\, a+b=0\pmod{2}\}$ and $V=\{(a,b)\colon\, a+b=1\pmod{2}\}$.
According to Observation~\ref{obs:isoperimetric:regular}, minimising the isoperimetric cost 
$\Delta(A)$ for $A\subseteq V$ amounts to minimising the size of the edge boundary 
$\partial\big(A\cup N(A)\big)$. Let us partition $N(A)$ into four sets $N_1(A)$, $N_2(A)$, 
$N_3(A)$, $N_4(A)$, where $N_k(A)$ consists of those elements in $N(A)$ that have 
precisely $k$ neighbours in $A$. Clearly,
\begin{align}
\label{eq:isoperimetric:lattice:1}
\abs{\partial\big(A\cup N(A)\big)} 
&=3\abs{N_1(A)} + 2\abs{N_2(A)} + \abs{N_3(A)}.
\end{align}

Let us next consider the graph $L$ obtained from the odd sites $V$ by putting an 
edge between $(a,b)$ and $(a',b')$ if and only if $\abs{a'-a}=\abs{b'-b}=1$ (see 
Fig.~\ref{fig:isoperimetric:lattice:example:rotated}). Observe that $L$ is isomorphic to 
the original lattice $\ZZ\times\ZZ$. Divide the set $N_2(A)$ further into two sets
$N_{\symb{1100}}(A)$ and $N_{\symb{1010}}(A)$, according to whether the two
neighbours in $A$ are connected by an edge of $L$ or not, i.e.,
\begin{align}
N_{\symb{1100}} 
&\isdef \left\{ p\in N_2(A)\colon\, \text{$N(p)\cap A=\{i,j\}$ and $(i,j)\in E(L)$} \right\},\nonumber\\
N_{\symb{1010}} 
&\isdef \left\{ p\in N_2(A)\colon\, \text{$N(p)\cap A=\{i,j\}$ and $(i,j)\notin E(L)$} \right\}.
\end{align}
Denoting the edge boundary of $A\subseteq V=V(L)$ in $L$ by $\partial_L A$, we have the identity
\begin{align}
\label{eq:isoperimetric:lattice:2}
2\abs{\partial_L A} 
&= 2\abs{N_1(A)} + 2\abs{N_{\symb{1100}}(A)} + 4\abs{N_{\symb{1010}}(A)} + 2\abs{N_3(A)},
\end{align}
which is obtained by counting, in two different ways, the number of triangles $(e,e',e'')$,
where $e\in\partial_L A$, $e'\in\partial\big(A\cup N(A)\big)$ and $e''\in\partial A$.

Combining\eqref{eq:isoperimetric:lattice:1} and \eqref{eq:isoperimetric:lattice:2}, we get
\begin{align}
\label{eq:isoperimetric:lattice:3}
\Delta(A) 
&= \frac{1}{4}\abs{\partial\big(A\cup N(A)\big)} = \frac{1}{2}\abs{\partial_L A}
+ \frac{1}{4}\Big( \abs{N_1(A)} - 2\abs{N_{\symb{1010}}(A)} - \abs{N_3(A)} \Big)
\end{align}
for every finite $A\subseteq V$.

It can be verified by direct inspection that every non-empty set $A\subseteq V$ that is optimal with 
respect to the edge boundary in $L$ satisfies $\abs{N_1(A)} - 2\abs{N_{\symb{1010}}(A)} 
- \abs{N_3(A)} = 4$. We claim that the same equality holds when $A$ is optimal with respect 
to the bipartite isoperimetric cost $\Delta$.

\begin{lemma}[{\bf Optimality}]
\label{lem:isoperimetric:lattice}
Let $A\subseteq V$ be a non-empty finite set that is optimal with respect to the bipartite isoperimetric 
cost~$\Delta$. Then $N_{\symb{1010}}(A)=\varnothing$ and $\abs{N_1(A)} - \abs{N_3(A)} = 4$.
\end{lemma}

\noindent
A proof of the above lemma can be found in Appendix~\ref{apx:isoperimetric}.

In conclusion, we have the equality
\begin{align}
\label{eq:isoperimetric:lattice:4}
\Delta(A) &= \frac{1}{2}\abs{\partial_L A} + 1
\end{align}
for every non-empty $A\subseteq V$ that is optimal either with respect to the edge boundary in $L$
or with respect to the bipartite isoperimetric cost $\Delta$.
It follows that the solutions of 
the bipartite isoperimetric problem on the lattice $\ZZ\times\ZZ$ coincide with the solutions 
of the edge isoperimetric problem on the lattice $L$.
The edge boundary of an optimal set with $s$ vertices has size $2\lceil 2\sqrt{s}\rceil$ and
the optimal sets in $L$ are those described in Example~\ref{exp:isoperimetric:torus}.
Thus, $\Delta(s)=\lceil 2\sqrt{s}\rceil+1$ for $s>0$ and the optimal sets with respect to $\Delta$ 
are as described in Example~\ref{exp:isoperimetric:torus}.

Finally, we argue that the solutions of the bipartite isoperimetric problem on an even torus 
$\ZZ_m\times\ZZ_n$ are the same (modulo translations) as the solutions for the infinite 
lattice $\ZZ\times\ZZ$ as long as the size of the set is small compared to $m$ and $n$. To 
see why, it is enough to note that if $A$ has less than $\frac{1}{4}\min\{m,n\}$ vertices, then 
it cannot ``sense'' the distinction between $\ZZ_m\times\ZZ_n$ and $\ZZ\times\ZZ$. More 
precisely, let $A$ be an optimal set in $\ZZ_m\times\ZZ_n$ with $\abs{A}<\frac{1}{4}\min\{m,n\}$.
Then, the pre-image of $A$ under the canonical projection from $\ZZ\times\ZZ$ to $\ZZ_m
\times\ZZ_n$ can be partitioned into countably many sets $A'_i$ (for $i\in\ZZ\times\ZZ$) 
such that each $A'_i$ is a translated copy of $A$ and the sets $A'_i\cup N(A'_i)$ are disjoint.
In particular, that $\abs{A'_i}=\abs{A}$ and $\Delta(A'_i)=\Delta(A)$. Conversely, if $A'$ is an 
optimal set in $\ZZ\times\ZZ$ with $\abs{A}<\frac{1}{4}\min\{m,n\}$, then $A'\cup N(A')$ is 
connected (see the proof of Lemma~\ref{lem:isoperimetric:lattice}). It is easy to see that the 
canonical projection of $\ZZ\times\ZZ$ onto $\ZZ_m\times\ZZ_n$ maps every connected set 
with less than $\min\{m,n\}$ elements injectively. In particular, if $A$ denotes the projection 
of $A'$, then $A$ is simply a translated copy of $A$ and we have $\abs{A}=\abs{A'}$ and 
$\Delta(A)=\Delta(A')$.


\subsection{Reduction to vertex isoperimetry}
\label{sec:isoperimetric:vertex}

In the \emph{vertex isoperimetric problem}, the size of the boundary of a set $A$ is 
measured as $\abs{N(A)\setminus A}$. The bipartite isoperimetric problem on a 
doubled graph $G^{[2]}$ is equivalent to the vertex isoperimetric problem on the 
original graph $G$.

\begin{observation}[{\bf Reduction to vertex isoperimetry}]
\label{obs:isoperimetric:doubled}
Let $G$ be a locally finite graph and let $G^{[2]}$ be its doubled version. Let $U\isdef 
V(G)\times\{\red\}$ and $V\isdef V(G)\times\{\blue\}$ be the two parts of $G^{[2]}$. Then 
$\Delta(A\times\{\blue\})=\abs{N_G(A)\setminus A}$ for every $A\subseteq V(G)$, i.e., 
the bipartite isoperimetric cost of $A\times\{\blue\}$ in $G^{[2]}$ coincides with the vertex 
isoperimetric cost of $A$ in $G$. 
\end{observation}

\begin{observation}[{\bf Doubled version of bipartite graphs}]
\label{obs:bipartite:doubled}
The doubled version of a bipartite graph $G$ is isomorphic to the Cartesian product 
$G\times\ZZ_2$, where $\ZZ_2$ is the graph with two vertices and an edge between them.
\end{observation}

\noindent
The doubled version of a non-bipartite graph is similar, except that it has a ``M\"obius twist'' 
along each odd cycle (see Fig.~\ref{fig:graph:doubled-graph} and Fig.~\ref{fig:graph:bipartite:ladder}).

\subsubsection{Doubled torus}
\label{sec:isoperimetric:doubled-torus}

According to Observation~\ref{obs:isoperimetric:doubled}, the bipartite isoperimetric problem 
on a doubled torus is equivalent to the vertex isoperimetric problem on a torus. Since we will 
be concerned only with sets that are small in comparison with the dimensions of the torus, 
we may consider the infinite lattice $\ZZ\times\ZZ$ instead.  As mentioned in 
Example~\ref{exp:isoperimetric:doubled-torus}, Wang and Wang~\cite{WanWan77} have 
produced an isoperimetric numbering for the vertex isoperimetric problem on $\ZZ\times\ZZ$.
Vainsencher and Bruckstein~\cite{VaiBru08} have provided a characterisation of the optimal 
sets of certain \emph{critical} cardinalities. A complete characterisation of the optimal sets for 
the remaining cardinalities is beyond the scope of this paper. In this section, we propose a 
conjecture that, if true, will allow us to obtain sharp asymptotics for the metastable transition
in the Widom-Rowlinson model on a torus.

Every positive integer $s$ has a unique representation $s=\ell^2+(\ell-1)^2+r$
where $\ell> 0$ and $0\leq r<4\ell$. Note that $4\ell=(\ell-1) + \ell + \ell + (\ell+1)$.
We call a number $s=\ell^2+(\ell-1)^2+r$ \emph{critical} if $r\in\{0,\ell-1,2\ell-1,3\ell-1\}$.
Observe from~\eqref{eq:widom-rowlinson:isoperimetric} that $\Delta(s)$ is non-decreasing 
with $\Delta(s+1)>\Delta(s)$ if and only if $s$ is a critical cardinality. It follows that an optimal 
set $A$ has a critical cardinality if and only if it is also \emph{co-optimal}, meaning that
it has maximum cardinality among all sets $B$ with $\Delta(B)=\Delta(A)$.
A set that is both optimal and co-optimal is called \emph{Pareto optimal}.

For $A\subseteq\ZZ\times\ZZ$ and $k\geq 0$, let $N^k(A)$ denote the set of sites within 
graph distance $k$ from $A$, i.e., the ball of radius $k$ around $A$.
Vainsencher and Bruckstein~\cite{VaiBru08} have shown that a non-empty set is Pareto 
optimal if and only if it has the form $N^k(S)$ for $k\geq 0$ and a set $S$ that is obtained 
by translation and rotation from one of the basic forms in Fig.~\ref{fig:isoperimetric:doubled-torus:seed}.
We call the set $S$ the \emph{seed} of $N^k(S)$.

\begin{figure}[htbp]
	\centering
	\begin{subfigure}[b]{\textwidth}
		\centering
		\begin{tabular}{ccccc}
			\begin{minipage}[b]{0.15\textwidth}
				\centering 
				{
				\tikzsetfigurename{doubled_lattice_isoperimetric_seed_A}
				\begin{tikzpicture}[baseline,scale=0.35,>=stealth,shorten >=1]
					\DrawDoubledLatticeIsoperimetricSeedA
				\end{tikzpicture}
				}
			\end{minipage} &
			\begin{minipage}[b]{0.15\textwidth}
				\centering 
				{
				\tikzsetfigurename{doubled_lattice_isoperimetric_seed_E}
				\begin{tikzpicture}[baseline,scale=0.35,>=stealth,shorten >=1]
					\DrawDoubledLatticeIsoperimetricSeedE
				\end{tikzpicture}
				}
			\end{minipage} &
			\begin{minipage}[b]{0.15\textwidth}
				\centering 
				{
				\tikzsetfigurename{doubled_lattice_isoperimetric_seed_B}
				\begin{tikzpicture}[baseline,scale=0.35,>=stealth,shorten >=1]
					\DrawDoubledLatticeIsoperimetricSeedB
				\end{tikzpicture}
				}
			\end{minipage} &
			\begin{minipage}[b]{0.15\textwidth}
				\centering 
				{
				\tikzsetfigurename{doubled_lattice_isoperimetric_seed_C}
				\begin{tikzpicture}[baseline,scale=0.35,>=stealth,shorten >=1]
					\DrawDoubledLatticeIsoperimetricSeedC
				\end{tikzpicture}
				}
			\end{minipage} &
			\begin{minipage}[b]{0.15\textwidth}
				\centering 
				{
				\tikzsetfigurename{doubled_lattice_isoperimetric_seed_D}
				\begin{tikzpicture}[baseline,scale=0.35,>=stealth,shorten >=1]
					\DrawDoubledLatticeIsoperimetricSeedD
				\end{tikzpicture}
				}
			\end{minipage}
			\medskip \\
			I & II & IIIa & IIIb & IV
		\end{tabular}
		\caption{The seeds generating the Pareto optimal sets (up to rotations and translations).}
		\label{fig:isoperimetric:doubled-torus:seed}
	\end{subfigure}
	
	\bigskip
	
	\begin{subfigure}[b]{\textwidth}
		\centering
		\begin{tabular}{ccccc}
			\begin{minipage}[b]{0.15\textwidth}
				\centering 
				{
				\tikzsetfigurename{doubled_lattice_isoperimetric_seed_A_example}
				\begin{tikzpicture}[baseline,scale=0.2,>=stealth,shorten >=1]
					\draw (-0.5,-0.5) rectangle (0.5,0.5);
					\DrawDoubledLatticeIsoperimetricSeedAExample
				\end{tikzpicture}
				}
			\end{minipage} &
			\begin{minipage}[b]{0.15\textwidth}
				\centering 
				{
				\tikzsetfigurename{doubled_lattice_isoperimetric_seed_E_example}
				\begin{tikzpicture}[baseline,scale=0.2,>=stealth,shorten >=1]
					\draw (-0.5,-0.5) rectangle (0.5,0.5);
					\draw (-0.5,0.5) rectangle (0.5,1.5);
					\draw (0.5,-0.5) rectangle (1.5,0.5);
					\draw (-0.5,-1.5) rectangle (0.5,-0.5);
					\draw (-1.5,-0.5) rectangle (-0.5,0.5);
					\draw (0.5,0.5) rectangle (1.5,1.5);
					\DrawDoubledLatticeIsoperimetricSeedEExample
				\end{tikzpicture}
				}
			\end{minipage} &
			\begin{minipage}[b]{0.15\textwidth}
				\centering 
				{
				\tikzsetfigurename{doubled_lattice_isoperimetric_seed_B_example}
				\begin{tikzpicture}[baseline,scale=0.2,>=stealth,shorten >=1]
					\draw (-0.5,-0.5) rectangle (0.5,0.5);
					\draw (-0.5,0.5) rectangle (0.5,1.5);
					\DrawDoubledLatticeIsoperimetricSeedBExample
				\end{tikzpicture}
				}
			\end{minipage} &
			\begin{minipage}[b]{0.15\textwidth}
				\centering 
				{
				\tikzsetfigurename{doubled_lattice_isoperimetric_seed_C_example}
				\begin{tikzpicture}[baseline,scale=0.2,>=stealth,shorten >=1]
					\draw (-0.5,-0.5) rectangle (0.5,0.5);
					\draw (0.5,0.5) rectangle (1.5,1.5);
					\DrawDoubledLatticeIsoperimetricSeedCExample
				\end{tikzpicture}
				}
			\end{minipage} &
			\begin{minipage}[b]{0.15\textwidth}
				\centering 
				{
				\tikzsetfigurename{doubled_lattice_isoperimetric_seed_D_example}
				\begin{tikzpicture}[baseline,scale=0.2,>=stealth,shorten >=1]
					\draw (-0.5,-0.5) rectangle (0.5,0.5);
					\draw (-0.5,0.5) rectangle (0.5,1.5);
					\draw (0.5,-0.5) rectangle (1.5,0.5);
					\DrawDoubledLatticeIsoperimetricSeedDExample
				\end{tikzpicture}
				}
			\end{minipage}%
			\medskip \\
			$N^{\ell-1}(S)$ & $N^{\ell-2}(S)$ & $N^{\ell-1}(S)$ & $N^{\ell-1}(S)$ & $N^{\ell-1}(S)$
			\\[-1ex]
			& & \multicolumn{2}{c}{\upbracefill} & \\
			{\scriptsize $\mathclap{\ell^2 + (\ell-1)^2}$} &
			{\scriptsize $\mathclap{\ell^2 + (\ell-1)^2 + \ell-1}$} &
			\multicolumn{2}{c}{\scriptsize $\mathclap{\ell^2 + (\ell-1)^2 + 2\ell-1}$} &
			{\scriptsize $\mathclap{\ell^2 + (\ell-1)^2 + 3\ell-1}$}
		\end{tabular}
		\caption{Examples of sets generated from the seeds and their cardinalities.}
		\label{fig:isoperimetric:doubled-torus:inflated}
	\end{subfigure}
	\caption{%
		Every Pareto optimal set (i.e., an optimal set with a crtical cardinality)
		on the lattice is generated by a seed.
	}
	\label{fig:isoperimetric:doubled-torus:pareto}
\end{figure}
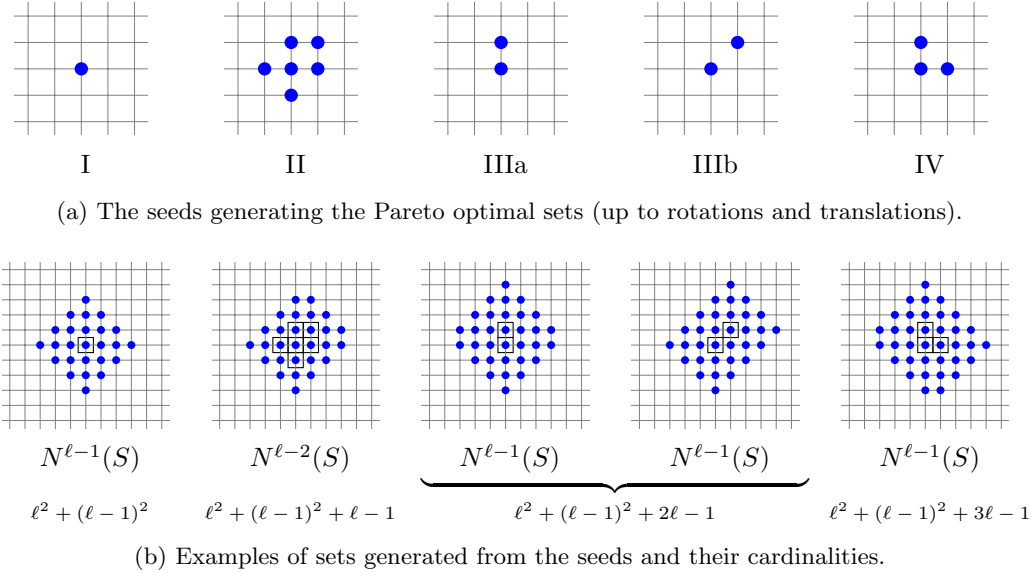

Pareto optimal sets of consecutive types can be connected via nested isoperimetric progressions.

\pagebreak[2]

\begin{observation}[{\bf Existence of connecting progressions}]\leavevmode
\label{obs:isoperimetric:doubled-torus:progressions:existence}
	\begin{enumerate}[label={\rm (\alph*)}]
		\item Let $S$ and $S'$ be seeds of type I and II
			of Fig.~\ref{fig:isoperimetric:doubled-torus:seed}, respectively,
			and suppose that $N(S)\subseteq S'$.  Then, for every $\ell\geq 2$,
			there is a nested isoperimetric progression from $N^{\ell-1}(S)$
			to~$N^{\ell-2}(S')$.
		\item Let $S$ and $S'$ be seeds of type II and III
			of Fig.~\ref{fig:isoperimetric:doubled-torus:seed}, respectively,
			and suppose that $S\subseteq N(S')$.  Then, for every $\ell\geq 2$,
			there is a nested isoperimetric progression from $N^{\ell-2}(S)$
			to~$N^{\ell-1}(S')$.
		\item Let $S$ and $S'$ be seeds of type III and IV
			of Fig.~\ref{fig:isoperimetric:doubled-torus:seed}, respectively,
			and suppose that $S\subseteq S'$.  Then, for every $\ell\geq 1$,
			there is a nested isoperimetric progression from $N^{\ell-1}(S)$
			to~$N^{\ell-1}(S')$.
		\item Let $S$ and $S'$ be seeds of type IV and I
			of Fig.~\ref{fig:isoperimetric:doubled-torus:seed}, respectively,
			and suppose that $S\subseteq N(S')$.  Then, for every $\ell\geq 1$,
			there is a nested isoperimetric progression from $N^{\ell-1}(S)$
			to~$N^{\ell}(S')$.
	\end{enumerate}	
\end{observation}
As an immediate consequence, we find that Pareto optimal sets are achieved
via isoperimetric numberings.
\begin{observation}[{\bf Pareto optimal sets via optimal numberings}]
\label{obs:isoperimetric:doubled-torus:numberings}
	Every Pareto optimal set is of the form $A=\{a_1,a_2,\ldots,a_n\}$
	for some unbounded isoperimetric numbering $a_1,a_2,\ldots$.
\end{observation}

In order to identify the critical gate for the Widom-Rowlinson model
on a torus, we will also need some information about all isoperimetric progressions
connecting Pareto optimal sets of consecutive types.
This requires a better understanding of the optimal sets with non-critical cardinalities,
which we do not have.
Nonetheless, we make the following conjecture.

\begin{conjecture}[{\bf Property of connecting progressions}]\leavevmode
\label{conj:isoperimetric:doubled-torus:progressions:property}
	\begin{enumerate}[label={\rm (\alph*)}]
		\item Let $B_0,B_1,\ldots,B_n$ be an isoperimetric progression
			with $\abs{B_0}=\ell^2+(\ell-1)^2+\ell-1$ and $\abs{B_n}=\ell^2+(\ell-1)^2+2\ell-1$
			and $\abs{B_0}<\abs{B_i}<\abs{B_n}$ for $0<i<n$.
			Let $S_0$ be the seed of $B_0$ and $S_n$ the seed of $B_n$,
			so that $B_0=N^{\ell-2}(S_0)$ and $B_n=N^{\ell-1}(S_n)$.
			Then, $S_0\subseteq N(S_n)$ and $B_0\subseteq B_1\subseteq B_n$.
		\item Let $B_0,B_1,\ldots,B_n$ be an isoperimetric progression
			with $\abs{B_0}=\ell^2+(\ell-1)^2+3\ell-1$ and $\abs{B_n}=(\ell+1)^2+\ell^2$
			and $\abs{B_0}<\abs{B_i}<\abs{B_n}$ for $0<i<n$.
			Let $S_0$ be the seed of $B_0$ and $S_n$ the seed of $B_n$,
			so that $B_0=N^{\ell-1}(S_0)$ and $B_n=N^{\ell}(S_n)$.
			Then, $S_0\subseteq N(S_n)$ and $B_0\subseteq B_1\subseteq B_n$.
	\end{enumerate}
\end{conjecture}


\subsubsection{Hypercube}
\label{sec:isoperimetric:hypercube}
According to Observations~\ref{obs:isoperimetric:doubled} and~\ref{obs:bipartite:doubled},
the bipartite isoperimetric problem on the $(d+1)$-dimensional hypercube $H_{d+1}$ is equivalent
to the vertex isoperimetric problem on the $d$-dimensional hypercube $H_d$. In this section, 
we present a recursive expression for the vertex isoperimetric function on the hypercube.

As mentioned in Example~\ref{exp:isoperimetric:hypercube}, from Harper's isoperimetric 
numbering~\cite{Har66}, we can immediately see that $\Delta_{d+1}\big(\sum_{i=0}^{r-1}
\binom {d}{i}\big)=\binom{d}{r}$ for $0<r\leq d$.  (Recall: $\Delta_{d+1}$ is the vertex 
isoperimetric function of the $d$-dimensional hypercube $H_d$.)  More generally, we can 
use the numbering to obtain a recursive expression for $\Delta_{d+1}$.

\begin{proposition}[{\bf Isoperimetric function of the hypercube}]
\label{prop:isoperimetric:hypercube:recursion}
	For $0\leq k\leq\binom{d}{r}$, we can write
	\begin{align}
		\Delta_{d+1}\left(\sum_{i=0}^{r-1}\binom {d}{i}+k\right)
			&= \binom{d}{r}+\psi_d(r,k)-k \;,
	\end{align}
	where $\psi_d(r,k)$ satisfies the recursion
	\begin{align}
		\label{eq:isoperimetric:hypercube:recursion}
		\psi_d(r,k) &=
			\begin{cases}
				\psi_{d-1}(r-1,k) & \text{if $0<r<d$ and $0< k\leq\binom{d-1}{r-1}$,} \\[1ex]
				\tbinom{d-1}{r} + \psi_{d-1}\big(r,k-\binom{d-1}{r-1}\big)
					& \text{if $0<r<d$ and $\binom{d-1}{r-1}< k\leq\binom{d}{r}$,} \\[1ex]
				0 & \text{otherwise.}
			\end{cases}
	\end{align}
	\end{proposition}
\noindent 
The proof can be found in Appendix~\ref{apx:isoperimetric}.


\section{Sophisticated examples: key results}
\label{sec:main-examples:proof}

After having collected in Section~\ref{sec:furtherprep}
the relevant tools, we are now ready to apply our results to the `sophisticated examples' 
in Section~\ref{sec:hard-core:isoperimetric}: torus, doubled torus, tree-like graphs, hypercube. 
In the case of the torus where a complete solution of the isoperimetric problem is known,
we obtain a complete picture of the metastable transition from $u$ to $v$.
In the case of the doubled torus, the complete picture relies on the validity of
Conjecture~\ref{conj:isoperimetric:doubled-torus:progressions:property}.
In other cases we still obtain an incomplete picture.


\subsection{Hard-core on an even torus}
\label{sec:torus:metastability}
In this section, we combine our results to give a description of the metastable 
transition of the hard-core dynamics on an even torus $\ZZ_m\times\ZZ_n$. We assume 
$0<\alpha<1$, $\nicefrac{2}{\alpha}\notin\ZZ$ and $m,n\gg\nicefrac{1}{\alpha}$.
Putting together the result in the paper, we are able to give a complete picture
of the transition from $u$ to $v$: exponential distribution for the crossover time,
sharp estimate for the expected crossover time, and a detailed description
of the critical droplet.

As discussed in 
Example~\ref{exp:isoperimetric:torus} (and proved in Section~\ref{sec:isoperimetric:torus}),
the isoperimetric function of $\ZZ_m\times\ZZ_n$ is given by $\Delta(s)=\left\lceil 2
\sqrt{s}\right\rceil+1$ as long as $s\ll m,n$. The proof of the following lemma can be 
found in Appendix~\ref{apx:critical-size}.

\begin{lemma}[{\bf Critical size: torus}]
\label{lem:torus:critical-size}
	Suppose $0<\alpha<1$ and $\nicefrac{2}{\alpha}\notin\ZZ$,
	and let $\Delta(s)$ be the isoperimetric function of
	a torus $\ZZ_m\times\ZZ_n$ with even $m,n\gg \nicefrac{1}{\alpha}$.
	Then, the function $g(s)=\Delta(s)-\alpha(s-1)$ has a unique maximum on $\NN\isdef\{0,1,\ldots\}$
	at $s^*\isdef\ell^*(\ell^*-1)+1$, where $\ell^*\isdef\lceil\nicefrac{1}{\alpha}\rceil$.
\end{lemma}

Finding the exact value of resettling size $\tilde{s}$ (i.e., the smallest $s$ for which $\Delta(s)\leq\alpha s$) 
is not necessary. It is sufficient to note that $\tilde{s}$ exists (the inequality is achieved for instance for 
$s>\nicefrac{8}{\alpha^2}$) and is independent of $m$ and $n$ (as long as $m,n\gg \nicefrac{1}{\alpha}$).

Hypothesis~\eqref{hypothesis:v-is-stable} is clearly satisfied.
The existence of isoperimetric numberings of length at least~$\tilde{s}$ was demonstrated in 
Example~\ref{exp:isoperimetric:torus} (as long as $\tilde{s}\ll m,n$). Hence 
hypothesis~\eqref{hypothesis:numbering:single} and (by translation symmetry) 
hypothesis~\eqref{hypothesis:numbering:all} are both satisfied. Therefore, 
Theorems~\ref{thm:mean-crossover:magnitude} and~\ref{thm:crossover:exponential}
establish the asymptotic exponentiality of the crossover time and provide the estimate
\begin{align}
	\xExp_u[\hat{T}_v] &\asymp
		\frac{\lambda^{\ell^*(\ell^*+1)+1}}{\bar{\lambda}^{\ell^*(\ell^*-1)}}
		= \lambda^{2\ell^*+1-\alpha\ell^*(\ell^*-1) + \smallo(1)}
		\qquad\text{as $\lambda\to\infty$,}
\end{align}
for its mean, where $\ell^*\isdef\lceil\nicefrac{1}{\alpha}\rceil$.

A more accurate estimate on the mean crossover time as well as a description of 
the critical droplet is provided by Theorem~\ref{thm:crossover:sharp}, which relies 
on hypotheses~\eqref{hypothesis:uniqueness} and~\eqref{hypothesis:critical-set}.
Hypothesis~\eqref{hypothesis:uniqueness} is already verified in Lemma~\ref{lem:torus:critical-size}.
Proposition~\ref{prop:critical-gate:identification} reduces the verification
of~\eqref{hypothesis:critical-set} to the verification of simpler 
conditions~\eqref{hypothesis:critical-set-replacement:values}
and~\eqref{hypothesis:critical-set-replacement:existence}.
Choose $\kappa\isdef\lceil\nicefrac{1}{\alpha}\rceil-1=\ell^*-1$.
Conditions~\eqref{hypothesis:critical-set-replacement:values:A} 
and~\eqref{hypothesis:critical-set-replacement:values:B}
follow from the monotonicity of $\Delta(s)=\big\lceil 2\sqrt{s}\big\rceil+1$,
and~\eqref{hypothesis:critical-set-replacement:values:C} is evident via direct 
calculation $\Delta(s^*)=2\ell^*+1$ and $\Delta(s^*-1)=2\ell^*$.
In order to verify~\eqref{hypothesis:critical-set-replacement:existence},
observe that $s^*-1=\ell^*(\ell^*-1)$ and $s^*+\kappa=(\ell^*)^2$.
From the characterisation of the isoperimetrically optimal sets in
Example~\ref{exp:isoperimetric:torus}, we find that
\begin{itemize}
	\item $\family{A}$ consists precisely of tilted $(\ell^*-1)\times\ell^*$
		rectangles, and
	\item $\family{C}$ consists precisely of tilted $\ell^*\times\ell^*$
		squares
\end{itemize}
of elements of $V$.
Conditions~\eqref{hypothesis:critical-set-replacement:before-A}
and~\eqref{hypothesis:critical-set-replacement:after-C} follow immediately
from the existence of isoperimetric numberings of length at least $\tilde{s}$
and symmetry.

In order to identify the family $\family{B}$, recall from Example~\ref{exp:isoperimetric:torus}
that each isoperimetrically optimal set~$B$ with $\abs{B}=s^*=(\ell^*-1)\ell^*+1$ consists 
of an element of $A\in\family{A}$ (i.e., an $(\ell^*-1)\times\ell^*$ tilted rectangle) and an 
extra site $b$ along one of the four sides of the rectangle
(see~Fig.~\ref{fig:isoperimetric:lattice:quasi-square}).
Observe that if $b$ is along a longer edge of $A$, then
$B$ can be extended via a nested isoperimetric progression to an element of~$\family{C}$
(i.e., an $\ell^*\times\ell^*$ tilted rectangle), whereas if $b$ is along a shorter edge of $A$, 
then every isoperimetric progression from $B$ to $\family{C}$ must pass through $\family{A}$.
Therefore,
\begin{itemize}
	\item $\family{B}$ consists precisely of tilted $(\ell^*-1)\times\ell^*$
		rectangles plus an extra element along one of the two longer sides of
		the rectangle.
\end{itemize}
A typical transition through the critical gate $[Q,Q^*]$ is
depicted in Figure~\ref{fig:gate:torus}.

\begin{figure}[htbp]
	\centering
	\begin{tabular}{ccc}
		\raisebox{-\height/2+2pt}{%
		{
		\tikzsetfigurename{torus_gate_Q}
		\begin{tikzpicture}[scale=0.175,>=stealth,shorten >=1]
			\DrawTorusGate
			\node[red particle] at (11,-5) {};
		\end{tikzpicture}
		}
		}
		& $\longrightarrow$ &
		\raisebox{-\height/2+2pt}{%
		{
		\tikzsetfigurename{torus_gate_Qstar}
		\begin{tikzpicture}[scale=0.175,>=stealth,shorten >=1]
			\DrawTorusGate
		\end{tikzpicture}
		}
		}		
	\end{tabular}
	\caption{A typical transition through the critical gate for the torus.
		The critical length $\ell^*$ is assumed to be~$6$. The `hole' 
		is along one of the two long edges of the rectangle. Once a 
		two-site `hole' is produced, with probability close to $1$ a (blue) 
		particle appears very quickly in the opened-up space.
	}
	\label{fig:gate:torus}
\end{figure}
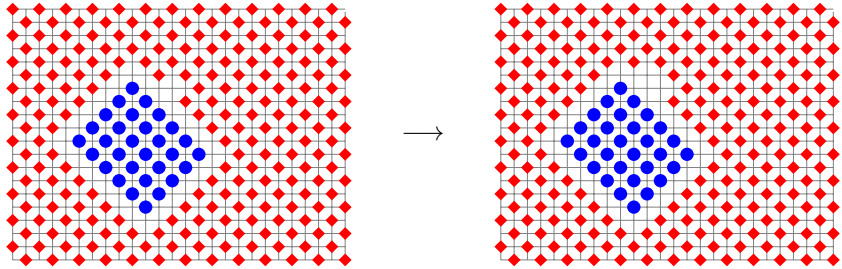

Counting the number of possible transitions in the critical gate
using~\eqref{eq:critical-gate:size}, we get
\begin{align}
	\abs{[Q,Q^*]} &= \abs{V}\times 2\times 2\ell^*\times 2 = 4mn\ell^* \;.
\end{align}
Theorem~\ref{thm:crossover:sharp} thus gives the sharp estimate
\begin{align}
	\Aboxed{
	\xExp_u[\hat{T}_v] 
		&= \frac{1}{4mn\ell^*}
			\frac{\lambda^{\ell^*(\ell^*+1)+1}}{\bar{\lambda}^{\ell^*(\ell^*-1)}}
			[1+\smallo(1)]
			} \qquad\text{as $\lambda\to\infty$,}
\end{align}
for the expected crossover time.


\subsection{Widom-Rowlinson on a torus}
\label{sec:doubled-torus:metastability}

As observed in Section~\ref{sec:intro:model}, the Widom-Rowlinson dynamics on 
the torus $\ZZ_m\times\ZZ_n$ is equivalent to the hard-core dynamics on the doubled torus.
We assume that $0<\alpha<1$, $\nicefrac{4}{\alpha}\notin\ZZ$ and $m,n\gg\nicefrac{1}{\alpha}$.
The isoperimetric function $\Delta(s)$ on the doubled is provided in
Example~\ref{exp:isoperimetric:doubled-torus}, using the equivalence
of the bipartite isoperimetric problem on a doubled torus and the vertex isoperimetric
problem on the torus and the known result about the vertex isoperimetric problem
on the torus. We shall obtain the exponentiality of the distribution of the crossover time
and the order of magnitude of its expected value. A sharp asymptotic for the expected 
crossover time and a description of the critical droplet are obtained assuming 
Conjecture~\ref{conj:isoperimetric:doubled-torus:progressions:property}
regarding the solutions of the vertex isoperimetric problem on $\ZZ\times\ZZ$
is true.

The proof of the following lemma appears in Appendix~\ref{apx:critical-size}.
\begin{lemma}[{\bf Critical size: doubled torus}]
\label{lem:doubled-torus:critical-size}
	Suppose $0<\alpha<1$ and $\nicefrac{4}{\alpha}\notin\ZZ$, and let $\Delta(s)$ 
	be the isoperimetric function of the doubled version of a torus $\ZZ_m\times\ZZ_n$ 
	with $m,n\gg \nicefrac{1}{\alpha}$. Then the function $g(s)=\Delta(s)-\alpha(s-1)$ 
	has a unique maximum on $\NN$ at
	\begin{align}
		s^*\isdef
			\begin{cases}
				(\ell^*)^2 + (\ell^*-1)^2 + \ell^*		&\text{if $\ell^*>\nicefrac{1}{\alpha}$,} \\
				(\ell^*)^2 + (\ell^*-1)^2 + 3\ell^*		&\text{if $\ell^*<\nicefrac{1}{\alpha}$,}
			\end{cases}
	\end{align}
	where $\ell^*\isdef[\nicefrac{1}{\alpha}]$ is the closest integer to $\nicefrac{1}{\alpha}$.
\end{lemma}

As in the previous section, finding the exact value of resettling size $\tilde{s}$ (i.e., the smallest 
$s$ for which $\Delta(s)\leq\alpha s$) is not necessary. It is sufficient to observe that $\tilde{s}$ 
exists (the inequality is achieved for instance for
$s>(\nicefrac{2}{\alpha}+1)^2+(\nicefrac{2}{\alpha})^2-1$)
and is independent of $m$ and $n$ (as long as $m,n\gg \nicefrac{1}{\alpha}$).

Hypothesis~\eqref{hypothesis:v-is-stable} is clearly satisfied.
The existence of isoperimetric numberings of length at least $\tilde{s}$ was demonstrated in 
Example~\ref{exp:isoperimetric:doubled-torus} (as long as $\tilde{s}\ll m,n$). As a result, 
hypothesis~\eqref{hypothesis:numbering:single} and (by translation symmetry) 
hypothesis~\eqref{hypothesis:numbering:all} are both satisfied.
Theorems~\ref{thm:mean-crossover:magnitude} and~\ref{thm:crossover:exponential} thus
establish the asymptotic exponentiality of the crossover time and provide the estimate
\begin{align}
	\xExp_u[\hat{T}_v] &\asymp
		\begin{dcases}
			\frac{%
				\lambda^{(\ell^*+1)^2+(\ell^*)^2+\ell^*+1}
			}{%
				\bar{\lambda}^{(\ell^*)^2+(\ell^*-1)^2+\ell^*-1}
			}
				& \text{if $\ell^*>\nicefrac{1}{\alpha}$,} \\[10pt]
			\frac{%
				\lambda^{(\ell^*+1)^2+(\ell^*)^2+3\ell^*+3}
			}{%
				\bar{\lambda}^{(\ell^*)^2+(\ell^*-1)^2+3\ell^*-1}
			}		& \text{if $\ell^*<\nicefrac{1}{\alpha}$,}
		\end{dcases}
		\qquad\text{as $\lambda\to\infty$,}
\end{align}
for its mean, where $\ell^*\isdef[\nicefrac{1}{\alpha}]$ is the closest integer to 
$\nicefrac{1}{\alpha}$.

A more accurate estimate on the mean crossover time as well as a description
of the critical droplet is provided by Theorem~\ref{thm:crossover:sharp},
which relies on hypotheses~\eqref{hypothesis:uniqueness} and~\eqref{hypothesis:critical-set}.
Hypothesis~\eqref{hypothesis:uniqueness} is already verified
in Lemma~\ref{lem:doubled-torus:critical-size}.
Proposition~\ref{prop:critical-gate:identification} reduces the verification
of~\eqref{hypothesis:critical-set} to the verification of simpler 
conditions~\eqref{hypothesis:critical-set-replacement:values}
and~\eqref{hypothesis:critical-set-replacement:existence}.
Choose $\kappa\isdef\lceil\nicefrac{1}{\alpha}\rceil-1$, which coincides
with either $\ell^*$ or $\ell^*-1$,
depending on whether the fractional part of $\nicefrac{1}{\alpha}$
is smaller or larger than $\nicefrac{1}{2}$.
Conditions~\eqref{hypothesis:critical-set-replacement:values:A} and~\eqref{hypothesis:critical-set-replacement:values:B} follow from the monotonicity of $\Delta(s)$. 
Condition~\eqref{hypothesis:critical-set-replacement:values:C} becomes evident once we 
note that $\Delta(s)=\Delta(s-1)+1$ whenever $s=\ell^2+(\ell-1)^2+\ell$ or $s=\ell^2+(\ell-1)^2+3\ell$.

To proceed, let us consider the two cases $\ell^*>\nicefrac{1}{\alpha}$
and $\ell^*<\nicefrac{1}{\alpha}$ separately.
\begin{case}[1]{$\ell^*>\nicefrac{1}{\alpha}$.}
	So, $s^*=(\ell^*)^2+(\ell^*-1)^2+\ell^*$ and $\kappa=\ell^*-1$.\\
	In order to verify~\eqref{hypothesis:critical-set-replacement:existence},
	observe that $s^*-1=(\ell^*)^2+(\ell^*-1)^2+\ell^*-1$
	and $s^*+\kappa=(\ell^*)^2+(\ell^*-1)^2+2\ell^*-1$
	are critical cardinalities of types~II and~III
	(see Fig.~\ref{fig:isoperimetric:doubled-torus:pareto}).
	From the characterisation of Pareto optimal sets in Section~\ref{sec:isoperimetric:doubled-torus},
	we find that
	\begin{itemize}
		\item $\family{A}$ consists precisely of sets $N^{\ell^*-2}(S)$
			where $S$ is a seed of type~II, and
		\item $\family{C}$ consists precisely of sets $N^{\ell^*-1}(S')$
			where $S'$ is a seed of type~III.
	\end{itemize}
	Conditions~\eqref{hypothesis:critical-set-replacement:before-A}
	and~\eqref{hypothesis:critical-set-replacement:after-C} follow
	from Observation~\ref{obs:isoperimetric:doubled-torus:numberings}.
	
	Assuming Conjecture~\ref{conj:isoperimetric:doubled-torus:progressions:property}
	is true, and using Observation~\ref{obs:isoperimetric:doubled-torus:progressions:existence},
	we obtain a characterisation of~$\family{B}$.
	\begin{itemize}
		\item $\family{B}$ consists precisely the sets $B$ with $\abs{B}=s^*$
			such that $N^{\ell^*-2}(S)\subseteq B\subseteq N^{\ell^*-1}(S')$
			for some seeds $S$ and $S'$ of type $II$ and $III$
			where $S\subseteq N(S')$.
	\end{itemize}
	A typical transition through the critical gate $[Q,Q^*]$ is
	depicted in Figure~\ref{fig:gate:doubled-torus:caseI}.
	
	Counting the number of possible transitions in the critical gate
	using~\eqref{eq:critical-gate:size}, we get
	\begin{align}
		\abs{[Q,Q^*]} &= \abs{V}\times 4\times 3\ell^*\times 2 = 24mn\ell^* \;.
	\end{align}
	Theorem~\ref{thm:crossover:sharp} thus gives the sharp estimate
	\begin{align}
		\Aboxed{
		\xExp_u[\hat{T}_v] 
			&= \frac{1}{24mn\ell^*}
				\frac{%
					\lambda^{(\ell^*+1)^2+(\ell^*)^2+\ell^*+1}
				}{%
					\bar{\lambda}^{(\ell^*)^2+(\ell^*-1)^2+\ell^*-1}
				}
				[1+\smallo(1)]
				} \qquad\text{as $\lambda\to\infty$,}
	\end{align}
	for the expected crossover time.
\end{case}
\begin{case}[2]{$\ell^*<\nicefrac{1}{\alpha}$.}
	So, $s^*=(\ell^*)^2+(\ell^*-1)^2+3\ell^*$ and $\kappa=\ell^*$.\\
	In order to verify~\eqref{hypothesis:critical-set-replacement:existence},
	observe that $s^*-1=(\ell^*)^2+(\ell^*-1)^2+3\ell^*-1$
	and $s^*+\kappa=(\ell^*+1)^2+(\ell^*)^2$
	are critical cardinalities of types~IV and~I
	(see Fig.~\ref{fig:isoperimetric:doubled-torus:pareto}).
	From the characterisation of Pareto optimal sets in Section~\ref{sec:isoperimetric:doubled-torus},
	we find that
	\begin{itemize}
		\item $\family{A}$ consists precisely of sets $N^{\ell^*-1}(S)$
			where $S$ is a seed of type~IV, and
		\item $\family{C}$ consists precisely of sets $N^{\ell^*}(S')$
			where $S'$ is a seed of type~I.
	\end{itemize}
	Conditions~\eqref{hypothesis:critical-set-replacement:before-A}
	and~\eqref{hypothesis:critical-set-replacement:after-C} follow
	from Observation~\ref{obs:isoperimetric:doubled-torus:numberings}.
	
	Assuming Conjecture~\ref{conj:isoperimetric:doubled-torus:progressions:property}
	is true, and using Observation~\ref{obs:isoperimetric:doubled-torus:progressions:existence},
	we obtain a characterisation of~$\family{B}$.
	\begin{itemize}
		\item $\family{B}$ consists precisely the sets $B$ with $\abs{B}=s^*$
			such that $N^{\ell^*-1}(S)\subseteq B\subseteq N^{\ell^*}(S')$
			for some seeds $S$ and $S'$ of type $IV$ and $I$
			where $S\subseteq N(S')$.
	\end{itemize}
	A typical transition through the critical gate $[Q,Q^*]$ is
	depicted in Figure~\ref{fig:gate:doubled-torus:caseII}.
	
	Counting the number of possible transitions in the critical gate
	using~\eqref{eq:critical-gate:size}, we get
	\begin{align}
		\abs{[Q,Q^*]} &= \abs{V}\times 4\times (\ell^*+1)\times 2 = 8mn(\ell^*+1)\;.
	\end{align}
	Theorem~\ref{thm:crossover:sharp} thus gives the sharp estimate
	\begin{align}
		\Aboxed{
		\xExp_u[\hat{T}_v] 
			&= \frac{1}{8mn(\ell^*+1)}
				\frac{%
					\lambda^{(\ell^*+1)^2+(\ell^*)^2+3\ell^*+3}
				}{%
					\bar{\lambda}^{(\ell^*)^2+(\ell^*-1)^2+3\ell^*-1}
				}
				[1+\smallo(1)]
				} \qquad\text{as $\lambda\to\infty$,}
	\end{align}
	for the expected crossover time.
\end{case}

\begin{figure}[htbp]
	\centering
	\begin{subfigure}[b]{\textwidth}
		\centering
		\begin{tabular}{ccc}
			\raisebox{-\height/2+2pt}{%
			{%
			\tikzsetfigurename{doubled_torus_gate_I_Q}
			\begin{tikzpicture}[scale=0.23,>=stealth,shorten >=1]
				\draw[yellow, line width=5] (9-0.3,-2+0.3) -- (13+0.3,-6-0.3);
				\draw[yellow, line width=5] (4-0.3,-7+0.3) -- (8+0.3,-11-0.3);
				\draw[yellow, line width=5] (4-0.3,-6-0.3) -- (8+0.3,-2+0.3);
				\DrawDoubledTorusGateI
				\node[red particle] at (10,-3) {};
			\end{tikzpicture}
			}%
			}%
			& $\longrightarrow$ &
			\raisebox{-\height/2+2pt}{%
			{%
			\tikzsetfigurename{doubled_torus_gate_I_Qstar}
			\begin{tikzpicture}[scale=0.23,>=stealth,shorten >=1]
				\draw[yellow, line width=5] (9-0.3,-2+0.3) -- (13+0.3,-6-0.3);
				\draw[yellow, line width=5] (4-0.3,-7+0.3) -- (8+0.3,-11-0.3);
				\draw[yellow, line width=5] (4-0.3,-6-0.3) -- (8+0.3,-2+0.3);
				\DrawDoubledTorusGateI
			\end{tikzpicture}
			}%
			}
		\end{tabular}
		\caption{Case 1: $\ell^*=\lceil\nicefrac{1}{\alpha}\rceil$.}
		\label{fig:gate:doubled-torus:caseI}
	\end{subfigure}
	
	\bigskip
	
	\begin{subfigure}[b]{\textwidth}
		\centering
		\begin{tabular}{ccc}
			\raisebox{-\height/2+2pt}{%
			{%
			\tikzsetfigurename{doubled_torus_gate_II_Q}
			\begin{tikzpicture}[scale=0.23,>=stealth,shorten >=1]
				\draw[yellow, line width=5] (9-0.3,-12-0.3) -- (14+0.3,-7+0.3);
				\DrawDoubledTorusGateII
				\node[red particle] at (9,-12) {};
			\end{tikzpicture}
			}%
			}%
			& $\longrightarrow$ &
			\raisebox{-\height/2+2pt}{%
			{%
			\tikzsetfigurename{doubled_torus_gate_II_Qstar}
			\begin{tikzpicture}[scale=0.23,>=stealth,shorten >=1]
				\draw[yellow, line width=5] (9-0.3,-12-0.3) -- (14+0.3,-7+0.3);
				\DrawDoubledTorusGateII
			\end{tikzpicture}
			}%
			}	
		\end{tabular}
		\caption{Case 2: $\ell^*=\lfloor\nicefrac{1}{\alpha}\rfloor$.}
		\label{fig:gate:doubled-torus:caseII}
	\end{subfigure}
	\caption{Typical transitions through the critical gate for the doubled torus.
		The critical length $\ell^*$ in both cases is assumed to be~$4$.
		The `hole' can be anywhere in the highlighted region.
		Once a two-site `hole' is produced,
		with probability close to $1$ a blue particle appears very quickly
		in the opened-up space.
	}
	\label{fig:gate:doubled-torus}
\end{figure}
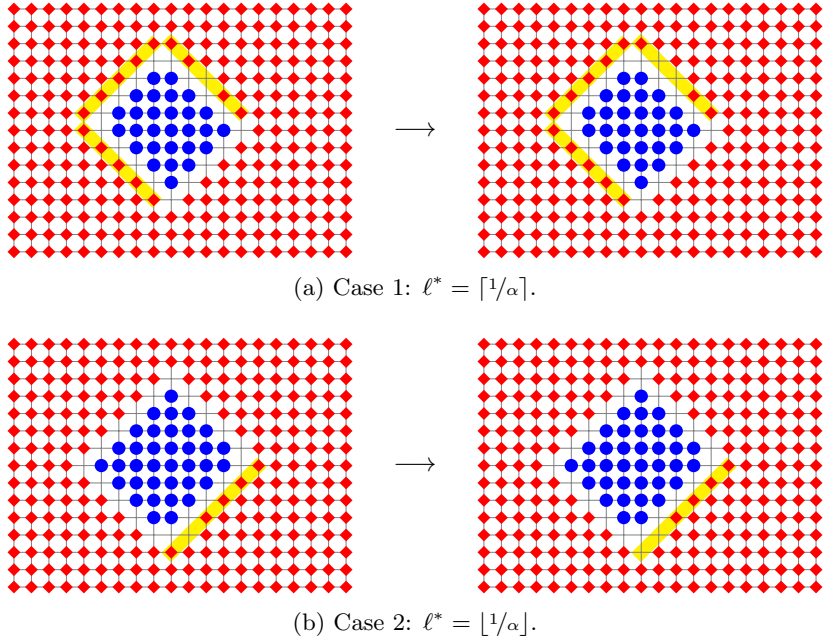


\subsection{Graph girth and crossover time}
\label{sec:tree-like:metastability}

In Example~\ref{exp:isoperimetric:tree-like}, we noted that the optimal isoperimetric 
cost in a regular bipartite graph with large girth grows linearly for small cardinalities.
Likewise, the optimal isoperimetric cost in a doubled version of a bipartite graph with 
large girth is linear when restricted to small cardinalities. Since $g(s)=\Delta(s)-\alpha(s-1)$ 
has no critical point when $\Delta(s)$ is linear, we obtain lower bounds for the order of 
magnitude of the crossover time of the hard-core dynamics and Widom-Rowlinson dynamics 
on a (bipartite) regular graph in terms of the girth of the graph.

First, let us consider a $d$-regular bipartite graph in which the length of each cycle
is at least $\ell$.  We know from Example~\ref{exp:isoperimetric:tree-like}
that $\Delta(s)=(d-2)s+1$ for $0<s<\nicefrac{\ell}{2}$. Therefore, $g(s)=\Delta(s)-\alpha(s-1)
=(d-2-\alpha)s+1+\alpha$. If $d=2$ (i.e., if the graph is a cycle), the critical size and the 
resettling size are $s^*=\tilde{s}=1$. The 
hypotheses~\eqref{hypothesis:v-is-stable}--\eqref{hypothesis:critical-set}
are trivially satisfied with $\family{A}=\{\varnothing\}$
and $\family{B}=\{\{b\}: b\in V\}$ and $\abs{[Q,Q^*]}=2\abs{V}$.
Therefore, in this case, we recover the result of Example~\ref{exp:hard-core:cycle}.
If, on the other hand, $d>2$, the function $g(s)$ is increasing for $0<s<\nicefrac{\ell}{2}$ 
and can achieve its maximum only at $s\geq\lfloor\nicefrac{\ell}{2}\rfloor$.
While Theorem~\ref{thm:mean-crossover:magnitude} is not applicable
(condition~\eqref{hypothesis:numbering:single} may not be satisfied),
direct application of Proposition~\ref{prop:escape:mean} and
Lemmas~\ref{lem:hard-core:progression:critical-resistance}--\ref{lem:hard-core:standard-path:optimality}
leads to the following lower bound for the expected crossover time.

\begin{proposition}[{\bf Lower bound for expected crossover time: hard-core}]
	Let $G=(U,V,E)$ be a $d$-regular bipartite graph with $d>2$
	in which the length of each cycle is at least $\ell$.
	Then the crossover time from $u$ to $v$ on $G$ satisfies
	\begin{align}
		\xExp_u[\hat{T}_v] &\succeq
			\frac{\lambda^{\Delta(s_\ell)+s_\ell-1}}{\bar{\lambda}^{s_\ell-1}} =
			\frac{\lambda^{(d-2-\alpha)s_\ell+1+\alpha}}{\bar{\lambda}^{s_\ell-1}}
			\qquad\text{as $\lambda\to\infty$,}
	\end{align}
	where $s_\ell\isdef\lfloor\nicefrac{\ell}{2}\rfloor$.
\end{proposition}

For the Widom-Rowlinson model on a $d$-regular graph we get a similar lower bound
on the expected crossover time from the all-red to all-blue configuration in terms of the 
graph girth. Recall that the Widom-Rowlinson dynamics on a graph $G$ is equivalent
to the hard-core dynamics on the doubled graph $G^{[2]}$. In 
Example~\ref{exp:isoperimetric:tree-like}, we saw that $\Delta(s)=(d-2)s+2$ for $0<s<\ell-1$.
Therefore, $g(s)=\Delta(s)-\alpha(s-1)=(d-2-\alpha)s+2+\alpha$. If $d=2$ (i.e., when the 
graph is a cycle), we again have $s^*=\tilde{s}=1$. The 
hypotheses~\eqref{hypothesis:v-is-stable}--\eqref{hypothesis:critical-set}
are again trivially satisfied with $\family{A}=\{\varnothing\}$ and $\family{B}
=\{\{b\}: b\in V\}$ and $\abs{[Q,Q^*]}=3\abs{V}$. Therefore, in this case, we recover the 
result of Example~\ref{exp:widom-rowlinson:cycle} even when the cycle is not even.
If, on the other hand, $d>2$, the function $g(s)$ is increasing for $0<s<\ell-1$ and can 
achieve its maximum only at $s\geq\ell-1$. Therefore, we get a similar lower bound for 
the expected crossover time using Proposition~\ref{prop:escape:mean}.

\begin{proposition}[{\bf Lower bound for expected crossover time: Widom-Rowlinson}]
	Let $G$ be a $d$-regular graph with $d>2$ in which the length of
	each cycle is at least $\ell$, and $G^{[2]}=(U,V,E)$ its doubled version.
	Then, the crossover time from $u$ to $v$ on $G^{[2]}$ satisfies
	\begin{align}
		\xExp_u[\hat{T}_v] &\succeq
			\frac{\lambda^{\Delta(s_\ell)+s_\ell-1}}{\bar{\lambda}^{s_\ell-1}} =
			\frac{\lambda^{(d-2-\alpha)s_\ell+2+\alpha}}{\bar{\lambda}^{s_\ell-1}}
			\qquad\text{as $\lambda\to\infty$,}
	\end{align}
	where $s_\ell\isdef\ell-1$.
\end{proposition}


\subsection{Hard-core and Widom-Rowlinson on a hypercube}
\label{sec:hypercube:metastability}

As we saw in Example~\ref{exp:isoperimetric:hypercube}, the doubled version of a $d$-dimensional 
hypercube $H_d$ is isomorphic to a $(d+1)$-dimensional hypercube $H_{d+1}$, hence the 
Widom-Rowlinson dynamics on $H_d$ is essentially the same as the hard-core dynamics
on $H_{d+1}$. As before, we assume that $0<\alpha<1$. 

Condition~\eqref{hypothesis:v-is-stable} is satisfied for every doubled graph.
From Example~\ref{exp:isoperimetric:hypercube}, we know that the sites of $H_d$ admit a 
complete (vertex) isoperimetric numbering. By symmetry, every site in $H_d$ is the starting 
point of an isoperimetric numbering. Therefore, conditions~\eqref{hypothesis:numbering:single} 
and~\eqref{hypothesis:numbering:all} are satisfied. The asymptotic exponentiality of the 
crossover time follows immediately from Theorem~\ref{thm:crossover:exponential}.
The conclusion of Theorem~\ref{thm:mean-crossover:magnitude} is also true, but in this
case, finding the critical size $s^*$ is more challenging because the function 
$g(s)\isdef\Delta(s)-\alpha(s-1)$ is known only implicitly (see 
Proposition~\ref{prop:isoperimetric:hypercube:recursion}). We state this is an open question.

\begin{question}[{\bf Critical size: hypercube}]
\label{q:hypercube:critical-size}
	Let $\Delta(s)$ denote the (vertex) isoperimetric function
	for the hypercube $H_d$.
	Which value $s$ maximizes the function $g(s)\isdef\Delta(s)-\alpha(s-1)$?
	What is the value of the maximum?
\end{question}

If we further assume that $\alpha$ is irrational,
then hypothesis~\eqref{hypothesis:uniqueness} will also be satisfied.
We do not know the status of conditions~\eqref{hypothesis:critical-set}
or~\eqref{hypothesis:critical-set-replacement:values}--\eqref{hypothesis:critical-set-replacement:existence}.

\begin{question}[{\bf Critical gate: hypercube}]
\label{q:hypercube:hypotheses}
	Are conditions~\eqref{hypothesis:critical-set-replacement:values}
	and~\eqref{hypothesis:critical-set-replacement:existence} satisfied
	for the hypercube?
	If not, how about condition~\eqref{hypothesis:critical-set}?
\end{question}


\appendix

\section{Proofs}	
\label{sec:proofspottheo}


\subsection{Nash-Williams inequality}
\label{apx:nash-williams}

\begin{proof}[Proof of Proposition~\ref{prop:nash-williams:dual:extended}]
The proof is similar to the proof of the (dual) Nash-Williams inequality. Let $W\isdef W_{A,B}$ 
be the voltage function when $B$ is connected to the ground and $A$ is connected to a unit 
voltage source (i.e., $W$ is harmonic on $\spX\setminus (A\cup B)$ with boundary condition 
$W|_A\equiv 1$ and $W|_B\equiv 0$). Let $I$ be the corresponding current flow. By definition,
\begin{align}
\effC{A}{B} &= (\divergence I)(A)\isdef \sum_{x\in A}(\divergence I)(x).
\end{align}
Write $\dd W(x,y)\isdef W(y)-W(x)$ for the relative voltage of two nodes. By the conservation 
of energy (also known as the adjointness of $\theta\mapsto\divergence\theta$ and $f\mapsto
\dd f$; see e.g.\ Lyons and Peres~\cite{LyoPer14}, Section~2.4), we have
\begin{align}
\effC{A}{B}\,^2 
&= \frac{1}{2}\sum_{x,y} c(x,y) \left(\dd W(x,y)\right)^2
\geq \sum_k \sum_{e\in \omega_k} \frac{1}{n(e)} c(e)\left(\dd W(e)\right)^2.
\end{align}
By the Cauchy-Schwartz inequality, for each $k$ we can write 
\begin{align}
\left(\sum_{e\in\omega_k} \frac{1}{n(e)} c(e)\left(\dd W(e)\right)^2\right)
\left(\sum_{e\in\omega_k} n(e)r(e)\right) 
&\geq \left(\sum_{e\in\omega_k} \sqrt{\frac{c(e)}{n(e)}}\dd W(e) \sqrt{n(e)r(e)}\right)^2\nonumber\\
&= \left(\sum_{e\in\omega_k} \dd W(e)\right)^2 = 1.
\end{align}
The claim follows.
\end{proof}


\subsection{Effective resistance versus critical resistance}
\label{apx:critical-resistance}

\begin{proof}[Proof of Proposition~\ref{prop:bounds:communication-height:abstract}]
The right-hand inequality is immediate from the dual Nash-Williams inequality
\eqref{eq:nash-williams:dual:simplified} by choosing $k$ to be the length of the longest 
path on the graph. The left-hand inequality follows from the simplified Nash-Williams 
inequality \eqref{eq:nash-williams:simplified}. Namely, let
\begin{align}
C &\isdef \left\{x\colon\, \Psi(x,A) < \Psi(A,B)\right\}.
\end{align}
Then, by the strong triangle inequality, $r(x,y)\geq\Psi(A,B)$ for every $x\in C$ and 
$y\notin C$. Therefore,
\begin{align}
\effC{A}{B} \leq \effC{C}{C^\complement} 
&\leq \abs{\partial C} \,\sup_{\mathclap{x\in C, y\notin C}}\; c(x,y)
\leq \abs{\partial C}\,\frac{1}{\Psi(A,B)}.
\end{align}
Therefore, the left inequality in~\eqref{eq:eff-resist:equiv-metric}
holds when $k$ is at least the size of the largest cut on 
the graph.
\end{proof}


\subsection{Estimates on voltage}
\label{apx:voltage:estimate}

\begin{proof}[Proof of Proposition~\ref{prop:voltage:bound:effective-resistance}]
By the short-circuit principle, we may assume that $A$ and $B$ are singletons, i.e., 
$A=\{a\}$ and $B=\{b\}$ for some nodes $a$ and $b$. We have
\begin{align}
W_{A,B}(x) = \xPr_x(T_a<T_b) 
&\leq \frac{\xPr_x(T_a<T_b)}{\xPr_a(T_x<T_b)} 
= \frac{\effR{x}{b}}{\effR{a}{b}},
\end{align}
where the last equality uses the reciprocity equality in \eqref{eq:hitting-order:reciprocity}.
The other inequality follows symmetrically, by noting that $W_{A,B}(x)=1-\xPr_x(T_b<T_a)$.
\end{proof}

\begin{proof}[Proof of Proposition~\ref{prop:voltage:bounds:valleys:abstract}]
For brevity, we write $W(x)$ instead of $W_{A,B}(x)$. If $x,y\in A\cup B$, then the claim 
is trivial. If $x$ or $y$ is in $A\cup B$ and the other is not, then the conclusion follows 
directly from Proposition~\ref{prop:voltage:bounds:communication-height:abstract}.
So, assume that $x,y\in\spX\setminus (A\cup B)$. We verify that
\begin{align}
W(x) &\leq W(y) + \bar{k} \frac{\Psi(x,y)}{\Psi(A,B)}.
\end{align}
The opposite inequality follows by symmetry.
	
By the ultra-metric inequality, we have $\Psi(A,B)\leq \max\{\Psi(y,A),\Psi(x,y),\Psi(x,B)\}$.
If $\Psi(x,y)\geq \Psi(A,B)$, then the claim is trivial. There remain two cases.

\begin{case}[1]{$\Psi(y,A)\geq\Psi(A,B)$.}
By conditioning on the order of the occurrence of $T_A$, $T_B$ 
and $T_y$, we can write
\begin{align}
W(x) &= \xPr_x(T_A<T_B) \leq \xPr_x(T_y<T_A<T_B) + \xPr_x(T_A<T_y).
\end{align}
The first term can be estimated as
\begin{align}
\xPr_x(T_y<T_A<T_B) 
&= \xPr_x(T_y<T_{A\cup B}) \xPr_y(T_A<T_B) \leq \xPr_y(T_A<T_B) = W(y).
\end{align}
For the second term, by Proposition~\ref{prop:voltage:bounds:communication-height:abstract},
\begin{align}
\xPr_x(T_A<T_y) &\leq \bar{k} \frac{\Psi(x,y)}{\Psi(A,y)} \leq
\bar{k} \frac{\Psi(x,y)}{\Psi(A,B)}.
\end{align}
It follows that
\begin{align}
W(x) &\leq W(y) + \bar{k} \frac{\Psi(x,y)}{\Psi(A,B)}.
\end{align}
\end{case}
	
\begin{case}[2]{$\Psi(x,B)\geq\Psi(A,B)$.}
By a symmetric reasoning as in the first case,
\begin{align}
1-W(y) &= \xPr_y(T_B<T_A) \leq
\xPr_y(T_x<T_B<T_A) + \xPr_y(T_B<T_x).
\end{align}
Again, the first term reduces to
\begin{align}
\xPr_y(T_x<T_B<T_A) &= \xPr_y(T_x<T_{A\cup B}) \xPr_x(T_B<T_A)
\leq \xPr_x(T_B<T_A) = 1-W(x).
\end{align}
Similarly, for the second term, we get
\begin{align}
\xPr_y(T_B<T_x) &\leq \bar{k} \frac{\Psi(x,y)}{\Psi(B,x)}
\leq \bar{k} \frac{\Psi(x,y)}{\Psi(A,B)}.
\end{align}
Therefore
\begin{align}
1-W(y) &\leq 1-W(x) + \bar{k} \frac{\Psi(x,y)}{\Psi(A,B)},
\end{align}
which again gives
\begin{align}
	W(x) &\leq W(y) + \bar{k} \frac{\Psi(x,y)}{\Psi(A,B)}.\\
	& \qedhere
\end{align}
\end{case}
\end{proof}


\subsection{Characterisation of transience}
\label{apx:transience:characterization}

\begin{proof}[Proof of Proposition~\ref{prop:transience:characterization}]
First, suppose that $\xExp_a[T_{J^-(a)}]\prec\tau$. By conditioning, we have
\begin{align}
G_\tau(a,a) &= \sum_{s,b}\xPr_a\big(T_{J^-(a)}=s, X(s)=b\big)
\xExp_a\left[\text{$\#$ of visits to $a$ before $\tau$}\,\middle|\,T_{J^-(a)}=s, X(s)=b\right]\nonumber\\
&\leq \sum_{s=0}^\infty\sum_{b\in J^-(a)}\xPr_a\big(T_{J^-(a)}=s, X(s)=b\big)(s+G_{\tau-s}(b,a))\nonumber\\
&\leq \sum_{s=0}^\infty\xPr_a(T_{J^-(a)}=s)\cdot s + \sum_{b\in J^-(a)} G_\tau(b,a)
\end{align}
as $\lambda\to\infty$. The first sum on the righthand side is simply $\xExp_a[T_{J^-(a)}]$, which 
is $\prec\tau$. For the second sum, we recall that $\pi(a)G_\tau(a,b)=\pi(b)G_\tau(b,a)$, by 
reversibility and since $\pi(a)\prec\pi(b)\preceq 1$, we again obtain $G_\tau(b,a)\prec G_\tau(a,b)
\preceq\tau$. Put together, we find that $G_\tau(a,a)\prec\tau$ as $\lambda\to\infty$.
	
Conversely, assume that $G_\tau(a,a)\prec\tau$. Let $A\subseteq\spX$ be the set of states that 
can be reached from $a$ without passing through $J^-(a)$. Note that $\pi(b)\preceq\pi(a)$ for 
each $b\in A$. Therefore, by the reciprocity identify in \eqref{eq:green:reciprocity}, we have that
\begin{align}
G_{T_{J^-(a)}}(a,b) &\preceq G_{T_{J^-(a)}}(b,a) \leq G_{T_{J^-(a)}}(a,a).
\end{align}
Now we can write
\begin{align}
G_{T_{J^-(a)}}(a,a) &= \sum_{t=0}^\infty \xPr_a\big(X(t)=a, T_{J^-(a)}>t\big)\nonumber\\
&= \sum_{t=0}^{\tau-1} \xPr_a\big(X(t)=a, T_{J^-(a)}>t\big) 
+ \sum_{t=\tau}^\infty \sum_{b\in A} \xPr_a\big(X(\tau)= b, X(t)=a, T_{J^-(a)}>t\big)\nonumber\\
&\leq G_\tau(a,a) + \sum_{t=\tau}^\infty \sum_{b\in A}
\xPr_a\big(X(\tau)=b, T_{J^-(a)}>\tau\big) \xPr_b\big(X(t-\tau)=a, T_{J^-(a)}>t-\tau\big)\nonumber\\
&\leq G_\tau(a,a) + \xPr_a(T_{J^-(a)}>\tau) \sum_{b\in A} G_{T_{J^-(a)}}(b,a)\nonumber\\
&\leq G_\tau(a,a) + \abs{A} \xPr_a(T_{J^-(a)}>\tau) G_{T_{J^-(a)}}(a,a),
\end{align}
which implies
\begin{align}
G_{T_{J^-(a)}}(a,a) &\leq \frac{1}{1-\abs{A}\xPr_a(T_{J^-(a)}>\tau)} G_\tau(a,a)
\end{align}
whenever $\abs{A}\xPr_a(T_{J^-(a)}>\tau)<1$. On the other hand, we note that
\begin{align}
	G_\tau(a,a) &\succeq \sum_{b\in A}G_\tau(a,b)
		= \sum_{k=0}^{\tau-1}\xPr_a(T_{J^-(a)}>k)
		\geq \tau\cdot\xPr_a(T_{J^-(a)}>\tau) \;,
\end{align}
hence,
\begin{align}
\xPr_a(T_{J^-(a)}>\tau) &\preceq \frac{1}{\tau} G_\tau(a,a) \prec 1.
\end{align}
Therefore, as $\lambda\to\infty$, we have $G_{T_{J^-(a)}}(a,a)\preceq G_\tau(a,a)\prec\tau$.
Finally,
\begin{align}
	\xExp_a[T_{J^-(a)}] &= \sum_{b\in A} G_{T_{J^-(a)}}(a,b) \preceq G_\tau(a,a) \prec \tau. \\
	& \qedhere
\end{align}
\end{proof}


\subsection{Mean escape time}
\label{apx:escape:mean}

\begin{proof}[Proof of Proposition~\ref{prop:escape:mean}]
We know from \eqref{eq:green:expression} that
\begin{align}
G_{T_{J^-(a)}}(a,x) &= \effR{a}{J^-(a)}\pi(x)\xPr_x(T_a<T_{J^-(a)})\nonumber\\
&= \pi(a)\effR{a}{J^-(a)}\frac{\pi(x)}{\pi(a)}\xPr_x(T_a<T_{J^-(a)}),
\end{align}
for every $x\in\spX$. For $x\notin J(a)\cup\{a\}$, we have, by definition, that $\pi(x)\prec\pi(a)$ 
as $\lambda\to\infty$, whereas for $x\in J(a)\setminus J^-(a)$, we have $\pi(x)\asymp\pi(a)$.
Therefore
\begin{align}
G_{T_{J^-(a)}}(a,x) &= 
\begin{cases}
\pi(a)\effR{a}{J^-(a)},						
\quad & \text{if $x=a$,} \\
\pi(a)\effR{a}{J^-(a)} \,\smallo(1),
\quad & \text{if $x\notin J(a)\cup\{a\}$,} \\
\pi(a)\effR{a}{J^-(a)} \,\bigo(1),
\quad & \text{if $x\in J(a)\setminus J^-(a)$,} \\
0,											
\quad & \text{if $x\in J^-(a)$,}
\end{cases}
\end{align}
which gives
\begin{align}
\xExp_a[T_{J^-(a)}] &= \sum_x G_{T_{J^-(a)}}(a,x)
\asymp \pi(a)\Psi(a,J^-(a))
\end{align}
as $\lambda\to\infty$.
\end{proof}

\begin{proof}[Proof of Proposition~\ref{prop:almost-escape:mean}]
This is similar to the proof of Proposition~\ref{prop:escape:mean}. Starting from
\begin{align}
G_{T_{J(a)}}(a,x) &= \pi(a)\effR{a}{J(a)}\frac{\pi(x)}{\pi(a)}\xPr_x(T_a<T_{J(a)}),
\end{align}
this time we can write
\begin{align}
G_{T_{J(a)}}(a,x) &=
\begin{cases}
\pi(a)\effR{a}{J(a)},					
&\quad \text{if $x=a$,} \\
\pi(a)\effR{a}{J(a)} \,\smallo(1),
\quad & \text{if $x\notin J(a)\cup\{a\}$,} \\
0,										
\quad & \text{if $x\in J(a)$,}
\end{cases}
\end{align}
as $\lambda\to\infty$. It follows
\begin{align}
\xExp_a[T_{J(a)}] &= \sum_x G_{T_{J(a)}}(a,x) = \pi(a)\effR{a}{J(a)} [1+\smallo(1)]
\end{align}
as $\lambda\to\infty$.
\end{proof}

\begin{proof}[Proof of Corollary~\ref{cor:escape:mean:cascade}]
Let $Z_0=J(a)$, and recursively define
\begin{align}
Z_{k+1} &\isdef J^-(Z_k)\isdef\{y: \text{$\pi(y)\succ\pi(x)$ for some $x\in Z_k$}\}.
\end{align}
Since the chain is finite, we must have $\varnothing\neq Z_n\subseteq Z$ for some $n$.
By conditioning, we have
\begin{align}
\xExp_a[T_{Z_{k+1}}] &= \xExp\big[\xExp_a[T_{Z_{k+1}} \,|\, T_{Z_k}, X(T_{Z_k})]\,\big],
\end{align}
where
\begin{align}
&\xExp_a\big[T_{Z_{k+1}} \,|\, T_{Z_k}, X(T_{Z_k})\big] 
= T_{Z_k} + \xExp_{X(T_{Z_k})}[T_{Z_{k+1}}]\nonumber\\
&\leq T_{Z_k} + \xExp_{X(T_{Z_k})}[T_{J^-(X(T_{Z_k}))}]
= T_{Z_k} + \pi(a)\Psi(a,J(a))\smallo(1).
\end{align}
Therefore
\begin{align}
\xExp_a[T_{Z_{k+1}}] &= \xExp_a[T_{Z_k}] + \pi(a)\Psi(a,J(a))\smallo(1)
\end{align}
and the claim follows by induction.
\end{proof}


\subsection{Rapid transition}
\label{apx:avalanche}

\begin{proof}[Proof of Corollary~\ref{cor:transition:avalanche}]
Recall from \eqref{eq:green:expression} that $G_{T_{Z}}(a,a)=\pi(a)\effR{a}{Z}\geq\pi(a)
\effR{a}{J(a)}$. Combined with Corollary~\ref{cor:escape:mean:cascade}, we get
\begin{align}
\xExp_a[T_Z-T^{(N_Z)}_a] 
&\leq \xExp_a[T_Z] - G_{T_Z}(a,a) = \pi(a)\effR{a}{Z}\smallo(1)
\end{align}
as $\lambda\to\infty$.
\end{proof}


\subsection{Renewal arguments}
\label{apx:exponential-law}

\begin{proof}[Proof of Proposition~\ref{prop:exponential:renewal}]
Replacing $\delta T_k$ with $\delta T_k/\mu$, we may assume that $\mu=1$. Note that
$M=(\nicefrac{1}{\varepsilon}-1)\mu+\eta$, which gives
\begin{align}
1 &= (\frac{1}{\varepsilon}-1)\frac{\mu}{M}+\frac{\eta}{M}
= \frac{1}{\varepsilon}(1 - \varepsilon + \varepsilon\frac{\eta}{\mu})\frac{\mu}{M}.
\end{align}
Letting $\lambda\to\infty$, we get $\varepsilon M=1+\smallo(1)$ and $\nicefrac{\eta}{M}=\smallo(1)$.
	
Let $G(\theta)\isdef\xExp[\ee^{\ii\theta\delta T}\,|\,B=\symb{0}]$ and $H(\theta)\isdef
\xExp[\ee^{\ii\theta\delta T}\,|\,B=\symb{1}]$ be the characteristic functions of $\delta T$ conditional 
on $B=\symb{0}$ and $B=\symb{1}$. The characteristic function of $T_N$ can be written as
\begin{align}
F(\theta) \isdef \xExp[\ee^{\ii\theta T_N}] 
= \sum_{n\geq 1} (1-\varepsilon)^{n-1}\varepsilon G(\theta)^{n-1}H(\theta)
= \frac{\varepsilon H(\theta)}{1 - (1-\varepsilon)G(\theta)}.
\end{align}
Therefore the characteristic function of $T_N/M$ is
\begin{align}
F(\theta/M) = \xExp[\ee^{\ii\theta\frac{T_N}{M}}] 
&= \frac{\varepsilon H(\theta/M)}{1 - (1-\varepsilon)G(\theta/M)}.
\end{align}
It remains to show that $F(\theta/M)\to\tilde{F}(\theta)$ as $\lambda\to\infty$, where 
$\tilde{F}(\theta)\isdef\frac{1}{1-\ii\theta}$ is the characteristic function of an exponential 
random variable with rate $1$.
	
To estimate $H(\theta/M)$, we note that, conditional on $B=\symb{1}$, $\delta T/M$ is a positive 
random variable whose expected value $\frac{\eta}{M}$ tends to $0$. Therefore $\delta T/M$ 
converges in distribution to a unit mass at $0$, and hence $H(\theta/M)=1+\smallo(1)$.
	
For $G(\theta/M)$, we need a more accurate estimate. We note that, conditional on $B=\symb{0}$, 
$\delta T$ is a positive random variable with mean $1$.  Therefore $G(\theta)$ is continuously 
differentiable with $G'(0)=\ii$, and a Taylor approximation gives $G(\theta)=1 + \ii\theta 
+ \smallo(\theta)$ as $\theta\to 0$. It follows that, for each $\theta\in\RR$, $G(\theta/M)
= 1 + \ii\frac{1}{M}\theta + \smallo(\frac{1}{M}) = 1 + \ii\varepsilon\theta + \smallo(\varepsilon)$ 
as $\lambda\to\infty$.
	
Altogether, for each $\theta\in\RR$, we get
\begin{align}
F(\theta/M) &= \frac{\varepsilon[1+\smallo(1)]}
{1-(1-\varepsilon)\left(1 + \ii\varepsilon\theta + \smallo(\varepsilon)\right)}
\to \frac{1}{1-\ii\theta}
\end{align}
as $\lambda\to\infty$. Therefore $T_N/M$ converges in distribution to an exponential random 
variable with rate~$1$. Finally, since the exponential distribution $t\to\ee^{-t}$ is continuous,
the convergence in \eqref{eq:thm:exponential} is uniform.
\end{proof}

\begin{proof}[Proof of Corollary~\ref{cor:escape:exponential}]
Let $\delta T\isdef T^+_{\{a\}\cup Z}$ be the first hitting time of $\{a\}\cup Z$, and choose 
$B\isdef\symb{1}$ if $T_Z<T^+_a$ and $B\isdef\symb{0}$ otherwise. From Corollary~\ref{cor:escape:mean:cascade}, 
it follows $M=\xExp_a[T_Z]\asymp\pi(a)\Psi(a,Z)=\pi(a)\Psi(a,J(a))\to\infty$ as $\lambda\to\infty$.
It follows that  $\varepsilon=\xPr_a(T_Z<T^+_a)=\frac{1}{\pi(a)\effR{a}{Z}} \asymp\frac{1}{M}
=\smallo(1)$. Furthermore, $\frac{\eta}{M}=\smallo(1)$ by Corollary~\ref{cor:transition:avalanche},
where $\eta=\xExp_a[T_Z\,|\,T_Z<T^+_a]$. The claim follows from 
Proposition~\ref{prop:exponential:renewal}.
	
To verify the continuous-time statement, we note that
\begin{align}
\varepsilon' \isdef \xPr_a(\hat{T}_Z<\hat{T}^+_a)
&= \xPr_a(T_Z<T^+_a),\nonumber\\
M' \isdef \xExp_a[\hat{T}_Z] 
&= \frac{1}{\gamma}\xExp_a[T_Z],\nonumber\\
\eta' \isdef \xExp_a[\hat{T}_Z\,|\, \hat{T}_Z<\hat{T}^+_a] 
&= \frac{1}{\gamma}\xExp_a[T_Z\,|\, T_Z<T^+_a],
\end{align}
the last two equalities following by simple calculations, and apply again 
Proposition~\ref{prop:exponential:renewal}.
\end{proof}

\begin{proof}[Proof of Proposition~\ref{prop:SES:exit}]
	Let $\beta>0$ be a constant such that
	$\sup_{x\in A}\xExp_x\big[T_{\partial A}\big]\leq\beta\Gamma(A)$
	for all sufficiently large~$\lambda$.
	By the Markov inequality, we have
	\begin{align}
		\xPr_x\big(T_{\partial A}>m\Gamma(A)\big)
			&\leq \frac{\beta}{m}
	\end{align}
	for every $x\in A$ and every constant $m\geq 1$.
	Chopping time into intervals of length $m\Gamma(A)$
	and applying this inequality iteratively we get,
	via the Markov property and time homogeneity, that
	\begin{align}
		\xPr_x\big(T_{\partial A}>nm\Gamma(A)\big)
			&\leq \Big(\frac{\beta}{m}\Big)^n
	\end{align}
	for $n=1,2,\ldots$.
	Setting $n=\big\lfloor\frac{1}{m}\rho\big\rfloor$ we obtain
	\begin{align}
	\label{eq:escape:tail:SES:proof}
		\xPr_x\big(T_{\partial A}>\rho\Gamma(A)\big)
			&\leq \Big(\frac{\beta}{m}\Big)^{\big\lfloor\frac{1}{m}\rho\big\rfloor}
			\asymp \Big(\frac{\beta}{m}\Big)^{\frac{1}{m}\rho} \;.
	\end{align}
	Setting $m\isdef\beta\ee$ gives the smallest value for
	$\big(\frac{\beta}{m}\big)^{\frac{1}{m}}$,
	hence the sharpest inequality~(\ref{eq:escape:tail:SES:proof}).
	The claim follows for $\alpha\isdef\ee^{-\frac{1}{\beta\ee}}$.
\end{proof}

\begin{proof}[Proof of Proposition~\ref{prop:SES:exit:tail:conditional}]
	For $x\in A$, we can write
	\begin{align}
		\xPr_x\big(T_{B_1}>\rho\Gamma(A)\,\big|\, T_{B_1}<T_{B_2}\big) &=
			\frac{%
				\xPr_x\big(T_{B_2}>T_{B_1}>\rho\Gamma(A)\big)
			}{%
				\xPr_x(T_{B_1}<T_{B_2})
			}
		\leq
			\frac{%
				\xPr_x\big(T_{\partial A}>\rho\Gamma(A)\big)
			}{%
				\xPr_x(T_{B_1}<T_{B_2})
			}
	\end{align}
	By Proposition~\ref{prop:SES:exit},
	we have $\xPr_x\big(T_{\partial A}>\rho\Gamma(A)\big)
		\preceq\alpha^{\rho}$ for some constant $\alpha>0$ independent of~$\rho$.
	Let $w$ be a simple path from $x$ to $B_1$ that does not pass through $\partial A$.
	The length of $w$ is at most $\abs{A}$
	and so $\xPr_x(T_{B_1}<T_{B_2})\geq\xPr_x(\text{$X$ follows $w$})\geq\kappa^{\abs{A}}$.
	The claim follows.
\end{proof}

\begin{proof}[Proof of Proposition~\ref{prop:SES:exit:mean:conditional}]
	We have
	\begin{align}
		\xExp_x\big[T_{B_1}\,\big|\, T_{B_1}<T_{B_2}\big] &=
			\sum_{t\geq 0} \xPr_x\big(T_{B_1}>t\,|\,T_{B_1}<T_{B_2}\big) \\
		&=
			\sum_{i=0}^\infty\sum_{j=0}^{\lceil \rho\Gamma(A)\rceil - 1}
			\xPr_x\big(T_{B_1}>i\lceil\rho\Gamma(A)\rceil+j\,|\,T_{B_1}<T_{B_2}\big) \;.
	\end{align}
	Using the bound in Proposition~\ref{prop:SES:exit:tail:conditional} iteratively,
	we get, via the Markov property and time homogeneity, that
	\begin{align}
		\xPr_x\big(T_{B_1}>i\lceil\rho\Gamma(A)\rceil+j\,|\,T_{B_1}<T_{B_2}\big)
		&\preceq
			\big(\alpha^\rho\,\kappa^{-\abs{A}}\big)^i \;,
	\end{align}
	which gives
	\begin{align}
		\xExp_x\big[T_{B_1}\,\big|\, T_{B_1}<T_{B_2}\big] &\preceq
			\rho\Gamma(A)\sum_{i=0}^\infty
			\big(\alpha^\rho\,\kappa^{-\abs{A}}\big)^i
		=
			\rho\Gamma(A)
			\frac{1}{1-\alpha^\rho\,\kappa^{-\abs{A}}}
		\asymp
			\rho\Gamma(A)
	\end{align}
	as $\lambda\to\infty$.
\end{proof}


\subsection{Critical gate}
\label{apx:critical-gate}

\begin{proof}[Proof of Proposition~\ref{prop:critical-gate:behind}]
Suppose that $r(x,y)\preceq\Psi(A,B)$. Since $x\in S$, there is a path $A\pathto x$ 
with $\Psi(A\pathto x)\preceq\Psi(A,B)$ that does not pass $Q^*$. Continuing this path 
with the transition $x\to y$, we obtain another path $A\pathto x\to y$ with $\Psi(A\pathto x\to y)
\preceq\Psi(A,B)$ that does not hit $Q^*$, except possibly at $y$. But, since $y$ is 
assumed to be outside $S$, it must be in $Q^*$. By assumption, $\Psi(y,B)\prec\Psi(A,B)$, 
which means that there is a path $y\pathto B$ with $\Psi(y\pathto B)\prec\Psi(A,B)$.
Gluing this path with $A\pathto x\to y$, we get an optimal path $A\pathto x\to y\pathto B$,
which, by definition, must pass through the critical gate. It follows that $x$ must be in $Q$, 
because $A\pathto x$ does not pass $Q^*$ and $y\pathto B$ does not pass $Q$.
\end{proof}

\begin{proof}[Proof of Proposition~\ref{prop:effective-conductance:critical-gate}]
The upper bound follows from the simplified Nash-Williams inequality in 
\eqref{eq:nash-williams:simplified}. For the lower bound, we use the extended dual 
Nash-Williams inequality (Proposituion~\ref{prop:nash-williams:dual:extended}).

To get the upper bound, let $S\isdef S(A,Q,Q^*,B)$ be the set of states behind the 
critical gate. By the simplified Nash-Williams inequality in \eqref{eq:nash-williams:simplified}
and Proposition~\ref{prop:critical-gate:behind}, we have
\begin{align}
\effC{A}{B} &\leq \effC{S}{S^\complement} = c(S,S^\complement)
= c(Q,Q^*) + \smallo(\frac{1}{\Psi(A,B)})
= c(Q,Q^*)\, [1+\smallo(1)]
\end{align}
as $\lambda\to\infty$.

Next we verify the lower bound. For each $x\in Q$ and $y\in Q^*$ with $x\link y$, let 
$\omega_{x,y}$ be an optimal path $A\pathto x\to y\pathto B$ whose parts $A\pathto x$ 
and $y\pathto B$ are also optimal. Thus, the transition $x\to y$ is the unique transition 
on $\omega_{x,y}$ whose resistance has the highest order of magnitude as $\lambda\to\infty$.
For each pair $(a,b)$ with $a\link b$, let $n(a,b)$ denote the number of pairs $(x,y)$ such 
that $\omega_{x,y}$ passes through $a\to b$. By the extended dual Nash-Williams inequality 
(Proposition~\ref{prop:nash-williams:dual:extended}), we have
\begin{align}
\effC{A}{B} &\geq \mathop{\sum_{x\in Q}\sum_{y\in Q^*}}_{x\link y}
\frac{1}{\sum_{(a,b)\in\omega_{x,y}} n(a,b)r(a,b)}.
\end{align}
But
\begin{align}
\sum_{(a,b)\in\omega_{x,y}} n(a,b)r(a,b) &= r(x,y) + \smallo(\Psi(A,B))
\end{align}
as $\lambda\to\infty$. Hence
\begin{align}	
\effC{A}{B} &\geq\mathop{\sum_{x\in Q}\sum_{y\in Q^*}}_{x\link y}
c(x,y) \, [1+\smallo(1)]
\end{align}
as $\lambda\to\infty$.
\end{proof}

\begin{proof}[Proof of Proposition~\ref{prop:critical-gate:abstract:choice}]
Let $N_{T_B}(x\to y)$ denote the number of times the chain moves through transition 
$x\to y$ until $T_B$. Using \eqref{eq:green:expression}, we have
\begin{align}
\xExp_a[N_{T_B}(x\to y)] = G_{T_B}(a,x)K(x,y) &= \effR{a}{B}\xPr_x(T_a<T_B)c(x,y).
\end{align}
According to Proposition~\ref{prop:effective-conductance:critical-gate}, we have the 
estimate $\effR{a}{B}=\frac{1+\smallo(1)}{c(Q,Q^*)}$ as $\lambda\to\infty$. Therefore
\begin{align}
\label{eq:edge-crossing:expected:asymptotic}
\xExp_a[N_{T_B}(x\to y)] 
&= \xPr_x(T_a<T_B)\frac{c(x,y)}{c(Q,Q^*)}[1+\smallo(1)]
\end{align}
as $\lambda\to\infty$.

\begin{enumerate}[label=(\roman*)]
\item 
Suppose that $(x,y)\in (S\times S^\complement) \setminus (Q\times Q^*)$ with $x\link y$.
Then, according to the Proposition~\ref{prop:critical-gate:behind}, we have $c(x,y)\prec 
\frac{1}{\Psi(a,B)} \asymp c(Q,Q^*)$. Therefore $\xExp_a[N_{T_B}(x\to y)]=\smallo(1)$ 
as $\lambda\to\infty$. That $\xPr_a(T_{xy}<T_B)=\smallo(1)$ follows from the Markov inequality.
\item 
Suppose that $(x,y)\in S\times S^\complement$ with $x\link y$. Exchanging $x$ and $y$ 
in \eqref{eq:edge-crossing:expected:asymptotic}, we have
\begin{align}
\xExp_a[N_{T_B}(y\to x)] 
&= \xPr_y(T_a<T_B)\frac{c(x,y)}{c(Q,Q^*)}[1+\smallo(1)].
\end{align}
We consider two cases. If $y\notin Q^*$, then we have (by Proposition~\ref{prop:critical-gate:behind})
that $c(x,y)\prec \frac{1}{\Psi(a,B)} \asymp c(Q,Q^*)$. Therefore $\xPr_a(T_{yx}<T_B)\leq\xExp_a
[N_{T_B}(y\to x)]=\smallo(1)$. If, on the other hand, $y\in Q^*$, then, by definition, $\Psi(y,B)\prec
\Psi(a,B)$. It follows from Proposition~\ref{prop:voltage:bounds:communication-height:abstract}
that
\begin{align}
\xPr_y(T_a<T_B) &\leq \bigo(1)\frac{\Psi(y,B)}{\Psi(a,B)} = \smallo(1)
\end{align}
as $\lambda\to\infty$. Therefore again $\xPr_a(T_{yx}<T_B)\leq\xExp_a[N_{T_B}(y\to x)]=\smallo(1)$.
\item
By assumption, $\Psi(a,x)\prec\Psi(a,B)$. Using 
Proposition~\ref{prop:voltage:bounds:communication-height:abstract}, we get
\begin{align}
\xPr_x(T_a<T_B) &\geq 1-\bigo(1)\frac{\Psi(a,x)}{\Psi(a,B)} = 1 - \smallo(1).
\end{align}
Therefore
\begin{align}
\xExp_a[N_{T_B}(x\to y)] &= \frac{c(x,y)}{c(Q,Q^*)}[1+\smallo(1)]
\end{align}
as $\lambda\to\infty$. By the Markov inequality,
\begin{align}
\xPr_a(T_{xy}<T_B) &\leq \frac{c(x,y)}{c(Q,Q^*)}[1+\smallo(1)]
\end{align}
as $\lambda\to\infty$. To see that the equality must hold, we combine the latter inequality
with the result of the first part to write
\begin{align}
1\leq\sum_{\bar{x}\in S}\sum_{\bar{y}\in S^\complement} 
\xPr_a(T_{\bar{x}\bar{y}}<T_B) 
&\leq \smallo(1) +  \sum_{\bar{x}\in Q}\sum_{\bar{y}\in Q^*}
\frac{c(\bar{x},\bar{y})}{c(Q,Q^*)}[1+\smallo(1)] = 1 + \smallo(1).
\end{align}
We conclude that
\begin{align}
\xPr_a(T_{xy}<T_B) &= \frac{c(x,y)}{c(Q,Q^*)}[1+\smallo(1)]
\end{align}
as $\lambda\to\infty$.
\qedhere 
\end{enumerate}
\end{proof}


\subsection{Critical resistance of standard paths}
\label{apx:standard-path}

\begin{proof}[Proof of Lemma~\ref{lem:hard-core:progression:critical-resistance}]
	Let $\omega(k_0),\omega(k_1),\ldots,\omega(k_m)$ be the backbone of $\omega$.
	Let us identify the critical resistance of the segment $\omega^{(i)}\isdef\omega(k_{i-1})
	\to\omega(k_{i-1}+1)\to\cdots\to\omega(k_i)$ corresponding to $A_{i-1}\to A_i$.
	Set $s\isdef\max\{\abs{A_{i-1}},\abs{A_i}\}$. By symmetry, we can assume that 
	$A_i\supseteq A_{i-1}$, in which case $s=\abs{A_i}$.
	
	Based on whether $N(A_i)\subseteq N(A_{i-1})$ or not, we have two possibilities:
	\begin{align}
	\label{eq:progression:path:proof:cases}
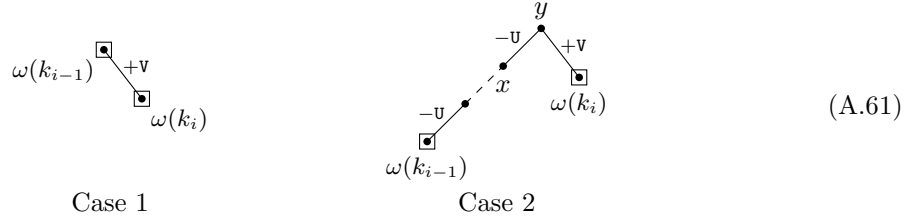
		\begin{array}{ccc}
			{		
			\tikzsetfigurename{progression_path_segment_A}
			\begin{tikzpicture}[baseline=(current bounding box.center),scale=0.5,>=stealth,shorten >=1]
				\PathSegmentIsoperimetricA
			\end{tikzpicture}
			}
			& \qquad\qquad &
			{
			\tikzsetfigurename{progression_path_segment_B}
			\begin{tikzpicture}[baseline=(current bounding box.center),scale=0.5,>=stealth,shorten >=1]
				\PathSegmentIsoperimetricB
			\end{tikzpicture}
			}
			\\
			\text{Case 1} & & \text{Case 2}
		\end{array}
	\end{align}
	When $N(A_i)\not\subseteq N(A_{i-1})$ (Case~2), we have
	\begin{align}
		\Psi(\omega^{(i)})= r(x,y) = \frac{\gamma}{\pi(x)} &=
			\frac{\gamma}{\pi(u)}\frac{\lambda^{\Delta(x)+\abs{x_V}}}{\bar{\lambda}^{\abs{x_V}}}
	\end{align}
	Observing that $\abs{x_V}=s-1$ and $\Delta(x)=\Delta\big(\omega(k_i)\big)=\Delta(s)$,
	we get
	\begin{align}
	\label{eq:progression:path:proof:case:2:resistance}
		\Psi(\omega^{(i)}) =
		r(x,y) &=
			\frac{\gamma}{\pi(u)}\frac{\lambda^{\Delta(s)+s-1}}{\bar{\lambda}^{s-1}}
		= \frac{\gamma}{\pi(u)} \lambda^{\Delta(s) - \alpha(s-1) + \smallo(1)} \;.
	\end{align}
	When $N(A_i)\subseteq N(A_{i-1})$ (Case~1),
	we similarly get
	\begin{align}
		\Psi(\omega^{(i)})= r\big(\omega(k_{i-1}),\omega(k_i)\big)
		&=
			\frac{\gamma}{\pi(u)}\frac{\lambda^{\Delta(s)+s}}{\bar{\lambda}^s} \nonumber \\
		&\prec
			\frac{\gamma}{\pi(u)}\frac{\lambda^{\Delta(s)+s-1}}{\bar{\lambda}^{s-1}}
		= \frac{\gamma}{\pi(u)} \lambda^{\Delta(s) - \alpha(s-1) + \smallo(1)} \;.
	\end{align}
	Letting $i$ run over $\{1,2,\ldots,n\}$ and maximizing $g(s)=\Delta(s)-\alpha(s-1)$,
	we find that
	\begin{align}
	\label{eq:progression:path:proof:sup}
		\Psi(\omega) = \sup_i\Psi(\omega^{(i)})&\leq
			\frac{\gamma}{\pi(u)}\frac{\lambda^{\Delta(s^\dagger)+s^\dagger-1}}{\bar{\lambda}^{s^\dagger-1}}
			= \frac{\gamma}{\pi(u)} \lambda^{\Delta(s^\dagger) - \alpha(s^\dagger-1) + \smallo(1)} \;.			
	\end{align}

	Next, suppose that the progression is nested,
	that is, $A_0\subsetneq A_1\subsetneq\cdots\subsetneq A_m$.
	Observe that the segment $\omega^{(i)}$ that achieves the
	maximum of $\Psi(\omega^{(i)})$ cannot be
	of the first type in~\eqref{eq:progression:path:proof:cases},
	unless $i=1$ and $N(A_1)\subseteq N(A_0)$.
	So, assume that $N(A_1)\subseteq N(A_0)$,
	and let $s^\dagger$ be a maximiser of $g(s)$ over
	$\{\abs{A_i}: 0<i\leq m\}=\{s_{\min}+1,\ldots,s_{\max}\}$.
	Let $0<i\leq m$ be such that $s^\dagger=\abs{A_i}$.
	Since $\abs{A_i}>\abs{A_{i-1}}$,
	we find from~\eqref{eq:progression:path:proof:case:2:resistance} that
	\begin{align}
		\Psi(\omega^{(i)}) &=
			\frac{\gamma}{\pi(u)}\frac{\lambda^{\Delta(s^\dagger)+s^\dagger-1}}{\bar{\lambda}^{s^\dagger-1}}
		= \frac{\gamma}{\pi(u)} \lambda^{\Delta(s^\dagger) - \alpha(s^\dagger-1) + \smallo(1)} \;.
	\end{align}
	This means that the equality in~\eqref{eq:progression:path:proof:sup} is achieved.
\end{proof}

\begin{proof}[Proof of Lemma~\ref{lem:hard-core:standard-path:optimality}]
	Let $\omega\isdef\omega(0)\to\omega(1)\to\cdots\to\omega(n)$ be a standard path
	and $A_0\subsetneq A_1\subsetneq\cdots\subsetneq A_m$
	the associated nested isoperimetric progression.
	Let $\sigma:\omega(0)\pathto\omega(n)$ be an arbitrary path
	from $\omega(0)$ to $\omega(n)$.
	We show that $\Psi(\sigma)\succeq\Psi(\omega)$ as $\lambda\to\infty$.
	It would then follow that $\Psi\big(\omega(0),\omega(n)\big)\asymp\Psi(\omega)$,
	that is, $\omega$ is optimal.
	
	For $0<i\leq m$,
	let $\omega^{(i)}\isdef\omega(k_{i-1})\to\omega(k_{i-1}+1)\to\cdots\to\omega(k_i)$
	be the segment of $\omega$ corresponding to $A_{i-1}\to A_i$.
	As observed in the proof of Lemma~\ref{lem:hard-core:progression:critical-resistance},
	the segment $\omega^{(i)}$ has one of the two forms in~\eqref{eq:progression:path:proof:cases}.
	Let $s\isdef\abs{A_i}$.
	We show that $\Psi(\sigma)\succeq\Psi(\omega^{(i)})$.
	
	When following $\sigma$, the number of particles on $V$ goes from
	$\abs{A_0}$ to $\abs{A_m}$, each step having at most one more particle on $V$
	than the previous step.  Therefore, there are configurations on $\sigma$
	that have exactly $s$ particles on $V$.
	Let $\sigma(\ell)$ be the first configuration on $\sigma$
	with $s$ particles on $V$.  Since $s>\abs{A_0}$, we have $\ell\geq 1$
	and the transition $\sigma(\ell-1)\to\sigma(\ell)$
	is of the type $\symb{+V}$ (i.e., adding a particle on $V$).
	
	When segment $\omega^{(i)}$ satisfies
	Case~1 of~\eqref{eq:progression:path:proof:cases},
	the resistance of the transition $\sigma(\ell-1)\to\sigma(\ell)$
	is at least as large as the critical resistance of $\omega^{(i)}$,
	because
	\begin{align}
		r\big(\sigma(\ell-1),\sigma(\ell)\big) =
			\frac{\gamma}{\pi(\sigma(\ell))}
		&\succeq
			\frac{\gamma}{\pi(\omega(k_i))} =
			r\big(\omega(k_{i-1}),\omega(k_i)\big) = \Psi(\omega^{(i)}) \;,
	\end{align}
	where the inequality follows from the isoperimetric optimality of $\omega(k_i)$.
	
	On the other hand,
	when $\omega^{(i)}$ satisfies Case~2 of~\eqref{eq:progression:path:proof:cases},
	we have $N(A_i)\not\subseteq N(A_{i-1})$, hence $\ell\geq 2$.
	There are two 
	possibilities for the transition $\sigma(\ell-2)\to\sigma(\ell-1)$:
	\begin{align}
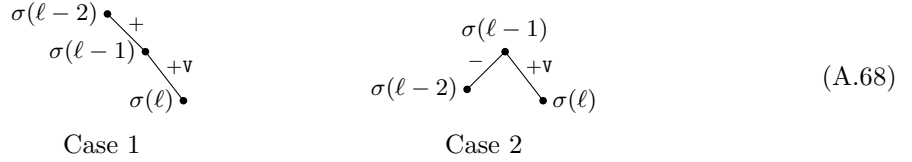
		\begin{array}{ccc}
			{	
			\tikzsetfigurename{standard_path_optimal_A}
			\begin{tikzpicture}[scale=0.5,>=stealth,shorten >=1]
				\coordinate[label=left:{\small $\sigma(\ell-2)$}] (l-2) at (0,2);
				\coordinate[label=left:{\small $\sigma(\ell-1)$}] (l-1) at (1,1);
				\coordinate[label=left:{\small $\sigma(\ell)$}] (l) at (2,-0.3);
				\draw
				(l-2) -- node[above right=-4pt] {\scriptsize $\symb{+}$}
				(l-1) -- node[above right=-4pt] {\scriptsize $\symb{+V}$}
				(l); \foreach \point in {l-2,l-1,l} \fill (\point) circle (0.1);
			\end{tikzpicture}
			}
			& \qquad\qquad &
			{	
			\tikzsetfigurename{standard_path_optimal_B}
			\begin{tikzpicture}[scale=0.5,>=stealth,shorten >=1]
				\coordinate[label=left:{\small $\sigma(\ell-2)$}] (l-2) at (0,0);
				\coordinate[label=above:{\small $\sigma(\ell-1)$}] (l-1) at (1,1);
				\coordinate[label=right:{\small $\sigma(\ell)$}] (l) at (2,-0.3);
				\draw
				(l-2) -- node[above left=-4pt] {\scriptsize $\symb{-}$}
				(l-1) -- node[above right=-4pt] {\scriptsize $\symb{+V}$}
				(l); \foreach \point in {l-2,l-1,l} \fill (\point) circle (0.1);
			\end{tikzpicture}
			}
			\\
			\text{Case 1} & & \text{Case 2}
		\end{array}
	\end{align}
		
	The first case is when $\sigma(\ell-1)$ is obtained from $\sigma(\ell-2)$ by adding 
	a particle. Then the resistance of the transition $\sigma(\ell-2)\to\sigma(\ell-1)$ is 
	strictly larger than $\Psi(\omega^{(i)})=r(x,y)$. Namely,
	\begin{align}
		r(\sigma(\ell-2),\sigma(\ell-1)) = \frac{\gamma}{\pi(\sigma(\ell-1))}
			&= \frac{\gamma}{\pi(\sigma(\ell))}\bar{\lambda}
			\succ \frac{\gamma}{\pi(\omega(k_i))}\frac{\bar{\lambda}}{\lambda}
			= \frac{\gamma}{\pi(x)} = r(x,y) \;.
	\end{align}
	The second case is when $\sigma(\ell-1)$ is obtained from $\sigma(\ell-2)$ by 
	removing a particle. This particle must be removed from $U$, for otherwise $\sigma(\ell-2)$ 
	would already have $s$ particles on $V$, which contradicts the choice of $\sigma(\ell)$.
	In this case the resistance of the transition $\sigma(\ell-2)\to\sigma(\ell-1)$ is still
	no smaller than $\Psi(\omega^{(i)})=r(x,y)$, because
	\begin{align}
		r(\sigma(\ell-2),\sigma(\ell-1)) = \frac{\gamma}{\pi(\sigma(\ell-2))}
			&= \frac{\gamma}{\pi(\sigma(\ell))}\frac{\bar{\lambda}}{\lambda} 
			\succeq \frac{\gamma}{\pi(\omega(k_i))}\frac{\bar{\lambda}}{\lambda}
			= \frac{\gamma}{\pi(x)} = r(x,y) \;.
	\end{align}
	Thus, in both cases we get that the critical resistance of $\sigma$ is at least
	$\Psi(\omega^{(i)})=r(x,y)$.
	
	In conclusion, $\Psi(\sigma)\succeq\Psi(\omega^{(i)})$.
	Running $i$ over $\{1,2,\ldots,m\}$, we find that
	$\Psi(\sigma)\succeq\Psi(\omega)$.  Since $\sigma$ was arbitrary, $\omega$ is optimal.
\end{proof}

\begin{proof}[Proof of Proposition~\ref{prop:hard-core:critical-resistance:identification}]
	Since the graph is connected and $V\neq\varnothing$,
	the neighbourhood $N(a)$ of every site $a\in V$ is non-empty.
	The claim thus follows immediately from Lemmas~\ref{lem:hard-core:progression:critical-resistance}
	and~\ref{lem:hard-core:standard-path:optimality}. 
\end{proof}


\subsection{No-trap condition via ordering}
\label{apx:no-trap}

\begin{proof}[Proof of Proposition~\ref{prop:hard-core:no-trap}]
Consider $x\notin\{u,v\}$. Let $i\in U$ and $j\in V$ be two adjacent sites that are not occupied 
in $x$. Such sites exist. Indeed, $N(U\setminus x_U)\not\subseteq x_V$, otherwise the graph 
would not be connected. By assumption, there is a standard path $u=\omega(0)\to\omega(1)
\to\cdots\to\omega(m)\in J(u)$ whose first particle on $V$ is on site $j$. Note that this path starts 
by removing particles from neighbours of $j$ until it is possible to place a particle on site $j$. 
Since re-ordering the removal of these particles from $U$ does not affect the condition of being a 
standard path, we may assume that the first particle to be removed is from site $i$.

We construct a path $\sigma\colon\,x\pathto y$ from $x$ to a configuration $y\in J^-(x)$ that 
verifies the claim $\pi(x)\Psi\big(x,J^-(x)\big)\prec\pi(u)\Psi\big(u,J(u)\big)$. The idea is to follow 
the moves of the path $\omega$. Specifically, for $k=0,1,\ldots,m$, define $\sigma'(k)\isdef 
x\lor\omega(k)$. The sequence $x=\sigma'(0),\sigma'(1),\ldots,\sigma'(m)$ potentially has 
repeated elements. For instance, $\sigma'(1)=\sigma'(0)$ because $x$ has no particle on $i$.
Removing the repeated elements from this sequence, we obtain a path $\sigma(0)\to\sigma(1)
\to\cdots\to\sigma(\bar{m})$, which we claim has the right property. Observe that this indeed 
makes a path: $\sigma'(k)$ and $\sigma'(k+1)$ differ in at most one position.
	
We will verify that 
\begin{enumerate}[label=(\roman*)]
\item 
$\pi(\sigma'(m))\succ \pi(x)$, 
\item 
$\pi(x)r(\sigma'(k),\sigma'(k+1))\prec\pi(u)r\big(\omega(k),\omega(k+1)\big)$
for each $k=0,1,\ldots,m-1$ such that $\sigma'(k)$ and $\sigma'(k+1)$ are not the same.
\end{enumerate}
Claim~(i) means that $y\isdef\sigma'(m)$ is in $J^-(x)$.  Claim~(ii) implies that
\begin{align}
\pi(x)\Psi\big(x,J^-(x)\big) 
&\preceq \sup_{0\leq\ell<\bar{m}} \pi(x)r\big(\sigma(\ell),\sigma(\ell+1)\big)\nonumber\\
&\prec \sup_{0\leq k<m} \pi(u)r\big(\omega(k),\omega(k+1)\big) = \pi(u)\Psi\big(u,J(u)\big)
\end{align}
as $\lambda\to\infty$, which proves the proposition.
	
To verify the above claims, we note that
\begin{align}
r\big(\omega(k),\omega(k+1)\big) 
= \frac{\gamma}{\max\big\{\pi(\omega(k)),\pi(\omega(k+1))\big\}} 
&= \begin{cases}
\frac{\gamma}{\pi(\omega(k))},
&\text{if $\omega(k)\xremove[U]\omega(k+1)$,} \\
\frac{\gamma}{\pi(\omega(k+1))},
&\text{if $\omega(k)\xadd[V]\omega(k+1)$,}
\end{cases} \\
r\big(\sigma'(k),\sigma'(k+1)\big) 
= \frac{\gamma}{\max\big\{\pi(\sigma'(k)),\pi(\sigma'(k+1))\big\}} 
&= \begin{cases}
\frac{\gamma}{\pi(\sigma'(k))}
&\text{if $\omega(k)\xremove[U]\omega(k+1)$,} \\
\frac{\gamma}{\pi(\sigma'(k+1))}
&\text{if $\omega(k)\xadd[V]\omega(k+1)$}
\end{cases}
\end{align}
provided $\sigma'(k)$ and $\sigma'(k+1)$ are not the same.
Claim~(ii) boils down to verifying that
\begin{align}
\label{eq:no-trap:proof:inequality}
\frac{\pi(u)}{\pi(\omega(k))} 
&\succ \frac{\pi(x)}{\pi(\sigma'(k))} = \frac{\pi(x)}{\pi(x\lor \omega(k))}
\end{align}
for $k=1,2,\ldots,m$ (recall: $\sigma'(0)=\sigma'(1)$).
The same inequality for $k=m$ also proves Claim~(i), because 
$\frac{\pi(u)}{\pi(\omega(m))}\preceq 1$. Finally, using the identity
\begin{align}
\pi(x\lor \omega(k))\pi(x\land \omega(k)) &= \pi(x)\pi(\omega(k))
\end{align}
(see Section~\ref{sec:hard-core:ordering}), the proof of the inequality in 
\eqref{eq:no-trap:proof:inequality} reduces to the proof of the following claim.
The configuration $x\land\omega(k)$ roughly keeps track of the moves that we are 
``saving'' by starting the path from $x$ rather than $u$.
	
\begin{claim}
For $k=1,2,\ldots,m$, $\pi(x\land \omega(k))\prec \pi(u)$ as $\lambda\to\infty$.
\end{claim}
\begin{argument}
Let $s$ denote the number of particles that $x\land \omega(k)$ has on $V$. We consider 
three separate cases.
		
\begin{case}[1]{$s=0$.}\\
The configuration $x\land \omega(k)$ has no particle on $V$. Moreover, the choice of 
$\omega$ ensures that $x\land \omega(k)$ has no particle on site $i\in U$. It immediately 
follows that $\pi(x\land \omega(k))\prec \pi(u)$.
\end{case}
		
\begin{case}[2]{$0<s<\abs{\omega_V(m)}$.}\\
Let $k_1$ be the first integer for which $\omega(k_1)$ has $s$ particles on $V$. Since 
$\omega$ is a standard path, $\omega(k_1)$ is isoperimetrically optimal. Therefore 
\begin{align}
\pi(x\land \omega(k)) &= \pi(u)\bar{\lambda}^s\lambda^{-\Delta(x\land \omega(k))-s} 
\preceq \pi(u)\bar{\lambda}^s\lambda^{-\Delta(s)-s} 
= \pi(\omega(k_1)) \prec \pi(u).
\end{align}
(For the latter inequality, recall that $\omega(k_1)\notin J(u)$.)
\end{case}
		
\begin{case}[3]{$s=\abs{\omega_V(m)}$.}\\
This is impossible. Indeed, every particle that $x\land \omega(k)$ has on $V$ is also 
present in $\omega(m)$. But, by the choice $\omega$, $\omega(m)$ has a particle on 
site $j\in V$ on which $x$ has no particle. Therefore $x\land \omega(k)$ has strictly less 
particles on $V$ than $\omega(m)$.
\end{case}
\end{argument}
This concludes the proof.
\end{proof}


\subsection{Passing the bottleneck}
\label{apx:bottleneck}

\begin{proof}[Proof of Lemma~\ref{lem:hard-core:bipartite:optimal-path:isoperimetric:1}]
Suppose that $x$ is a basic configuration in $\omega$ with $\abs{x_V}=s$ particles on $V$.
By the remark before the lemma, we have
\begin{align}
\Psi(\omega) \geq \frac{\gamma}{\pi(x)}
&= \frac{\gamma}{\pi(u)}\frac{\lambda^{\abs{U\setminus x_U}}}{\bar{\lambda}^{\abs{x_V}}}
= \frac{\gamma}{\pi(u)} \lambda^{\Delta(x)-\alpha\abs{x_V}+\smallo(1)}.
\end{align}
Writing $\Delta(x)=\Delta(s)+\Delta(x)-\Delta(s)$, $\Delta(s)=\Delta(s^*)+\dd\Delta(s)$ and 
$s=s^*+\dd s$, and using the assumption we obtain
\begin{align}
\Psi(\omega) &\geq \frac{\gamma}{\pi(u)}\lambda^{\Delta(s^*)-\alpha(s^*-1)+\smallo(1)}
\lambda^{\Delta(x)-\Delta(s)+\dd\Delta(s)-\alpha(\dd s + 1)}\nonumber\\
&\geq \frac{\gamma}{\pi(u)} \lambda^{\Delta(s^*)-\alpha(s^*-1)+\smallo(1)}
\lambda^{\Delta(x)-\Delta(s)-\varepsilon}.
\end{align}
On the other hand, since $\omega$ is optimal, we know by 
Lemma~\ref{prop:hard-core:critical-resistance:identification} that
\begin{align}
\Psi(\omega)=\Psi\big(u,J(u)\big) 
&= \frac{\gamma}{\pi(u)} \lambda^{\Delta(s^*)-\alpha(s^*-1)+\smallo(1)}.
\end{align}
It follows that $\Delta(x)-\Delta(s)-\varepsilon\leq 0$, i.e., $x$ is $\varepsilon$-optimal.
To see the latter claim, note that any configuration that is $\varepsilon$-optimal for some 
$\varepsilon<1$ is, in fact, optimal.
\end{proof}

\begin{proof}[Proof of Lemma~\ref{lem:hard-core:bipartite:optimal-path:isoperimetric:2}]
Suppose that $x=\omega(k)$ is the first configuration in $\omega$ with $\abs{x_V}=s+1$ 
particles on $V$.  This means that $\omega(k)$ has one more particle on $V$ compared 
to $\omega(k-1)$.
Observe that $k$ must be at least $2$ for otherwise we get a contradiction with
the connectedness of the graph.
There are two possibilities for $\omega(k-2)$:
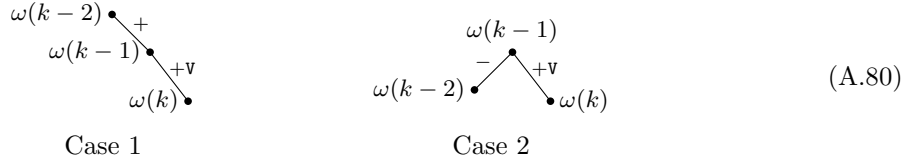
\begin{align}
\begin{array}{ccc}
{			
\tikzsetfigurename{optimal_path_top_A}
\begin{tikzpicture}[scale=0.5,>=stealth,shorten >=1]
\coordinate[label=left:{\small $\omega(k-2)$}] (k-2) at (0,2);
\coordinate[label=left:{\small $\omega(k-1)$}] (k-1) at (1,1);
\coordinate[label=left:{\small $\omega(k)$}] (k) at (2,-0.3);
\draw
(k-2) -- node[above right=-4pt] {\scriptsize $\symb{+}$}
(k-1) -- node[above right=-5pt] {\scriptsize $\symb{+V}$}
(k); \foreach \point in {k-2,k-1,k} \fill (\point) circle (0.1);
\end{tikzpicture}
}
& \qquad\qquad &
{
\tikzsetfigurename{optimal_path_top_B}
\begin{tikzpicture}[scale=0.5,>=stealth,shorten >=1]
\coordinate[label=left:{\small $\omega(k-2)$}] (k-2) at (0,0);
\coordinate[label=above:{\small $\omega(k-1)$}] (k-1) at (1,1);
\coordinate[label=right:{\small $\omega(k)$}] (k) at (2,-0.3);
\draw
(k-2) -- node[above left=-4pt] {\scriptsize $\symb{-}$}
(k-1) -- node[above right=-5pt] {\scriptsize $\symb{+V}$}
(k); \foreach \point in {k-2,k-1,k} \fill (\point) circle (0.1);
\end{tikzpicture}
}
\\
\text{Case 1} & & \text{Case 2}
\end{array}
\end{align}

In the first case, $\omega(k-2)$ has one less particle than $\omega(k-1)$. This means 
that $\omega(k-1)$ is a basic step on $\omega$. Since $\omega(k-1)$ has $s$ particles 
on $V$ and $\dd \Delta(s)+\varepsilon\geq\alpha(\dd s+1)$,
Lemma~\ref{lem:hard-core:bipartite:optimal-path:isoperimetric:1} implies that $\omega(k-1)$ 
is isoperimetrically $\varepsilon$-optimal. Therefore, $\Delta(\omega(k))=\Delta(\omega(k-1))-1
\leq\Delta(s)+\varepsilon-1$. Using the assumption $\Delta(s+1)\geq\Delta(s)$, we obtain 
that $x=\omega(k)$ is $(\varepsilon-1)$-optimal.
	
In the second case, $\omega(k-2)$ has one more particle than $\omega(k-1)$. By the choice 
of $x=\omega(k)$, this extra particle is on $U$. Otherwise $\omega(k-2)$ would already 
have $s+1$ particles on $V$. Now, $\omega(k-2)$ is a basic configuration on $\omega$
with $s$ particles on $V$. Therefore the assumption $\dd \Delta(s)+\varepsilon\geq\alpha
(\dd s+1)$ and Lemma~\ref{lem:hard-core:bipartite:optimal-path:isoperimetric:1} imply that
$\omega(k-2)$ is isoperimetrically $\varepsilon$-optimal. Therefore $\Delta(\omega(k))
=\Delta(\omega(k-2))\leq\Delta(s)+\varepsilon$. By assumption, we also have $\Delta(s+1)
\geq\Delta(s)$. Hence $\Delta(\omega(k))\leq\Delta(s+1)+\varepsilon$, which means that
$\omega(k)$ is $\varepsilon$-optimal.
\end{proof}

\begin{proof}[Proof of Lemma~\ref{lem:hard-core:bipartite:optimal-path:isoperimetric:3}]
Let $\omega(i)$ be the next basic configuration after $\omega(p)$. Since $\omega(p)$ is 
the last basic configuration before $\omega(q)$ having less than $s^*-1$ particles on~$V$, 
$\omega(i)$ must have $s^*-1$ particles on~$V$. We have either of the following two 
possibilities when going from $\omega(p)$ to $\omega(i)$ on $\omega$:
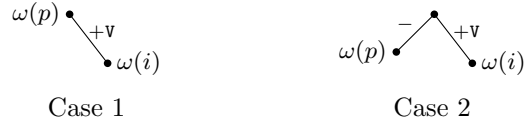
\begin{equation*}
\begin{array}{ccc}
{
\tikzsetfigurename{optimal_path_top_C}
\begin{tikzpicture}[scale=0.5,>=stealth,shorten >=1]
\coordinate[label=left:{\small $\omega(p)$}] (p) at (0,1);
\coordinate[label=right:{\small $\omega(i)$}] (i) at (1,-0.3);
\draw (p) -- node[above right=-5pt] {\scriptsize $\symb{+V}$} (i);
\foreach \point in {p,i} \fill (\point) circle (0.1);
\end{tikzpicture}
}
& \qquad\qquad &
{			
\tikzsetfigurename{optimal_path_top_D}
\begin{tikzpicture}[scale=0.5,>=stealth,shorten >=1]
\coordinate[label=left:{\small $\omega(p)$}] (p) at (0,0);
\coordinate (pp) at (1,1);
\coordinate[label=right:{\small $\omega(i)$}] (i) at (2,-0.3);
\draw
(p) -- node[above left=-4pt] {\scriptsize $\symb{-}$}
(pp) -- node[above right=-5pt] {\scriptsize $\symb{+V}$}(i);
\foreach \point in {p,pp,i}
\fill (\point) circle (0.1);
\end{tikzpicture}
}
\\
\text{Case 1} & & \text{Case 2}
\end{array}
\end{equation*}
Suppose that $\omega(i)$ has $t=t^*+\dd t$ particles on $U$. Then $\omega(p)$ has at most 
$t+1$ particles on $U$.
	
By the remark before Lemma~\ref{lem:hard-core:bipartite:optimal-path:isoperimetric:1},
the critical resistance of $\omega$ satisfies
\begin{align}
\Psi(\omega) \geq \frac{\gamma}{\pi(\omega(p))}
&= \frac{\gamma}{\pi(u)}
\frac{\lambda^{\abs{U\setminus \omega_U(p)}}}{\bar{\lambda}^{\abs{\omega_V(p)}}}
= \frac{\gamma}{\pi(u)}
\lambda^{\Delta(\omega(p))-\alpha\abs{\omega_V(p)}+\smallo(1)}
\end{align}
as $\lambda\to\infty$. Substituting $\Delta(\omega(p))\geq\abs{U}-(t^*+\dd t+1)-(s^*-2)
=\Delta(s^*)-\dd t+1$ and $\abs{\omega_V(p)}=s^*-2$, we get
\begin{align}
\Psi(\omega) &\geq \frac{\gamma}{\pi(u)}
\lambda^{\Delta(s^*)-\alpha(s^*-1)+1+\alpha-\dd t + \smallo(1)}
\end{align}
as $\lambda\to\infty$. But, since $\omega$ is optimal, we know from 
Lemma~\ref{prop:hard-core:critical-resistance:identification} that
\begin{align}
\Psi(\omega)=\Psi\big(u,J(u)\big) 
&= \frac{\gamma}{\pi(u)} \lambda^{\Delta(s^*)-\alpha(s^*-1)+\smallo(1)}.
\end{align}
It follows that $\dd t\geq 1+\alpha$. Since $\dd t$ is integer, $\dd t\geq 2$, 
which proves the first claim. In particular,
\begin{align}
\Delta(\omega(i)) &= \abs{U}-(t^*+\dd t) - (s^*-1) = \Delta(s^*)+1-\dd t
\leq \Delta(s^*)-1 = \Delta(s^*-1) + \delta - 1,
\end{align}
which means $\omega(i)$ is $(\delta-1)$-optimal.
\end{proof}

\begin{proof}[Proof of Proposition~\ref{prop:hard-core:bipartite:optimal-path:top-of-the-hill}]
Let $\omega$ be an optimal path from $u$ to $J(u)$.
By definition, $s^*\geq 1$.  Let us first assume that $s^*\geq 2$.
Let $\omega(q)$ be the first basic 
configuration in $\omega$ that has $s^*+\kappa$ particles on $V$. Let $\omega(p)$ (with $p<q$) 
be the last basic configuration before $\omega(q)$ with $s^*-2$ particles on $V$. Finally, 
let $\omega(r)$ (with $p<r<q$) be the last (not necessarily basic) configuration before 
$\omega(q)$ having $s^*-1$ particles on $V$. Set $y\isdef\omega(r)$, $x\isdef\omega(r-1)$ 
and $z\isdef\omega(r+1)$.
	
Clearly, $z$ is a basic configuration with $s^*$ particles on $V$ and $y\xadd[V]z$. By 
Lemmas~\ref{lem:hard-core:bipartite:optimal-path:isoperimetric:1} and
\ref{lem:hard-core:bipartite:optimal-path:isoperimetric:2}, every basic configuration in the 
segment $\omega(r+1)\to\omega(r+2)\to\cdots\to\omega(q)$ is isoperimetrically optimal.
Let $\omega(k_0),\omega(k_1),\ldots,\omega(k_\ell)$ be the subsequence of these
basic configurations obtained after removing the repetitions, and set $B_i\isdef\omega_V(k_i)$.
The sequence $B_0,B_1,\ldots,B_\ell$ satisfies the required properties in part (c).
	
Let $\omega(t)$ (with $p<t\leq r<q$) be the first basic configuration after $\omega(p)$.
Then $\omega(t)$ has $s^*-1$ particles on $V$ and, according to 
Lemma~\ref{lem:hard-core:bipartite:optimal-path:isoperimetric:3},
is isoperimetrically $(\delta-1)$-optimal.
Note that every configuration $\omega(i)$ with $t\leq i\leq r$ has exactly $s^*-1$ particles 
on $V$. Therefore $\omega_V(i)=\omega_V(t)$ is an isoperimetrically $(\delta-1)$-optimal set.
In particular, $x_V$ is isoperimetrically $(\delta-1)$-optimal.  We argue 
that $x\xremove[U]y$. Indeed, otherwise, we would have $x\xadd y$, which means that $y$ is 
a basic configuration with $\Delta(y)=\Delta(s^*)+1$. It would then follow that
\begin{align}
\Psi(\omega) \geq r(x,y) = \frac{\gamma}{\pi(y)}
&= \frac{\gamma}{\pi(u)}\lambda^{\Delta(y) - \alpha\abs{y_V} + \smallo(1)}\nonumber\\
&= \frac{\gamma}{\pi(u)}\lambda^{\Delta(s^*) + 1 - \alpha(s^*-1) + \smallo(1)} 
\succ \Psi\big(u,J(u)\big),
\end{align}
where the latter inequality follows from Proposition~\ref{prop:hard-core:critical-resistance:identification}.
This contradicts the optimality of $\omega$. Since $x\xremove[U]y\xadd[V]z$ and $z$ is 
isoperimetrically optimal, we also get $\Delta(x)=\Delta(z)=\Delta(s^*)=\Delta(s^*-1)+\delta$,
which means $x$ is isoperimetrically $\delta$-optimal.

If $s^*=1$, we set $t\isdef 0$ and choose $r$ ($t\leq r<q$) to be the last configuration before
$\omega(q)$ having no particle on $V$.  Note that $r>t$ for otherwise the graph will not be connected.
In this case, $\omega(t)=u$ is optimal and the rest of the argument goes without change.
\end{proof}


\subsection{Identification of critical gate}
\label{apx:critical-gate:identification}

\begin{proof}[Proof of Proposition~\ref{prop:critical-gate:progressions}]
	Using Hypothesis~\eqref{hypothesis:numbering:single} and
	Proposition~\ref{prop:hard-core:critical-resistance:identification},
	we have
	\begin{align}
		\Psi\big(u,J(u)\big) 
		&= \frac{\gamma}{\pi(u)}
		\frac{\lambda^{\Delta(s^*)+s^*-1}}{\bar{\lambda}^{s^*-1}} \;.
		= \frac{\gamma}{\pi(u)} \lambda^{\Delta(s^*) - \alpha(s^*-1) + \smallo(1)}
		\qquad\text{as $\lambda\to\infty$.}
	\end{align}
	We verify that the pair $(Q,Q^*)$ satisfies the four conditions
	for being a critical pair (Sec.~\ref{sec:sharper})
	between $A\isdef\{u\}$ and $B\isdef J(u)$.

	First, observe that for every $x\in Q$ and $y\in Q^*$ with $x\link y$,
	we have
	\begin{align}
		r(x,y) &= \frac{\gamma}{\pi(x)}
			= \frac{\gamma}{\pi(u)}
				\frac{\lambda^{\Delta(s^*)+s^*-1}}{\bar{\lambda}^{s^*-1}}
			= \Psi\big(u,J(u)\big) \;,
	\end{align}
	hence the first condition is satisfied.

	Let $x\in Q$ and $y\in Q^*$ be such that $x\xremove[U] y$.
	By definition, there exist $A\in\family{A}$ and $B\in\family{B}$
	such that $y_V=A$ and $y_U=U\setminus N(B)$.
	Let $i$ be the unique element of $B\setminus A$.
	Then, $x_V=y_V$ and $x_U=y_U\cup\{j_0\}$ for some $j_0\in N(i)\setminus N(A)$.
	(Note that $N(i)\setminus N(A)$ is non-empty. Otherwise,
	$\Delta(B)=\Delta(A)-1$, which gives $g(s^*-1)>g(s^*)$.
	The latter inequality clearly cannot happen if $s^*>1$.
	On the other hand, when $s^*=1$, the set $N(i)\setminus N(A)=N(i)$
	cannot be empty, because the graph is assumed to be connected.)
	Let $j_1,j_2,\ldots,j_d$ be an enumeration of $N(i)\setminus\{j_0\}$.
	
	According to~\eqref{hypothesis:critical-set:before-A},
	there is an isoperimetric progression from $\varnothing$ to $A$,
	consisting only of sets of size at most $s^*-1$.
	Let $\omega$ be the path associated to such a progression.
	Then, by Lemma~\ref{lem:hard-core:progression:critical-resistance},
	\begin{align}
		\Psi(\omega) &\leq
			\frac{\gamma}{\pi(u)}\frac{\lambda^{\Delta(s^\dagger)+s^\dagger-1}}{\bar{\lambda}^{s^\dagger-1}}
			= \frac{\gamma}{\pi(u)} \lambda^{\Delta(s^\dagger) - \alpha(s^\dagger-1) + \smallo(1)}
			\qquad\text{as $\lambda\to\infty$,}
	\end{align}
	where $s^\dagger$ is the maximiser of $g(s)$ over $\{1,2,\ldots,s^*-1\}$.
	Since $g(s^\dagger)<g(s^*)$, we find that $\Psi(\omega)\prec\Psi\big(u,J(u)\big)$.
	Let $\omega'$ be the path from $u$ to $x$, obtained by first following $\omega$
	and then removing particles from $j_1,j_2,\ldots,j_d$ one after another.
	The resistance of the new transitions are all smaller than $r(x,y)$.
	Therefore, $\Psi(u,x)\preceq\Psi(\omega')\prec\Psi\big(u,J(u)\big)$.
	
	Showing that $\Psi\big(y,J(u)\big)\prec\Psi\big(u,J(u)\big)$
	is similar.  According to~\eqref{hypothesis:critical-set:after-B},
	there is an isoperimetric progression from $B$ to a set of size $\tilde{s}$,
	consisting only of sets of size at least $s^*$.
	Let $\omega$ be the path associated to such a progression.
	Then, by Lemma~\ref{lem:hard-core:progression:critical-resistance},
	\begin{align}
		\Psi(\omega) &\leq
			\frac{\gamma}{\pi(u)}\frac{\lambda^{\Delta(s^\dagger)+s^\dagger-1}}{\bar{\lambda}^{s^\dagger-1}}
			= \frac{\gamma}{\pi(u)} \lambda^{\Delta(s^\dagger) - \alpha(s^\dagger-1) + \smallo(1)}
			\qquad\text{as $\lambda\to\infty$,}
	\end{align}
	where $s^\dagger$ is the maximiser of $g(s)$ over $\{s^*+1,s^*+2,\ldots,\tilde{s}\}$.
	Using~\eqref{hypothesis:uniqueness} we know that $g(s^\dagger)<g(s^*)$,
	from which it follows that $\Psi(\omega)\prec\Psi\big(u,J(u)\big)$.
	Let $\omega'$ be the path from $y$ to $J(u)$, obtained by first
	placing a particle on $i$ and then following $\omega$.
	Note that
	\begin{align}
		r\big(\omega'(0),\omega'(1)\big) &=
			\frac{\gamma}{\pi\big(\omega(0)\big)}
			= \frac{\gamma}{\pi(u)}\frac{\lambda^{\Delta(s^*)+s^*}}{\bar{\lambda}^{s^*}}
			= \frac{\gamma}{\pi(u)} \lambda^{\Delta(s^*) - \alpha s^* + \smallo(1)}
			\prec r(x,y) \;.
	\end{align}
	Therefore, $\Psi(u,x)\preceq\Psi(\omega')\prec\Psi\big(u,J(u)\big)$.
	
	Lastly, let $\omega$ be an optimal path from $u$ to $J(u)$.
	The trace of $\omega$ on $V$ (see Section~\ref{sec:hard-core:progressions})
	is a progression $A_0,A_1,\ldots,A_m$
	with $A_0=\varnothing$ and $\Delta(A_m)\leq\alpha\abs{A_m}$. 
	Let us verify that this progression is $\alpha$-bounded in the sense that
	$\Delta(A_i)-\alpha\abs{A_i}\leq\Delta(s^*)-\alpha s^*$ for each $0\leq i\leq m$.
	Indeed, suppose that $\Delta(A_i)-\alpha(\abs{A_i}-1)>\Delta(s^*)-\alpha(s^*-1)=g(s^*)$
	for some $0\leq i\leq m$.
	From~\eqref{hypothesis:uniqueness}, we know that $i\geq 1$.
	Let $\omega(k)$ be the first configuration in $\omega$ such that $\omega_V(k)=A_i$.
	Since $i\geq 1$ and the graph is connected, we have $k\geq 2$.
	There are two possibilities for the two transitions leading to $\omega(k)$:
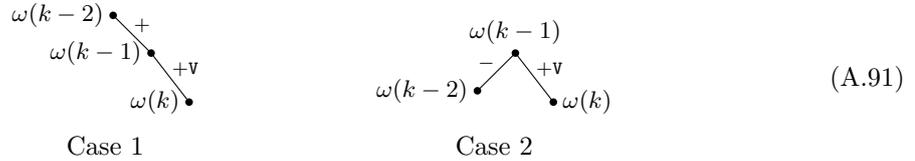
	\begin{align}
	\label{eq:critical-gate:progressions:cases}
		\begin{array}{ccc}
			{
			\tikzsetfigurename{optimal_path_progression_A}
			\begin{tikzpicture}[scale=0.5,>=stealth,shorten >=1]
				\coordinate[label=left:{\small $\omega(k-2)$}] (k-2) at (0,2);
				\coordinate[label=left:{\small $\omega(k-1)$}] (k-1) at (1,1);
				\coordinate[label=left:{\small $\omega(k)$}] (k) at (2,-0.3);
				\draw
				(k-2) -- node[above right=-4pt] {\scriptsize $\symb{+}$}
				(k-1) -- node[above right=-4pt] {\scriptsize $\symb{+V}$}
				(k); \foreach \point in {k-2,k-1,k} \fill (\point) circle (0.1);
			\end{tikzpicture}
			}
			& \qquad\qquad &
			{	
			\tikzsetfigurename{optimal_path_progression_B}
			\begin{tikzpicture}[scale=0.5,>=stealth,shorten >=1]
				\coordinate[label=left:{\small $\omega(k-2)$}] (k-2) at (0,0);
				\coordinate[label=above:{\small $\omega(k-1)$}] (k-1) at (1,1);
				\coordinate[label=right:{\small $\omega(k)$}] (k) at (2,-0.3);
				\draw
				(k-2) -- node[above left=-4pt] {\scriptsize $\symb{-}$}
				(k-1) -- node[above right=-4pt] {\scriptsize $\symb{+V}$}
				(k); \foreach \point in {k-2,k-1,k} \fill (\point) circle (0.1);
			\end{tikzpicture}
			}
			\\
			\text{Case 1} & & \text{Case 2}
		\end{array}
	\end{align}
	As in the proof of Lemma~\ref{lem:hard-core:standard-path:optimality},
	we can verify that in either case,
	\begin{align}
		r(\omega(k-2),\omega(k-1)) &\succeq
			\frac{\gamma}{\pi(u)} \lambda^{\Delta(A_i) - \alpha(\abs{A_i}-1) + \smallo(1)} \nonumber \\
		&\succ \frac{\gamma}{\pi(u)} \lambda^{\Delta(s^*) - \alpha(s^*-1) + \smallo(1)}
			\asymp \Psi\big(u,J(u)\big) \;,
	\end{align}
	which is a contradiction. Since $A_0,A_1,\ldots,A_m$ is $\alpha$-bounded,
	according to~\eqref{hypothesis:critical-set:mandatory},
	there exists an index $0<i\leq m$ such that $A_{i-1}\in\family{A}$ and $A_i\in\family{B}$.
	Let $\omega(k)$ be the first configuration in $\omega$ such that $\omega_V=A_i$.
	Then, the two transitions leading to $\omega(k)$
	are of the second type in~\eqref{eq:critical-gate:progressions:cases},
	where $x\isdef\omega(k-2)\in Q$ and $y\isdef\omega(k-1)\in Q^*$.
	
	We conclude that $(Q,Q^*)$ is a critical pair between $u$ and $J(u)$.
\end{proof}

\begin{proof}[Proof of Proposition~\ref{prop:critical-gate:identification}]
	Condition~\eqref{hypothesis:critical-set:size} is clearly satisfied.
	Condition~\eqref{hypothesis:critical-set:before-A} is the same
	as~\eqref{hypothesis:critical-set-replacement:before-A}.
	Condition~\eqref{hypothesis:critical-set:after-B} follows
	from~\eqref{hypothesis:critical-set-replacement:after-C} and the definition of~$\family{B}$.

	Condition~\eqref{hypothesis:critical-set:mandatory} follows from
	Proposition~\ref{prop:hard-core:bipartite:optimal-path:top-of-the-hill}.
	Namely, let $A_0,A_1,\ldots,A_n$ be an $\alpha$-bounded progression
	(i.e., a progression satisfying $\Delta(A_i)-\alpha\abs{A_i}\leq\Delta(s^*)-\alpha s^*$)
	with $A_0=\varnothing$ and $\Delta(A_n)\leq\alpha\abs{A_n}$. 
	Let $\omega$ be the path associated to this progression
	(see Section~\ref{sec:hard-core:progressions}).
	
	Since $\Delta(A_n)\leq\alpha\abs{A_n}$,
	the path ends at a configuration $\omega(N)\in J(u)$.
	Furthermore, as in the proof of Lemma~\ref{lem:hard-core:progression:critical-resistance}
	(and using Proposition~\ref{prop:hard-core:critical-resistance:identification}),
	we can verify that $\omega$ is optimal (in the sense of Section~\ref{sec:sharper}).
	Proposition~\ref{prop:hard-core:bipartite:optimal-path:top-of-the-hill}
	and~\eqref{hypothesis:critical-set-replacement:values}
	now ensure that there is a $0\leq k<n$ such that $A_k\in\family{A}$
	and $A_{k+1}\in\family{B}$.
\end{proof}


\subsection{Isoperimetric problems}
\label{apx:isoperimetric}

\begin{proof}[Proof of Lemma~\ref{lem:isoperimetric:lattice}]
The proof follows Cirillo and Nardi~\cite[Lemma 6.16]{CirNar13}. Let us refer to the 
two principal directions of the lattice $L$ as \emph{horizontal} and \emph{vertical}. We 
say that $A$ is \emph{convex} when its intersection with every horizontal or vertical line
induces a (connected) path in $L$. We first show that $A$ is convex and connected in $L$.
We afterwards verify that every finite convex and connected set satisfies $N_{\symb{1010}}(A)
=\varnothing$ and $\abs{N_1(A)} - \abs{N_3(A)} = 4$.
	
First, let us verify that $A\cup N(A)$ is connected in the original lattice. If not, then $A$ 
can be partitioned into two sets $A_1$ and $A_2$ such that $A_1\cup N(A_1)$ and 
$A_2\cup N(A_2)$ are disjoint. We can then shift $A_2$ to obtain a set $A'_2$ that is 
still disjoint from $A_1$, but satisfies  $N(A'_2)\cap N(A_1)\neq\varnothing$. It follows
that $\Delta(A_1\cup A'_2)<\Delta(A)$, which contradicts the optimality of $A$.
	
Next, let $\overline{A}\subseteq V$ be the smallest rectangular region in $L$ having 
horizontal and vertical sides that contains $A$.
Consider the following construction that enlarges $A$
(Fig.~\ref{fig:isoperimetric:lattice:lemma:extension}). Set $B_0\isdef A$. To construct 
$B_t$ from $B_{t-1}$, find a vertex $k_t\in V\setminus B_{t-1}$ that is adjacent in $L$
to at least two elements of $B_{t-1}$ and set $B_t\isdef B_{t-1}\cup\{k_t\}$.
It is easy to see that this construction stops precisely 
when $B_t=\overline{A}$. Furthermore, $\Delta(B_{t-1})\geq\Delta(B_t)$ with equality
if and only if $N_L(k_t)\cap B_{t-1}=\{i,j\}$ where $N(i)\cap N(j)\neq\varnothing$ (i.e., 
$B_{t-1}$ has exactly two elements adjacent in $L$ to $k_t$, and those two elements 
form a right triangle with $k_t$). The latter happens for every $t$ precisely when $A$ 
is convex. It follows that $\Delta(A)\geq\Delta(\overline{A})$ with equality if and only if 
$A$ is convex.
		
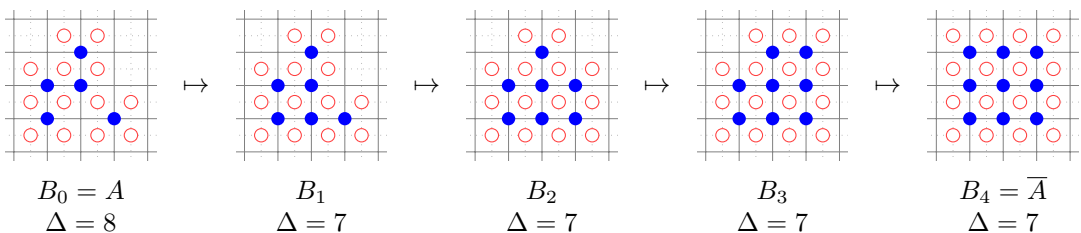
\begin{figure}[htb]
\centering
$\begin{array}{ccccccccc}
{
\tikzsetfigurename{lattice_isoperimetric_enlargement_A}
\begin{tikzpicture}[scale=0.22,baseline={([yshift=-.6ex]current bounding box.center)}]	\DrawLatticeIsoperimetricEnlargement{\LatticeIsoperimetricEnlargementA}
\end{tikzpicture}
}
& \mapsto &
{
\tikzsetfigurename{lattice_isoperimetric_enlargement_B}
\begin{tikzpicture}[scale=0.22,baseline={([yshift=-.6ex]current bounding box.center)}]	\DrawLatticeIsoperimetricEnlargement{\LatticeIsoperimetricEnlargementB}
\end{tikzpicture}
}
& \mapsto &
{
\tikzsetfigurename{lattice_isoperimetric_enlargement_C}
\begin{tikzpicture}[scale=0.22,baseline={([yshift=-.6ex]current bounding box.center)}]	\DrawLatticeIsoperimetricEnlargement{\LatticeIsoperimetricEnlargementC}
\end{tikzpicture}
}
& \mapsto &
{
\tikzsetfigurename{lattice_isoperimetric_enlargement_D}
\begin{tikzpicture}[scale=0.22,baseline={([yshift=-.6ex]current bounding box.center)}]	\DrawLatticeIsoperimetricEnlargement{\LatticeIsoperimetricEnlargementD}
\end{tikzpicture}
}
& \mapsto &
{
\tikzsetfigurename{lattice_isoperimetric_enlargement_E}
\begin{tikzpicture}[scale=0.22,baseline={([yshift=-.6ex]current bounding box.center)}]	\DrawLatticeIsoperimetricEnlargement{\LatticeIsoperimetricEnlargementE}
\end{tikzpicture}
}
\medskip \\
B_0=A & & B_1 & & B_2 & & B_3 & & B_4=\overline{A} \\
\Delta=8 & & \Delta=7 & & \Delta=7 & & \Delta=7 & & \Delta=7
\end{array}$
\caption{Enlarging a set $A\subseteq V$ into the encompassing rectangle.}
\label{fig:isoperimetric:lattice:lemma:extension}
\end{figure}
	
We next argue that $A$ is in fact connected in $L$. Indeed, suppose that $A$ is not connected.
Let $A_1,A_2,\ldots,A_k$ be the connected components of $A$. Since $A$ is convex and 
$A\cup N(A)$ is connected in the original lattice, we can re-order the sets $A_1,A_2,\ldots,A_k$ 
in such a way that the two sets $N(A_1\cup\cdots\cup A_{k-1})$ and $N(k)$ share exactly one 
element (Fig.~\ref{fig:isoperimetric:lattice:connectivity}). However, since $A$ is convex, we can 
shift $A_k$ to obtain a set $A'_k$ disjoint from $A_1\cup\cdots\cup A_{k-1}$ such that $N(A'_k)$ 
and $N(A_1\cup\cdots\cup A_{k-1})$ share at least two elements. It follows that $\Delta(A_1\cup
\cdots\cup A_{k-1}\cup A'_2)<\Delta(A)$, which is a contradiction.

\begin{figure}[htbp]
\centering
\begin{subfigure}[b]{0.4\textwidth}
\centering
{
\tikzsetfigurename{lattice_isoperimetric_not_connected}
\begin{tikzpicture}[baseline,scale=0.3,>=stealth,shorten >=1]
\DrawLatticeIsoperimetricNotConnected
\end{tikzpicture}
}
\caption{$A=A_1\cup A_2$, $\Delta=11$.}
\label{fig:isoperimetric:lattice:not-connected}
\end{subfigure}
\begin{subfigure}[b]{0.4\textwidth}
\centering
{
\tikzsetfigurename{lattice_isoperimetric_connected}
\begin{tikzpicture}[baseline,scale=0.3,>=stealth,shorten >=1]
\DrawLatticeIsoperimetricConnected
\end{tikzpicture}
}
\caption{$A_1\cup A'_2$, $\Delta=10$.}
\label{fig:isoperimetric:lattice:connected}
\end{subfigure}
\caption{Optimal sets are connected in $L$. The isoperimetric cost of a convex 
disconnected set can be decreased by shifting one of the components.}
\label{fig:isoperimetric:lattice:connectivity}
\end{figure}
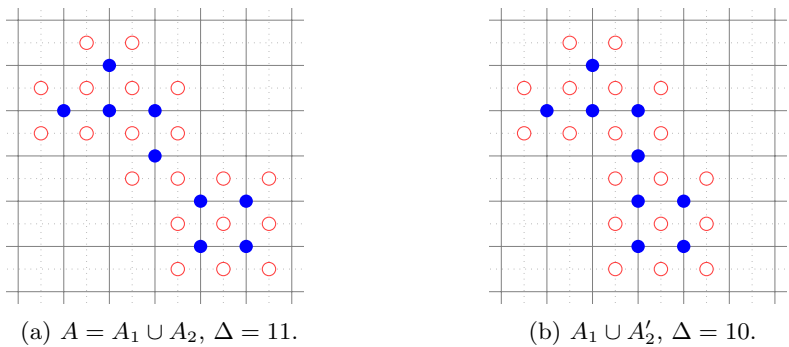

\begin{figure}[htbp]
\centering
{
\tikzsetfigurename{lattice_isoperimetric_labeling}
\begin{tikzpicture}[baseline,scale=0.4,>=stealth,shorten >=1]
\DrawLatticeIsoperimetricLabeling
\end{tikzpicture}
}
\caption{The labeling of the vertices of $c(\overline{A})$.}
\label{fig:isoperimetric:lattice:labeling}
\end{figure}
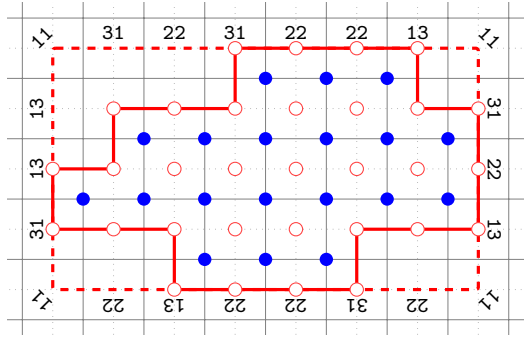
	
A convex and connected set in $L$ is easily seen to satisfy $N_{\symb{1010}}(A)=\varnothing$.
Let $L'$ be the graph with vertex set $U$ and with an edge between $(a,b)$ and $(a',b')$ if 
and only if $\abs{a'-a}=\abs{b'-b}=1$. This is the lattice dual to $L$.
Since $A$ is connected and convex,
the elements of $N_1(A)\cup N_2(A)\cup N_3(A)$ induce a simple cycle in $L'$,  which is 
the contour encompassing $A$. We denote this cycle by $c(A)$. Since $A$ is convex and 
connected, there is a natural one-to-one correspondence between the edges of $c(A)$ 
and the edges of the contour $c(\overline{A})$ encompassing the rectangle $\overline{A}$. 
Let us label the vertices of $c(\overline{A})$ with pairs in $\{1,2,3\}^2$ as follows (see 
Fig.~\ref{fig:isoperimetric:lattice:labeling}). Let $x$ be a vertex of $c(\overline{A})$,
and let $(y,x)$ and $(x,z)$ be two two edges incident to $x$. Let $(y',x')$ and $(x'',z'')$ be the 
edges of $c(A)$ corresponding to $(y,x)$ and $(x,z)$, respectively. If $x'\in N_i(A)$ and 
$x''\in N_j(A)$, then we label $x$ with $(i,j)$. Note that the only possible labels are $(1,1)$, 
$(2,2)$, $(3,1)$ and $(1,3)$, and that the four corners of $c(\overline{A})$ are precisely the 
vertices with label $(1,1)$. Counting reveals that $\abs{N_1(A)} - \abs{N_3(A)} = 4$.	
\end{proof}

\begin{proof}[Proof of Proposition~\ref{prop:isoperimetric:hypercube:recursion}]
	Let $S(d,m)\subseteq\{\symb{0},\symb{1}\}^d$ denote the set consisting
	of the $m$ first elements of Harper's isoperimetric ordering of the vertices of $H_d$.
	Observe that $S\big(d,\sum_{i=0}^{r-1}\binom {d}{i}\big)$ consists precisely of
	the words $w\in\{\symb{0},\symb{1}\}^d$ with $\norm{w}<r$.
	For $0\leq k\leq\binom{d}{r}$, we have
	\begin{align}
		S\left(d,\sum_{i=0}^{r-1}\binom {d}{i} + k\right) &=
			S\left(d,\sum_{i=0}^{r-1}\binom {d}{i}\right) \cup L(d,r,k)
	\end{align}
	where $L(d,r,k)$ consists of the first $k$ elements of
	the set $\{w\in\{\symb{0},\symb{1}\}^d: \norm{w}=r\}$
	according to the reverse lexicographic ordering.
	
	The sets $L(d,r,k)$ satisfy the recursion
	\begin{align}
		\label{eq:isoperimetric:hypercube:recursion:proof}
		L(d,r,k) &=
			\begin{cases}
				\symb{1}L(d-1,r-1,k)	& \text{if $0< k\leq\binom{d-1}{r-1}$,} \\[1ex]
				\symb{1}L\big(d-1,r-1,\binom{d-1}{r-1}\big) \cup \symb{0}L\big(d-1,r,k-\binom{d-1}{r-1}\big)
					& \text{if $\binom{d-1}{r-1}< k\leq\binom{d}{r}$,}
			\end{cases}
	\end{align}
	whenever $0< r\leq d$.
	
	Observe that for $0<r\leq d$, the vertex boundary of $S\big(d,\sum_{i=0}^{r-1}\binom {d}{i}\big)$
	is simply the set $L\big(d,r,\binom{d}{r}\big)=\{w\in\{\symb{0},\symb{1}\}^d: \norm{w}=r\}$
	which has cardinality $\binom{d}{r}$.
	Hence, $\Delta_{d+1}\big(\sum_{i=0}^{r-1}\binom {d}{i}\big)=\binom{d}{r}$.
	For $0\leq k\leq\binom{d}{r}$,
	the boundary of $S\big(d,\sum_{i=0}^{r-1}\binom {d}{i}\big)$
	can be divided into those elements $w$ with $\norm{w}=r$ and those with $\norm{w}=r+1$.
	The first part is simply the set
	$B_1(d,r,k)\isdef\{w\in\{\symb{0},\symb{1}\}^d: \norm{w}=r\}\setminus L(d,r,k)$
	and has cardinality $\binom{d}{r}-k$.
	The second part is $B_2(d,r,k)\isdef N(L(d,r,k))\cap\{w\in\{\symb{0},\symb{1}\}^d: \norm{w}=r+1\}$.
	The elements of $B_2(d,r,k)$ are the words obtained from the elements of $L(d,r,k)$
	by turning a~$\symb{0}$ into a~$\symb{1}$.
	Denoting the cardinality of $B_2(d,r,k)$ by $\psi_d(r,k)$,
	the recursion~\eqref{eq:isoperimetric:hypercube:recursion}
	follows easily from~\eqref{eq:isoperimetric:hypercube:recursion:proof}.
\end{proof}


\subsection{Calculation of the critical size}
\label{apx:critical-size}

\begin{proof}[Proof of Lemma~\ref{lem:torus:critical-size}]
	Using the explicit expressions~\eqref{eq:hard-core:isoperimetric:case-1}
	and~\eqref{eq:hard-core:isoperimetric:case-2} for $\Delta(s)$,
	for $s>1$ we have
	\begin{align}
		g(s)-g(s-1) &=
			\begin{cases}
				1-\alpha	& \text{if $s=\ell^2+1$ or $s=\ell(\ell+1)+1$ for some $\ell>0$,} \\
				-\alpha		& \text{otherwise.}			
			\end{cases}
	\end{align}
	Since $-\alpha<0<1-\alpha$, it follows that every maximiser of $g(s)$ must be of the form 
	$s=\ell^2+1$ or $s=\ell(\ell+1)+1$ for some $\ell>0$. Let $g_1(\ell)\isdef g(\ell^2+1)=2(\ell+1)
	-\alpha\ell^2$ and $g_2(\ell)\isdef g(\ell(\ell+1)+1)=2(\ell+1)+1-\alpha\ell(\ell+1)$. These are 
	quadratic functions.  Since $\nicefrac{2}{\alpha}\notin\ZZ$, the function $g_1$ has a unique 
	maximiser at $\ell_1\isdef[\nicefrac{1}{\alpha}]$, i.e., the closest integer to $\nicefrac{1}{\alpha}$.
	Similarly, since $\nicefrac{1}{\alpha}\notin\ZZ$, the function $g_2$ has a unique maximiser at 
	$\ell_2\isdef\lfloor\nicefrac{1}{\alpha}\rfloor$, which is the closest integer to $\nicefrac{1}{\alpha}
	-\nicefrac{1}{2}$. Note that either $\ell_1=\ell_2$ or $\ell_1=\ell_2+1$. In either case, it is 
	straightforward to verify that $g_1(\ell_1)<g_2(\ell_2)$. We find that $s^*\isdef \ell_2(\ell_2+1)+1$ 
	is the unique maximiser of $g(s)$. Finally, observe that $\ell^*=\lceil\nicefrac{1}{\alpha}\rceil
	=\lfloor\nicefrac{1}{\alpha}\rfloor + 1=\ell_2+1$.
\end{proof}

\begin{proof}[Proof of Lemma~\ref{lem:doubled-torus:critical-size}]
	Using the expression~\eqref{eq:widom-rowlinson:isoperimetric} for $\Delta(s)$,
	for $s>1$ we have
	\begin{align}
		g(s) - g(s-1) &=
			\begin{cases}
				1-\alpha	& \text{if $s=\ell^2+(\ell-1)^2+r$ with $r\in\{1,\ell,2\ell,3\ell\}$,} \\
				-\alpha		& \text{otherwise.}
			\end{cases}
	\end{align}
	since $-\alpha<0<1-\alpha$, it follows that every maximiser of $g(s)$ must be of the form
	$s=\ell^2+(\ell-1)^2+r$ for some $\ell>0$ and $r\in\{1,\ell,2\ell,3\ell\}$.
	Let us thus consider the functions
	\begin{align}
		g_1(\ell) &\isdef
			g(\ell^2+(\ell-1)^2+1) = 4\ell+1-\alpha(\ell^2+(\ell-1)^2) \;, \\
		g_{1+k}(\ell) &\isdef
			g(\ell^2+(\ell-1)^2+k\ell) = 4\ell+1+k-\alpha(\ell^2+(\ell-1)^2+k\ell-1) \;,
	\end{align}
	for $k=1,2,3$.
	The maximum of a concave quadratic function over integers is achieved at the closest 
	integer to its critical point. Since $\nicefrac{4}{\alpha}\notin\ZZ$, the maximisers of
	$g_1$, $g_2$, $g_3$ and $g_4$ are unique:
	the maximums are respectively achieved at
	\begin{align}
		\ell^*_1 &\isdef \left[\frac{1}{\alpha}+\frac{1}{2}\right] \;, &
			\ell^*_2 &\isdef \left[\frac{1}{\alpha}+\frac{1}{4}\right] \;, &
			\ell^*_3 &\isdef \left[\frac{1}{\alpha}\right] \;, &
			\ell^*_4 &\isdef \left[\frac{1}{\alpha}-\frac{1}{4}\right] \;,
	\end{align}
	where $[a]$ denotes the closest integer to $a$. Let $\{\nicefrac{1}{\alpha}\}$ denote 
	the fractional part of $\nicefrac{1}{\alpha}$. We consider four cases:
	\begin{case}[1]{$0<\{\nicefrac{1}{\alpha}\}<\nicefrac{1}{4}$.}
		In this case, $\ell^*_2=\ell^*_3=\ell^*_4=\lfloor\nicefrac{1}{\alpha}\rfloor
			< \lceil\nicefrac{1}{\alpha}\rceil = \ell^*_1$. \\
		Observe that $\ell^*=\lfloor\nicefrac{1}{\alpha}\rfloor$.
		We have
		\begin{align}
			g_1(\ell^*_1) &= g\big((\ell^*+1)^2+(\ell^*)^2+1\big) 
				= 4(\ell^*+1)+1-\alpha\big((\ell^*+1)^2+(\ell^*)^2\big) \;, \\
			g_2(\ell^*_2) &= g\big((\ell^*)^2+(\ell^*-1)^2+\ell^*\big)
				= 4\ell^*+2-\alpha\big((\ell^*)^2+(\ell^*-1)^2+\ell^*-1\big) \;, \\
			g_3(\ell^*_3) &= g\big((\ell^*)^2+(\ell^*-1)^2+2\ell^*\big)
				= 4\ell^*+3-\alpha\big((\ell^*)^2+(\ell^*-1)^2+2\ell^*-1\big) \;, \\
			g_4(\ell^*_4) &= g\big((\ell^*)^2+(\ell^*-1)^2+3\ell^*\big)
				= 4\ell^*+4-\alpha\big((\ell^*)^2+(\ell^*-1)^2+3\ell^*-1\big) \;.
		\end{align}
		A straightforward calculation shows that
		\begin{align}
			g_1(\ell^*_1)<g_2(\ell^*_2)<g_3(\ell^*_3)<g_4(\ell^*_4) \;,
		\end{align}
		where for the first inequality, we have used $3-\alpha(3\ell^*+1)<0$,
		and for the others, we have used $1-\alpha\ell^*>0$.
		Hence, in this case $g(s)$ has a unique maximiser at $s^*=(\ell^*)^2+(\ell^*-1)^2+3\ell^*$.
	\end{case}
	\begin{case}[2]{$\nicefrac{1}{4}<\{\nicefrac{1}{\alpha}\}<\nicefrac{1}{2}$.}
		In this case, $\ell^*_3=\ell^*_4=\lfloor\nicefrac{1}{\alpha}\rfloor
			< \lceil\nicefrac{1}{\alpha}\rceil = \ell^*_1=\ell^*_2$. \\
		Observe that again $\ell^*=\lfloor\nicefrac{1}{\alpha}\rfloor$.
		We have
		\begin{align}
			g_1(\ell^*_1) &= g\big((\ell^*+1)^2+(\ell^*)^2+1\big) 
				= 4(\ell^*+1)+1-\alpha\big((\ell^*+1)^2+(\ell^*)^2\big) \;, \\
			g_2(\ell^*_2) &= g\big((\ell^*+1)^2+(\ell^*)^2+\ell^*+1\big)
				= 4(\ell^*+1)+2-\alpha\big((\ell^*+1)^2+(\ell^*)^2+\ell^*\big) \;, \\
			g_3(\ell^*_3) &= g\big((\ell^*)^2+(\ell^*-1)^2+2\ell^*\big)
				= 4\ell^*+3-\alpha\big((\ell^*)^2+(\ell^*-1)^2+2\ell^*-1\big) \;, \\
			g_4(\ell^*_4) &= g\big((\ell^*)^2+(\ell^*-1)^2+3\ell^*\big)
				= 4\ell^*+4-\alpha\big((\ell^*)^2+(\ell^*-1)^2+3\ell^*-1\big) \;.
		\end{align}
		In this case, we have
		\begin{align}
			g_1(\ell^*_1)&<g_2(\ell^*_2) \;,
				& g_3(\ell^*_3)&<g_4(\ell^*_4) \;, 
				& g_1(\ell^*_1)&<g_3(\ell^*_3) \;,
				& g_2(\ell^*_2)&<g_4(\ell^*_4) \;,
		\end{align}
		where the first two inequalities follow from $1-\alpha\ell^*>0$
		and the last two inequalities from $2-\alpha(2\ell^*+1)<0$.
		Hence, $g(s)$ again has a unique maximiser at $s^*=(\ell^*)^2+(\ell^*-1)^2+3\ell^*$.
	\end{case}
	\begin{case}[3]{$\nicefrac{1}{2}<\{\nicefrac{1}{\alpha}\}<\nicefrac{3}{4}$.}
		In this case, $\ell^*_4=\lfloor\nicefrac{1}{\alpha}\rfloor
			< \lceil\nicefrac{1}{\alpha}\rceil = \ell^*_1=\ell^*_2=\ell^*_3$. \\
		In this case, $\ell^*=\lceil\nicefrac{1}{\alpha}\rceil$.
		Therefore, we have
		\begin{align}
			g_1(\ell^*_1) &= g\big((\ell^*)^2+(\ell^*-1)^2+1\big) 
				= 4\ell^*+1-\alpha\big((\ell^*)^2+(\ell^*-1)^2\big) \;, \\
			g_2(\ell^*_2) &= g\big((\ell^*)^2+(\ell^*-1)^2+\ell^*\big)
				= 4\ell^*+2-\alpha\big((\ell^*)^2+(\ell^*-1)^2+\ell^*-1\big) \;, \\
			g_3(\ell^*_3) &= g\big((\ell^*)^2+(\ell^*-1)^2+2\ell^*\big)
				= 4\ell^*+3-\alpha\big((\ell^*)^2+(\ell^*-1)^2+2\ell^*-1\big) \;, \\
			g_4(\ell^*_4) &= g\big((\ell^*-1)^2+(\ell^*-2)^2+3(\ell^*-1)\big)
				= 4\ell^*-\alpha\big((\ell^*-1)^2+(\ell^*-2)^2+3\ell^*-4\big) \;.
		\end{align}
		With straightforward calculation we obtain
		\begin{align}
			g_1(\ell^*_1)&<g_2(\ell^*_2) \;,
				& g_3(\ell^*_3)&<g_2(\ell^*_2) \;, 
				& g_4(\ell^*_4)&<g_2(\ell^*_2) \;,
		\end{align}
		where the first inequality follows from $1-\alpha(\ell^*-1)>0$,
		the second from $1-\alpha\ell^*<0$, and the third from $2-\alpha(2\ell^*-1)>0$.
		Hence, in this case,
		the unique maximiser of $g(s)$ is $s^*=(\ell^*)^2+(\ell^*-1)^2+\ell^*$.
	\end{case}
	\begin{case}[4]{$\nicefrac{3}{4}<\{\nicefrac{1}{\alpha}\}<1$.}
		In this case, $\lfloor\nicefrac{1}{\alpha}\rfloor
			< \lceil\nicefrac{1}{\alpha}\rceil = \ell^*_1=\ell^*_2=\ell^*_3=\ell^*_4$. \\
		In this case, we again have $\ell^*=\lceil\nicefrac{1}{\alpha}\rceil$.
		Therefore, 
		\begin{align}
			g_1(\ell^*_1) &= g\big((\ell^*)^2+(\ell^*-1)^2+1\big) 
				= 4\ell^*+1-\alpha\big((\ell^*)^2+(\ell^*-1)^2\big) \;, \\
			g_2(\ell^*_2) &= g\big((\ell^*)^2+(\ell^*-1)^2+\ell^*\big)
				= 4\ell^*+2-\alpha\big((\ell^*)^2+(\ell^*-1)^2+\ell^*-1\big) \;, \\
			g_3(\ell^*_3) &= g\big((\ell^*)^2+(\ell^*-1)^2+2\ell^*\big)
				= 4\ell^*+3-\alpha\big((\ell^*)^2+(\ell^*-1)^2+2\ell^*-1\big) \;, \\
			g_4(\ell^*_4) &= g\big((\ell^*)^2+(\ell^*-1)^2+3\ell^*\big)
				= 4\ell^*+4-\alpha\big((\ell^*)^2+(\ell^*-1)^2+3\ell^*-1\big) \;.
		\end{align}
		Similar calculation leads to
		\begin{align}
			g_1(\ell^*_1)&<g_2(\ell^*_2) \;,
				& g_4(\ell^*_4)<g_3(\ell^*_3)&<g_2(\ell^*_2) \;, 
		\end{align}
		where the first inequality follows from $1-\alpha(\ell^*-1)>0$,
		the other two from $1-\alpha\ell^*<0$.
		Hence, the unique maximiser of $g(s)$ in this case
		is $s^*=(\ell^*)^2+(\ell^*-1)^2+\ell^*$.
		\qedhere
	\end{case}
\end{proof}


\bibliographystyle{plain}
\bibliography{files/bibliography}


\end{document}

%% file: files/macros.tex
\usepackage{amsmath}
\usepackage{amssymb}
\usepackage{amsfonts}
\usepackage{amssymb}
\usepackage{wasysym}
\usepackage{mathrsfs}
\usepackage{amsthm}

\usepackage{mathtools}
\usepackage{enumitem}
\usepackage{mdwlist}		
\usepackage[normalem]{ulem}		
\usepackage{braket}		
\usepackage{nicefrac}		
\usepackage{comment}		
\usepackage{caption}
\usepackage{subcaption}		
\usepackage{multirow}		

\usepackage[pdftex=true,hyperindex=true,pdfborder={0 0 0}]{hyperref}
\usepackage[a4paper, left=2.7cm, right=2.7cm, top=3.5cm, bottom=2cm]{geometry}
\usepackage{graphicx}
\usepackage{tikz}
\usepgflibrary{arrows}

\numberwithin{equation}{section}	

\theoremstyle{plain}
\newtheorem{theorem}{Theorem}[section]
\newtheorem{lemma}[theorem]{Lemma}
\newtheorem{corollary}[theorem]{Corollary}
\newtheorem{proposition}[theorem]{Proposition}
\newtheorem{observation}[theorem]{Observation}
\newtheorem{question}[theorem]{Question}
\theoremstyle{definition}

\newtheorem{example}[theorem]{Example}
\newtheorem{conjecture}[theorem]{Conjecture}

\newenvironment{claim}[1][\unskip]{%

	\smallskip
	
	\noindent\textsl{Claim #1}:  %
}{%

}

\makeatletter
\newenvironment{argumentenv}{%
	\list{}{%
		\leftmargin 3em%
		\listparindent 1em%
		\itemindent \listparindent
		\parsep        \z@ \@plus\p@}%
		\item\relax%
}{\endlist}
\makeatother
\newenvironment{argument}{%
	\begin{argumentenv}
	\noindent\small\textsl{Argument.}%
}{%
	\end{argumentenv}
}

\newenvironment{case}[2][\unskip]{%
	
	\smallskip
	
	\noindent\textsf{Case #1}: {#2} %
}{%

}
\newcommand{\exampleqed}{\ensuremath{\ocircle}\par}

\newcommand{\ZZ}{\mathbb{Z}}			
\newcommand{\NN}{\mathbb{N}}			
\newcommand{\RR}{\mathbb{R}}			
\newcommand{\symb}[1]{\mathtt{#1}}		
\newcommand{\isdef}{\triangleq}			
\newcommand{\abs}[1]{
	\left\lvert#1\right\rvert%
}
\newcommand{\norm}[1]{
	\left\lVert#1\right\rVert%
}
\renewcommand{\complement}{
	\mathsf{c}%
}
\newcommand{\dd}{\mathrm{d}}			
\newcommand{\ee}{\mathrm{e}}			
\newcommand{\ii}{\mathrm{i}}			
\newcommand{\smallo}{o}					
\newcommand{\bigo}{O}					
\newcommand{\xPr}{\operatorname{\mathbb{P}}}		
\newcommand{\xExp}{\operatorname{\mathbb{E}}}		

\newcommand{\spX}{\mathscr{X}}								
\newcommand{\spT}{\mathscr{T}}								
\newcommand{\family}[1]{\mathfrak{#1}}						
\newcommand{\effR}[3][\unskip]{\mathcal{R}^{#1}(#2\leftrightarrow#3)}		
\newcommand{\effC}[3][\unskip]{\mathcal{C}^{#1}(#2\leftrightarrow#3)}		
\newcommand{\divergence}{\operatorname{\mathrm{div}}}		
\newcommand{\link}{\sim}									
\newcommand{\pathto}[1][\unskip]{\overset{#1}{\leadsto}}	
\newcommand{\xadd}[1][\unskip]{
	\xrightarrow{\symb{+ #1\ifx#1\empty\else\;\fi}}%
}
\newcommand{\xremove}[1][\unskip]{
	\xrightarrow{\symb{- #1\ifx#1\empty\else\;\fi}}%
}

\newcommand{\birth}{\mathsf{b}}								
\newcommand{\death}{\mathsf{d}}								
\newcommand{\blue}{\mathsf{b}}								
\newcommand{\red}{\mathsf{r}}								



%% file: files/diagrams-ready.tex
\usepackage{graphicx}
\usepackage{tikzexternal}
\tikzsetexternalprefix{figures/}	

%% file: hard-core-paper.bbl
\newcommand{\noopsort}[1]{}
\begin{thebibliography}{10}

\bibitem{AldFil14}
D.~Aldous and J.~A. Fill.
\newblock {\em Reversible {M}arkov Chains and Random Walks on Graphs}.
\newblock Unfinished monograph, 2002 (recompiled 2014).
\newblock {\tt [\url{http://www.stat.berkeley.edu/~aldous/RWG/book.html}]}.

\bibitem{AloCer96}
L.~Alonso and R.~Cerf.
\newblock The three dimensional polyominoes of minimal area.
\newblock {\em The Electronic Journal of Combinatorics}, 3(1), 1996.

\bibitem{AroCer96}
G.~Ben Arous and R.~Cerf.
\newblock Metastability of the three-dimensional {I}sing model on a torus at
  very low temperatures.
\newblock {\em Electronic Journal of Probability}, 1(10), 1996.

\bibitem{BelLan10}
J.~Beltr\'an and C.~Landim.
\newblock Tunneling and metastability of continuous time {M}arkov chains.
\newblock {\em Journal of Statistical Physics}, 140(6):1065--1114, 2010.

\bibitem{BerSte94}
J.~{\noopsort{Berg}}{van~den~Berg} and J.~E. Steif.
\newblock Percolation and the hard-core lattice gas model.
\newblock {\em Stochastic Processes and their Applications}, 49:179--197, 1994.

\bibitem{BerKon90}
K.~A. Berman and M.~H. Konsowa.
\newblock Random paths and cuts, electrical networks, and reversible {M}arkov
  chains.
\newblock {\em SIAM Journal of Discrete Mathematics}, 3(3):311--319, 1990.

\bibitem{Bez89}
S.~L. Bezrukov.
\newblock On the construction of solutions of a discrete isoperimetric problem
  in {H}amming space.
\newblock {\em Mathematics of the USSR-Sbornik}, 63(1), 1989.

\bibitem{Bez94}
S.~L. Bezrukov.
\newblock Isoperimetric problems in discrete spaces.
\newblock In {\em Extremal Problems in Finite Sets}, volume~3 of {\em Bolyai
  Society Mathematical Studies}, pages 59--91, 1994.

\bibitem{BovEckGayKle00}
A.~Bovier, M.~Eckhoff, V.~Gayrard, and M.~Klein.
\newblock Metastability and small eigenvalues in {M}arkov chains.
\newblock {\em Journal of Physics A: Mathematical and General},
  33(46):L447--L451, 2000.

\bibitem{BovEckGayKle01}
A.~Bovier, M.~Eckhoff, V.~Gayrard, and M.~Klein.
\newblock Metastability in stochastic dynamics of disordered mean-field models.
\newblock {\em Probability Theory and Related Fields}, 119(1):99--161, 2001.

\bibitem{BovEckGayKle02}
A.~Bovier, M.~Eckhoff, V.~Gayrard, and M.~Klein.
\newblock Metastability and low lying spectra in reversible {M}arkov chains.
\newblock {\em Communications in Mathematical Physics}, 228(2):219--255, 2002.

\bibitem{BovHol15}
A.~Bovier and F.~{\noopsort{Hollander}}{den~Hollander}.
\newblock {\em Metastability: A Potential-Theoretic Approach}.
\newblock Springer, 2015.

\bibitem{BovHolNar06}
A.~Bovier, F.~{\noopsort{Hollander}}{den~Hollander}, and F.~R. Nardi.
\newblock Sharp asymptotics for {K}awasaki dynamics on a finite box with open
  boundary.
\newblock {\em Probability Theory and Related Fields}, 135(2):265--310, 2006.

\bibitem{BovMan02}
A.~Bovier and F.~Manzo.
\newblock Metastability in {G}lauber dynamics in the low temperature limit:
  beyond exponential asymptotics.
\newblock {\em Journal of Statistical Physics}, 107(3--4):757--779, 2002.

\bibitem{CasGalOliVar84}
M.~Cassandro, A.~Galves, E.~Olivieri, and M.~E. Vares.
\newblock Metastable behavior of stochastic dynamics: A pathwise approach.
\newblock {\em Journal of Statistical Physics}, 35(5--6):603--634, 1984.

\bibitem{CirNar03}
E.~N.~M. Cirillo and F.~R. Nardi.
\newblock Metastability for a stochastic dynamics with a parallel heat bath
  updating rule.
\newblock {\em Journal of Statistical Physics}, 110(1--2):183--–217, 2003.

\bibitem{CirNar13}
E.~N.~M. Cirillo and F.~R. Nardi.
\newblock Relaxation height in energy landscapes: An application to multiple
  metastable states.
\newblock {\em Journal of Statistical Physics}, 150:1080--1114, 2013.

\bibitem{CirNarSoh15}
E.~N.~M. Cirillo, F.~R. Nardi, and J.~Sohier.
\newblock Metastability for general dynamics with rare transitions: Escape time
  and critical configurations.
\newblock {\em Journal of Statistical Physics}, 161:365--403, 2015.

\bibitem{CirNarSpi08}
E.~N.~M. Cirillo, F.~R. Nardi, and C.~Spitoni.
\newblock Metastability for reversible probabilistic cellular automata with
  self-interaction.
\newblock {\em Journal of Statistical Physics}, 132(3):431--471, 2008.

\bibitem{Dpr}
S.~Dommers.
\newblock Metastability of the {I}sing model on random regular graphs at zero
  temperature.
\newblock {\em Probability Theory and Related Fields}, 167(1):305--324, 2017.

\bibitem{DdHJNpr}
S.~Dommers, F.~{\noopsort{Hollander}}den Hollander, O.~Jovanovski, and F.~R.
  Nardi.
\newblock Metastability for {G}lauber dynamics on random graphs.
\newblock {\em Annals of Applied Probability}, To appear, 2016.

\bibitem{DoySne84}
P.~G. Doyle and J.~L. Snell.
\newblock {\em Random Walks and Electric Networks}.
\newblock The Mathematical Association of America, 1984.

\bibitem{FerManNarSco15}
R.~Fernandez, F.~Manzo, F.~R. Nardi, and E.~Scoppola.
\newblock Asymptotically exponential hitting times and metastability: a
  pathwise approach without reversibility.
\newblock {\em Eletronic Journal of Probability}, 20(122), 2015.

\bibitem{FerManNarScoSoh16}
R.~Fernandez, F.~Manzo, F.~R. Nardi, E.~Scoppola, and J.~Sohier.
\newblock Conditioned, quasi-stationary, restricted measures and escape from
  metastable states.
\newblock {\em The Annals of Applied Probability}, 26(2):760--793, 2016.

\bibitem{GauOliSco05}
A.~Gaudilli\`ere, E.~Olivieri, and E.~Scoppola.
\newblock Nucleation pattern at low temperature for local {K}awasaki dynamics
  in two dimensions.
\newblock {\em Markov Processes and Related Fields}, 11(4):553--628, 2005.

\bibitem{Gri10}
G.~Grimmett.
\newblock {\em Probability on Graphs}.
\newblock Cambridge University Press, 2010.

\bibitem{HarHar76}
F.~Harary and F.~Harborth.
\newblock Extremal animals.
\newblock {\em Journal of Combinatorics, Information and System}, 1:1--8, 1976.

\bibitem{Har66}
L.~H. Harper.
\newblock Optimal numberings and isoperimetric problems on graphs.
\newblock {\em Journal of Combinatorial Theory}, 1:385--393, 1966.

\bibitem{Har04}
L.~H. Harper.
\newblock {\em Global Methods for Combinatorial Isoperimetric Problems}.
\newblock Cambridge University Press, 2004.

\bibitem{dHJpr}
F.~{\noopsort{Hollander}}den Hollander and O.~Jovanovski.
\newblock Metastability on the hierarchical lattice.
\newblock {\em Journal of Physics A: Theoretical and Mathematical}, 50:305001,
  2017.

\bibitem{HolNarOliSco03}
F.~{\noopsort{Hollander}}den Hollander, F.~R. Nardi, E.~Olivieri, and
  E.~Scoppola.
\newblock Droplet growth for three-dimensional {K}awasaki dynamics.
\newblock {\em Probability Theory and Related Fields}, 125(2):153–--194,
  2003.

\bibitem{HolNarTro11}
F.~{\noopsort{Hollander}}den Hollander, F.~R. Nardi, and A.~Troiani.
\newblock {K}awasaki dynamics with two types of particles: Stable/metastable
  configurations and communication heights.
\newblock {\em Journal of Statistical Physics}, 145:1423--1457, 2011.

\bibitem{HolOliSco00}
F.~{\noopsort{Hollander}}den Hollander, E.~Olivieri, and E.~Scoppola.
\newblock Metastability and nucleation for conservative dynamics.
\newblock {\em Journal of Mathematical Physics}, 41(3):1424--1498, 2000.

\bibitem{Jpr}
O.~Jovanovski.
\newblock Metastability for the {I}sing model on the hypercube.
\newblock {\em Journal of Statistical Physics}, 167(1):135--159, 2017.

\bibitem{Kei79}
J.~Keilson.
\newblock {\em Markov Chain Models --- Rarity and Exponentiality}.
\newblock Springer-Verlag, 1979.

\bibitem{KotOli94}
R.~Koteck\'y and E.~Olivieri.
\newblock Shapes of growing droplets---a model of escape from a metastable
  phase.
\newblock {\em Journal of Statistical Physics}, 75(3--4):409--506, 1994.

\bibitem{LebGal71}
J.~L. Lebowitz and G.~Gallavotti.
\newblock Phase transitions in binary lattice gases.
\newblock {\em Journal of Mathematical Physics}, 12(7):1129--1133, 1971.

\bibitem{LevPerWil08}
D.~A. Levin, Y.~Peres, and E.~L. Wilmer.
\newblock {\em Markov Chains and Mixing Times}.
\newblock American Mathematical Society, 2008.

\bibitem{LyoPer14}
R.~Lyons and Y.~Peres.
\newblock {\em Probability on Trees and Networks}.
\newblock Cambridge University Press, 2016.

\bibitem{ManNarOliSco04}
F.~Manzo, F.~R. Nardi, E.~Olivieri, and E.~Scoppola.
\newblock On the essential features of metastability: Tunnelling time and
  critical configurations.
\newblock {\em Journal of Statistical Physics}, 115(1/2):591--642, 2004.

\bibitem{NarZocBor15}
F.~R. Nardi, A.~Zocca, and S.~C. Borst.
\newblock Hitting times asymptotics for hard-core interactions on grids.
\newblock {\em Journal of Statistical Physics}, 162(2):522--576, 2015.

\bibitem{NevSch91}
E.~J. Neves and R.~H. Schonmann.
\newblock Critical droplets and metastability for a {G}lauber dynamics at very
  low temperatures.
\newblock {\em Communications in Mathematical Physics}, 137(2):209--230, 1991.

\bibitem{OliVar04}
E.~Olivieri and M.~E. Vares.
\newblock {\em Large Deviations and Metastability}.
\newblock Cambridge University Press, 2004.

\bibitem{VaiBru08}
D.~Vainsencher and A.~M. Bruckstein.
\newblock On isoperimetrically optimal polyforms.
\newblock {\em Theoretical Computer Science}, 406:146--159, 2008.

\bibitem{WanWan77}
D.{-L.} Wang and P.~Wang.
\newblock Discrete isoperimetric problems.
\newblock {\em {SIAM} Journal on Applied Mathematics}, 32(4):860--870, 1977.

\bibitem{Zoc17b}
A.~Zocca.
\newblock Low-temperature behavior of the multicomponent {W}idom--{R}owlison
  model on finite square lattices.
\newblock {\em Preprint}, 2017.
\newblock \href{https://arxiv.org/abs/1701.09185}{\tt [arXiv:1701.09185]}.

\bibitem{Zoc17a}
A.~Zocca.
\newblock Tunneling of the hard-core model on finite triangular lattices.
\newblock {\em Preprint}, 2017.
\newblock \href{https://arxiv.org/abs/1701.07004}{\tt [arXiv:1701.07004]}.

\bibitem{ZocBorLeeNar13}
A.~Zocca, S.~C. Borst, J.~S.~H.~{\noopsort{Leeuwaarden}}van Leeuwaarden, and
  F.~R. Nardi.
\newblock Delay performance in random-access grid networks.
\newblock {\em Performance Evaluation}, 70(10):900--915, 2013.

\end{thebibliography}
